\documentclass[a4paper,12pt]{article}
\usepackage{amssymb, amsmath, amsthm}
\usepackage{a4wide}
\usepackage{array}
\usepackage{dsfont}             
\usepackage{mathrsfs}           
\usepackage{accents}            
\usepackage{pstricks}
\usepackage{pst-plot}
\usepackage{pst-node}
\usepackage{stmaryrd}           
\usepackage[all]{xy}            
\usepackage{ulem}               
\usepackage[lowtilde]{url}               

\newtheorem{thm}{Theorem}[section]
\newtheorem{cor}[thm]{Corollary}
\newtheorem{lem}[thm]{Lemma}

\theoremstyle{definition}
\newtheorem{ex}[thm]{Example}
\newtheorem{rem}[thm]{Remark}

\newtheorem{definition}[thm]{Definition}

\newtheorem{convention}[thm]{Convention}
\newtheorem{assumption}[thm]{Assumption}
\newtheorem{proposition}[thm]{Proposition}

\newcommand{\C}{\ensuremath{\mathds C}}                     

\newcommand{\R}{\ensuremath{\mathds R}}						

\newcommand{\N}{\ensuremath{\mathds N}}		            
\newcommand{\Nn}{\ensuremath{\mathds N_0}}				    
\newcommand{\wP}{\ensuremath{\mathds P}}				    
\newcommand{\E}{\ensuremath{\mathds E}}	                    
\newcommand{\one}{\ensuremath{\mathds 1}}                   

\newcommand{\cA}{\ensuremath{\mathcal A}}
\newcommand{\cB}{\ensuremath{\mathcal B}}

\newcommand{\cF}{\ensuremath{\mathcal F}}
\newcommand{\cG}{\ensuremath{\mathcal G}}
\newcommand{\cH}{\ensuremath{\mathcal H}}
		
\newcommand{\cL}{\ensuremath{\mathcal L}}
\newcommand{\cM}{\ensuremath{\mathcal M}}
\newcommand{\cP}{\ensuremath{\mathcal P}}
\newcommand{\cR}{\ensuremath{\mathcal R}}
\newcommand{\cS}{\ensuremath{\mathcal S}}

\newcommand{\nnrm}[2]{\ensuremath{\| #1 \|_{#2}}}

\newcommand{\gnnrm}[2]{\ensuremath{\big\| #1 \big\|_{#2}}}

\newcommand{\sgnnrm}[2]{\ensuremath{\Big\| #1 \Big\|_{#2}}}

\renewcommand{\H}{\ensuremath{\mathcal{H}}}

\renewcommand{\epsilon}{\ensuremath{\varepsilon}}

\renewcommand{\geq}{\ensuremath{\geqslant}}
\renewcommand{\leq}{\ensuremath{\leqslant}}

\newcommand{\dl}{\ensuremath{\mathrm{d}}}

\newcommand{\dom}{\ensuremath{\mathcal{O}}}

\newcommand{\supp}{\ensuremath{\mathop{\operatorname{supp}}}}

\newcommand{\bo}{\ensuremath{\mathscr L}}

\newcommand{\nuc}{\ensuremath{\mathscr L_1}}

\newcommand{\hs}{\ensuremath{\mathscr L_2}}

\newcommand{\tr}{\ensuremath{\mathop{\operatorname{Tr}}}}

\newcommand{\id}[1]{\ensuremath{\operatorname{id}}_{#1}}

\newcommand{\Lap}{\ensuremath{\Delta_{\dom}^\text{D}}}

\renewcommand{\Re}{\ensuremath{\operatorname{Re}}}

\definecolor{mylightGray}{gray}{0.95}

\allowdisplaybreaks
\swapnumbers
\numberwithin{equation}{section}

\title{
Singular Behavior of the Solution to the Stochastic Heat Equation on a Polygonal Domain
\thanks{May 5, 2013. This work has been supported by the Deutsche Forschungsgemeinschaft (DFG Priority Program 1324, grants SCHI 419/5-1, SCHI 419/5-2).}
}
\author{
Felix~Lindner
}

\date{}

\begin{document}
\maketitle

\begin{abstract}
We study the stochastic heat equation with trace class noise and zero Dirichlet boundary condition on a bounded polygonal domain $\dom\subset\R^2$. It is shown that the solution $u$ can be decomposed into a regular part $u_{\mathrm R}$ and a singular part $u_{\mathrm S}$ which incorporates the corner singularity functions for the Poisson problem. 
Due to the temporal irregularity of the noise, both $u_{\mathrm R}$ and $u_{\mathrm S}$ have negative $L_2$-Sobolev regularity of order $s<-1/2$ in time.
The regular part $u_{\mathrm R}$ admits spatial Sobolev regularity of order $r=2$, while the spatial Sobolev regularity of $u_{\mathrm S}$ is restricted by $r<1+\pi/\gamma$, where $\gamma$ is the largest interior angle at the boundary $\partial\dom$. We obtain estimates for the Sobolev norm of $u_{\mathrm R}$ and the Sobolev norms of the coefficients of the singularity functions. The proof is based on a Laplace transform argument w.r.t.\ the time variable.
The result is of interest in the context of numerical methods for stochastic PDEs.
\end{abstract}
\\[.5em]

\section{Introduction}
\label{Introduction}
Let $\dom\subset\R^2$ be a (possibly non-convex) bounded polygonal domain, $T\in(0,\infty)$ and let $\Lap:D(\Lap)\subset L_2(\dom)\to L_2(\dom)$ be the Laplace operator on $\dom$ with zero Dirichlet boundary condition. In this paper, we analyse the regularity and the singular behavior of the $L_2(\dom;\R)$-valued mild solution $u=(u(t))_{t\in[0,T]}$ to the stochastic heat equation
\begin{equation}\label{SHE}
\begin{aligned}
\dl u(t)&=\big[\Lap u(t)+ F(u(t))\big]\dl t + G(u(t))\dl W(t),\quad t\in[0,T],\\
u(0)&=u_0,
\end{aligned}
\end{equation}
where $W=(W(t))_{t\in[0,T]}$ is a Wiener process in some Hilbert space $U$, and the nonlinear operators $F$ and $G$ are assumed to satisfy appropriate global Lipschitz conditions. 

We are in particular interested in an explicit bound for the $L_2(\dom)$-Sobolev regularity of $u$, which is closely connected to the order of convergence that can be achieved by {\it uniform} numerical approximation methods if the error is measured in $L_2(\dom)$. 
In this respect, our result complements the Besov regularity results in \cite{CioDahKin11}, \cite{CioKimLee12}, which are related to the order of convergence for {\it non-uniform}, {\it adaptive} approximation methods. We refer to \cite{DeV98} or \cite[Section~1]{CioDahKin11} for details on the connection between regularity and approximation.

It is well known that the solutions to boundary value problems on non-smooth domains may have singularities at the boundary. In the {\it deterministic} setting, the singular behavior of the solutions has been analysed by many authors, e.g., by Borsuk and Kondratiev~\cite{BorKon06}, Dauge~\cite{Dau88}, Grisvard~\cite{Gri85}, \cite{Gri87}, \cite{Gri92}, \cite{Gri95}, Jerison and Kenig~\cite{JerKen81}, \cite{JerKen95}, Kozlov, Maz'ya and Ro{\ss}mann \cite{KozMazRos97}, \cite{KozMazRos01}, \cite{MazRos10} and Kweon \cite{Kwe12}, just to mention a few. For piecewise smooth domains, the singularities can be described more or less explicitly. In the {\it stochastic} parabolic case, singularities at the boundary occur naturally even on smooth domains, cf.~\cite{Fla90}, \cite[Section~1]{Kry94}. The reason is the low regularity of the noise term in time, which is in general incompatible with the boundary data unless the noise vanishes near the boundary. Thus, there are (at least) two possible sources for singularites of the solution $u=(u(t))_{t\in[0,T]}$ to Eq.~\eqref{SHE}: The corners of the boundary $\partial \dom$ (as in the deterministic case) and the irregularity of the driving Wiener process $W=(W(t))_{t\in[0,T]}$.

The Dirichlet boundary value problem for stochastic parabolic equations has been studied with the help of weighted Sobolev spaces $H^r_{p,\theta}(\dom)$, $r\geq0$, by N.V.~Krylov and collaborators; see, e.g., Krylov \cite{Kry94}, Krylov and Lototsky~\cite{KryLot99}, \cite{KryLot99b} and Kim~\cite{Kim04}, \cite{Kim11}. These spaces are such that the higher order derivatives of functions belonging to them are allowed to explode at the boundary. They have been used in the first place to handle the influence of the noise at the boundary for equations on smooth domains (\cite{Kim04}, \cite{Kry94}, \cite{KryLot99}, \cite{KryLot99b}), but they are also well-suited to treat stochastic equations on non-smooth domains (\cite{Kim11}). However, the regularity results in terms of weighted Sobolev spaces do not imply explicit bounds for the regularity in scales of Sobolev spaces without weights. Moreover, the results in \cite{Kim11} are the outcome of a worst-case analysis: the only assumption on the domain $\dom$ is that the Hardy inequality holds, but no specific geometric features of simple domains (such as polygonal domains) are exploited to optimize the results for such domains. As a consequence, there is no explicit description of the singularities of the solution that are due to the shape of the domain.
The situation is in a certain sense similar when considering regularity results that have been obtained in the framework of other approaches to stochastic PDEs, such as the semigroup approach; see, e.g., Da Prato and Zabzcyk \cite{DaPraZab}, Jentzen and R\"ockner \cite{JenRoeck12}, Kruse and Larsson \cite{KruLar12}, van Neerven, Veraar and Weis \cite{NeeVerWei08}, \cite{NeeVerWei12}. There, the spatial regularity of the solution is typically measured in terms of the domains of fractional powers of the governing linear operator; in our case, in tems of the spaces $D\big((-\Lap)^{r/2}\big)$, $r\geq0$. However, for non-smooth domains, the regularity of the solution in this scale differs from the regularity in the $L_2$-Sobolev scale $H^r(\dom)$, $r\geq0$. For instance, for non-convex polygonal domains, the functions in $D(\Lap)$ are in general {\it not} in the Sobolev space $H^{1+\pi/\gamma}(\dom)$, where $\gamma\in(\pi,2\pi)$ is the largest interior angle at a vertice of $\partial\dom$, cf.~\cite{Gri85}. 
Thus, if one applies the typical regularity results from the semigroup approach to SPDEs to equations on non-smooth domains, 
the spatial singularities of the solution process induced by the shape of the domain remain somewhat hidden behind the abstract framework.

We present a regularity result concerning the solution $u=(u(t))_{t\in[0,T]}$ to Eq.~\eqref{SHE} which, on the one hand, gives an explicit bound for the $L_2(\dom)$-Sobolev regularity of $u$ and, on the other hand, describes the singular behavior of $u$ induced by the shape of the domain.
It is based on and improves in several directions the corresponding result in \cite[Chapter 4]{Lin11}. 

To give a first description of the result, assume for simplicity that, in a neighborhood of zero, the domain $\dom$ coincides with the sector
\[\big\{x\in\R^2:x=(x_1,x_2)=(r\cos\theta,r\sin\theta),\;r>0,\;\theta\in (0,\gamma)\big\},\]
where $\gamma\in(\pi,2\pi)$. Also assume that all interior angles at vertices of $\partial\dom$ except the one at zero are smaller than $\pi$. Set $\alpha:=\pi/\gamma\in (1/2,1)$ and consider the corner singularity function  for the Poisson problem
\begin{equation*}\label{defSIntro}
S(x):=\eta(x)r^\alpha\sin(\alpha\theta),\quad x=(r\cos\theta,r\sin\theta)\in\dom,
\end{equation*}
where $\eta$ is a smooth cut-off function that equals one near zero and vanishes in a neighborhood
of the sides of $\partial\dom$ which do not end at zero. Assume that $u$ is continuous in $L_2(\dom;\R)$ and that the noise term in \eqref{SHE} is regular enough for $u$ to take values in $H^1_0(\dom)$ and to satisfy $\sup_{t\in[0,T]}\E\nnrm{u(t)}{H^1_0(\dom)}^2<\infty$, where $H^1_0(\dom)$ is the $L_2(\dom)$-Sobolev space of order one with zero Dirichlet boundary condition. Denote by $u_+$ the extension of $u$ by zero to the whole time axis $\R$ and pick $s>1/2$. Then, if $(\Omega,\cA,\wP)$ is the underlying probability space, Theorem~\ref{mainRes} below states that there exist
\[u_{+,\mathrm R}\in L_2\big(\Omega;H^{-s}(\R)\hat\otimes H^2(\dom)\big)\quad\text{ and }\quad \Phi\in L_2\big(\Omega;H^{(1-\alpha)/2-s}(\R)\big)\]
with $\supp\Phi(\omega)\subset [0,\infty)$ for all $\omega\in\Omega$, such that the decomposition
\begin{equation}\label{mainResIntro}
u_+=u_{+,\mathrm R}+\Phi* E_0\, S
\end{equation}
holds as an equality in the space $L_2(\Omega;H^{-s}(\R)\hat\otimes H^1_0(\dom))$. Here $\hat\otimes$ denotes the Hilbert-Schmidt tensor product, $*$ is the convolution in time, and $E_0:\R\times\dom\to\R$ is an auxiliary kernel function with support in $[0,\infty)\times\dom$.
Using a linear and bounded extension operator from $H^s(0,T)$ to $H^s(\R)$, we obtain in Corollary~\ref{mainResCor} a corresponding decomposition 
\begin{equation}\label{mainResCorIntro}
u=u_{\mathrm R}+u_{\mathrm S}
\end{equation}
in the space $L_2(\Omega;(H^s(0,T))'\hat\otimes H^1_0(\dom))$. The regular part $u_{\mathrm R}$ satisfies 
\[u_{\mathrm R}\in L_2\big(\Omega;(H^s(0,T))'\hat\otimes H^2(\dom)\big)\] 
and the singular part $u_{\mathrm S}$ contains the corner singularity function $S$. The precise meaning of the decompositions \eqref{mainResIntro} and \eqref{mainResCorIntro} is explained in Sections~\ref{Setting and assumptions} and \ref{Main result} below.  We also derive estimates for the norms of $u_{+,\mathrm R}$, $u_{\mathrm R}$ and $\Phi$, and we show that 
\[\Phi(\omega)* E_0\, S\notin H^{-s}(\R)\hat\otimes H^{1+\alpha}(\dom),\quad u_{\mathrm S}(\omega)\notin (H^s(0,T))'\hat\otimes H^{1+\alpha}(\dom)\]
whenever $\Phi(\omega)$ is not zero. The random element $\Phi\in L_2(\Omega;H^{(1-\alpha)/2-s}(\R))$ is determined by $u$, $F$, $G$ and $W$.

The fact that the components in the decompositions \eqref{mainResIntro} and \eqref{mainResCorIntro} have negative regularity in time is owed to the low temporal regularity of the driving Wiener process $W=(W(t))_{t\in[0,T]}$. It is, so to speak, the price we pay for unveiling the structure of that part of the spatial singular behavior of $u=(u(t))_{t\in[0,T]}$ which is due to the specific geometry of $\dom$. The decomposition \eqref{mainResIntro} can be considered as a stochastic version of Grisvard's result for the deterministic heat equation; cf.~\cite[Section~5]{Gri87}, \cite[Section~5.2]{Gri92}. We follow Grisvard's ansatz of using the Laplace transform w.r.t.\ the time variable $t$ in order to turn the equation into an elliptic equation with parameter. The solution to the elliptic equation can then be decomposed explicitly into a regular and a singular part. In a last step, the Laplace transform is inverted. The main difficulty in the stochastic case is to handle the irregularity of the noise and, connected with it, to handle the fact that the stochastic integrals are not defined pathwise, but in an $L_2(\wP)$-sense. Besides, it takes a careful analysis to keep track of the measurability in $\omega\in\Omega$ of all random objects appearing in the course of the calculations. We use It\^o's formula to transform Eq.~\eqref{SHE} into a random elliptic equation with complex parameter. The main technical tool to derive the necessary estimates for the regular and the singular part of the solution to the transformed equation is Lemma~\ref{X(phi)Modification}, which describes the effect of the temporal irregularity of the noise in an appropriate way. After choosing suitable versions of all random objects, the inverse transform can be carried out $\omega$ by $\omega$. We use a framework of tensor products of (duals of) Sobolev spaces to make sense of the resulting objects, which are random generalized functions in time when applied to spatial testfunctions. 

The article is structured as follows. In Section~\ref{Setting and assumptions} we describe the setting and all assumptions concering Eq.~\eqref{SHE} (Subsection~\ref{Stochastic heat equation}), the framework of tensor products of Sobolev spaces and how the solution process $u=(u(t))_{t\in[0,T]}$ is embedded into this framework (Subsection~\ref{The solution process as a tensor product-valued random variable}).
We follow the semigroup approach to SPDEs in Subsection~\ref{Stochastic heat equation}, but we note that this is not essential and that other approaches could be used to derive similar results. Several supplementary details concerning tensor products of Sobolev spaces are postponed to Appendix~\ref{Tensor products of Sobolev spaces}. In Section~\ref{Main result} we formulate our main result for the case of polygonal domains with exactly one non-convex corner in Theorem~\ref{mainRes}. Here, Corollary~\ref{mainResCor} is stated and proved, and the results are illustrated with conrete examples. The extensions of  Theorem~\ref{mainRes} and Corollary~\ref{mainResCor} to the general case of arbitrary bounded polygonal domains are stated Appendix~\ref{General bounded polygonal domains}. Auxiliary results concerning the Laplace transform (Subsection~\ref{A Paley-Wiener type theorem}) and the Helmholtz equation (Subsection~\ref{Estimates for the Helmholtz equation}) are collected in Section~\ref{Auxiliary results}. The proof of Theorem~\ref{mainRes} is given in Section~\ref{Proof of the main result}, which consists of three subsections concerning the Laplace transform of Eq.~\eqref{SHE} (Subsection~\ref{Laplace transform of the stochastic heat equation}), the decomposition of the transformed equation (Subsection~\ref{Decomposition of the transformed equation}) and the inverse transform (Subsection~\ref{Inverse transform}).

\vspace{0.2cm}
\noindent\textbf{Notation and Conventions.}
The Borel-$\sigma$-algebra on a normed space $X=(X,\nnrm{\cdot}X)$ w.r.t.\ the topology induced by the norm $\nnrm{\cdot}X$ is denoted by $\cB(X)$. If $X$ is a Banach space, $(M,\cM,\mu)$ a $\sigma$-finite measure space and $1\leq p<\infty$, we write $L_p(M,\cM,\mu; X)$ for the space of all ($\mu$-equivalence classes of) strongly measurable functions $f:M\to X$ with finite $L_p$-norm $\nnrm{f}{L_p(M,\cM,\mu;X)}:=(\int_M\nnrm{f}X^p\dl\mu)^{1/p}$. If the context is clear, we also write $L_p(M;X)$ instead of $L_p(M,\cM,\mu;X)$. For $X=\C$ we omit the notation of the image space, i.e., $L_p(M,\cM,\mu):=L_p(M,\cM,\mu;\C)$. If $D$ is a subset of $\R^d$ we set $L_2(D):=L_2(D,\cB(D),\lambda^d;\C)$ and $L_2(D;\R):=L_2(D,\cB(D),\lambda^d;\R)$, where $\lambda^d$ denotes Lebesgue measure. The $L_2$-Sobolev-Slobodeckij space of order $s\geq0$ on a domain $D\subset\R^d$ is denoted by $H^s(D)$; see Appendix~\ref{Tensor products of Sobolev spaces} for the definition. $H^s_0(D)$ is the closure of the space of compactly supported, smooth testfunctions $C_0^\infty(D)$ within $H^s(D)$. If $M$ or $D$ is the open interval $(0,T)$, we write $L_2(0,T;X)$, $H^s(0,T)$ and $H^s_0(0,T)$ instead of $L_2((0,T);X)$, $H^s((0,T))$ and $H^s_0((0,T))$. The spaces of rapidly decreasing, smooth functions and tempered distributions on $\R^d$ are denoted by $\cS(\R^d)$ and $\cS'(\R^d)$, respectively. All spaces of (generalized) functions are understood as complex vector spaces of $\C$-valued functions, unless explicitly indicated otherwise, e.g., by writing $L_2(D;\R)$, $H^s(D;\R)$ and $H^s_0(D;\R)$. By $\mathcal F:\mathcal S'(\R^d)\to\mathcal S'(\R^d)$ we denote the Fourier transform on $S'(\R^d)$, normed according to $(\mathcal Ff)(\xi)=\int_{\R^d} e^{-i\xi x}f(x)\,dx$, $\xi\in\R^d$, $f\in L_1(\R^d)$. All derivatives of (locally integrable) functions defined on domains in $\R^d$ are meant in the distributional sense.  
The duality form of a topological vector space $X$ and its (topological) dual $X'$ is denoted by $\langle\cdot,\cdot\rangle_{X\times X'}$, i.e., $\langle x,x'\rangle_{X\times X'}:=x'(x)$ for all $x\in X$, $x'\in X'$. The (topological) dual $\cH'$ of a Hilbert-space $\cH$ is always endowed with the strong dual topology and the respective norm $\nnrm{\cdot}{\cH'}:=\sup_{\nnrm{h}{\cH}\leq1}|\langle h,\cdot\rangle_{\cH\times\cH'}|$. The notation $X\hookrightarrow Y$ means that a topological vector space $X$ is linearly and continuously embedded into another topological vector space $Y$. Inner products $\langle\cdot,\cdot\rangle_\cH$ of complex Hilbert spaces $\cH$ are assumed to be conjugate linear in the second argument. Composite expressions following an expectation sign `$\E$' are evaluated prior to taking the expectation, e.g., $\E\nnrm{\ldots}\H^p:=\E(\nnrm{\ldots}\H^p)$. For separable Hilbert spaces $\cH$ and $\cG$, we denote by $\bo(\cH;\cG)$, $\hs(\cH;\cG)$, $\nuc(\cH;\cG)$ the spaces linear and bounded operators, Hilbert-Schmidt operators and nuclear operators, respectively. If $\cH=\cG$, we write  $\bo(\cH)$, $\hs(\cH)$ and $\nuc(\cH)$ instead of $\bo(\cH;\cH)$, $\hs(\cH;\cH)$ and $\nuc(\cH;\cH)$. Throughout the paper, $C$ denotes a positive and finite constant which may change its value with every new appearance.

\section{Setting and assumptions}
\label{Setting and assumptions}

\subsection{Stochastic heat equation}
\label{Stochastic heat equation}

Throughout this paper, $\dom\subset\R^2$ denotes a simply connected
and bounded open subset of $\R^2$ with polygonal boundary $\partial\dom$ such that $\dom$ lies only on one side of $\partial\dom$. The generic element in $\R^2$ is denoted by $x=(x_1,x_2)$. Let
\[\Lap: D(\Lap)\subset L_2(\dom)\to L_2(\dom)\]
be the Dirichlet-Laplacian with domain
\[D(\Lap)=\Big\{v\in H^1_0(\dom):\Delta v:=\frac{\partial^2}{\partial x_1^2}v+\frac{\partial^2}{\partial x_2^2}v \in L_2(\dom)\Big\}.\]
By $\big(e^{t\Lap}\big)_{t\geq0}$ we denote the analytic semigroup of contractions on $L_2(\dom)$ generated by $\Lap$. Note that the operators $\Lap$ and $e^{t\Lap}$, $t\geq0$, map real-valued functions to real-valued functions, i.e., they can also be considered as operators on $L_2(\dom;\R)$.

Let $(\Omega,\mathcal A,\wP)$ be a complete probability space, $T\in(0,\infty)$ and let $(\mathcal F_t)_{t\in[0,T]}$ be a normal filtration of sub-$\sigma$-algebras of $\mathcal A$. On $(\Omega,\cA,\wP)$ let $W=(W(t))_{t\in[0,T]}$ be a $U$-valued Wiener process w.r.t.\ $(\mathcal F_t)_{t\in[0,T]}$, $U$ being a real and separable Hilbert space. The covariance operator and the reproducing kernel Hilbert space of $W$ are denoted by $Q\in\nuc(U)$ and $\big(U_0,\langle \cdot,\cdot\rangle_{U_0}\big):=\big(Q^{1/2}U,\langle Q^{-1/2}\,\cdot\,,Q^{-1/2}\,\cdot\,\rangle_U\big)$, respectively. Here $Q^{-1/2}$ is the pseudo-inverse of $Q^{1/2}$. Standard references for this setting are \cite{DaPraZab}, \cite{PesZab}, \cite{PreRoeck}.

We are interested in the regularity of the mild solution $u=(u(t))_{t\in[0,T]}$ to Eq.\ \eqref{SHE},
where $F$ and $G$ are mappings from $L_2(\dom;\R)$ to $L_2(\dom;\R)$ and to $\bo(U_0;L_2(\dom;\R))$, respectively. We make the following assumptions on $F$, $G$ and the initial condition $u_0$.
\begin{assumption}\label{AssFBu0}
$G$ takes values in the space $\hs(U_0;L_2(\dom;\R))$ of Hilbert-Schmidt operators; the mappings $F:L_2(\dom;\R)\to L_2(\dom;\R)$ and $G:L_2(\dom;\R)\to \hs(U_0;L_2(\dom;\R))$ are globally Lipschitz continuous. The initial condition $u_0$ satisfies \[u_0\in L_2\big(\Omega,\mathcal F_0,\wP;H^1_0(\dom;\R)\big)\cap L_p\big(\Omega,\mathcal F_0,\wP;L_2(\dom;\R)\big)\]
for some $p>2$.
\end{assumption}

By a {\it mild solution} to Eq.\ \eqref{SHE} we mean an $(\cF_t)_{t\in[0,T]}$-predictable $L_2(\dom;\R)$-valued stochastic process $u=(u(t))_{t\in[0,T]}$ on $(\Omega,\cA,\wP)$ such that $\sup_{t\in[0,T]}\E\nnrm{u(t)}{L_2(\dom)}^2<\infty$ and for every $t\in[0,T]$ the equality
\[u(t)=e^{t\Lap}u_0+\int_0^te^{(t-s)\Lap}F(u(s))\dl s+\int_0^t e^{(t-s)\Lap}G(u(s))\dl W(s)\]
holds $\wP$-almost surely. It is well-known that in the described setting the concept of a mild solution is equivalent to the concept of a so-called weak solution; see, e.g., \cite[Theorem~9.15]{PesZab}. The following existence and regularity result is a consequence of \cite[Theorem~7.4]{DaPraZab} and \cite[Theorem 4.2]{KruLar12}, compare also \cite[Theorem~11.8]{PesZab} and \cite{JenRoeck12}.

\begin{thm}\label{thmKruLar}
Given Assumption \ref{AssFBu0}, there exists a unique (up to modifications) mild solution $u=(u(t))_{t\in[0,T]}$ to Eq.\ \eqref{SHE}. It has a unique (up to indistinguishability) continuous modification in $L_2(\dom;\R)$, i.e., a modification $\tilde u=(\tilde u(t))_{t\in[0,T]}$ such that for all $\omega\in\Omega$ the trajectory
\[[0,T]\ni t\mapsto\tilde u(\omega,t):=\tilde u(t)(\omega)\in L_2(\dom;\R)\]
is continuous. This modification satisfies
\begin{equation}\label{DPZ(7.17)}
\E\sup_{t\in[0,T]}\nnrm{\tilde u(t)}{L_2(\dom;\R)}^p<C\big(1+\E\nnrm{u_0}{L_2(\dom;\R)}^p\big).
\end{equation}
 Moreover, for all $t\in[0,T]$ we have $u(t)\in L_2(\Omega;H^1_0(\dom;\R))$ and $t\mapsto u(t)$ is continuous as a mapping from $[0,T]$ to $L_2(\Omega;H^1_0(\dom;\R))$.
\end{thm}


In the sequel we also write $u(\omega,t)$ or $u(\omega,t,\cdot)$ instead of $u(t)(\omega)$, and a trajectory $[0,T]\in t\mapsto u(\omega,t)\in L_2(\dom;\R)$ may be denoted by $u(\omega)$ or $u(\omega,\cdot)$. This notation is motivated by the viewpoint of considering the solution as a scalar function of $(\omega,t,x)\in \Omega\times[0,T]\times\dom$, which turns out to be convenient for our purpose.

Let us look at concrete examples for $W$, $U$, $U_0$ and $F$, $G$.
\begin{ex}\label{exSHE}
{\bf (i)} (additive trace class noise)
Let $W$ be an $L_2(\dom;\R)$-valued Wiener process and set $U:=L_2(\dom;\R)$. Define $G:L_2(\dom;\R)\to\hs(U_0;L_2(\dom;\R))$ as the constant mapping with value $\id{L_2(\dom;\R)}\in\bo(L_2(\dom;\R))\hookrightarrow\hs(U_0;L_2(\dom;\R))$. The embedding holds since the reproducing kernel Hilbert space $U_0$ of $W$ is embedded into $L_2(\dom;\R)$ via a Hilbert-Schmidt embedding. Let $f$ be a real-valued function on $\dom\times\R$ satisfiying the following condition: There exist $C>0$ and $b\in L_2(\dom;\R)$ such that, for all $x\in\dom$ and $\xi,\eta\in\R$,
\[|f(x,\xi)|\leq b(x)+C|\xi|,\quad |f(x,\xi)-f(x,\eta)|\leq C|\xi-\eta|.\]
Define $F:L_2(\dom;\R)\to L_2(\dom;\R)$ by
\[\big(F(v)\big)(x):=f\big(x,v(x)\big),\quad v\in L_2(\dom;\R),\;x\in\dom.\]
Then, the conditions on $F$ and $G$ in Assumption \ref{AssFBu0} are fulfilled and Eq.\ \eqref{SHE} is an abstract formulation of the problem
\begin{equation}\label{exProblem}
\left\{
\begin{aligned}
\dl u(t,x)&=\big[\Delta u(t,x)+f\big(x,u(t,x)\big)\big]\dl t + \dl W(t,x),\; &&(t,x)\in[0,T]\times\dom,\\
u(t,x)&=0, &&(t,x)\in[0,T]\times\partial\dom,\\
u(0,x)&=u_0(x), &&x\in\dom.
\end{aligned}
\right.
\end{equation}

{\bf (ii)} (multiplicative noise with sufficient smoothness)
Let $W$ be an $H^s(\dom;\R)$-valued Wiener process for some $s>1$ and set $U:=H^s(\dom;\R)$. Such a process can be obtained, e.g., by applying an integral operator on $L_2(\dom;\R)$ with sufficiently smooth kernel $k\in$$L_2(\dom\times\dom;\R)$ to a cylindrical Wiener process on $L_2(\dom;\R)$.
Let $g$ be a real-valued function on $\dom\times\R$ satisfying the same condition as formulated for the function $f$ in (i). Define $G:L_2(\dom;\R)\to \bo\big(H^s(\dom;\R);L_2(\dom;\R)\big)\hookrightarrow\hs(U_0;L_2(\dom;\R))$ by
\[\big(G(v)w\big)(x):=g\big(x,v(x)\big)w(x),\quad v\in L_2(\dom;\R),\;w\in H^s(\dom;\R),\;x\in\dom,\]
and let $F$ be defined as in (i).
Then, the conditions on $F$ and $G$ in Assumption \ref{AssFBu0} are fulfilled (note that we have the Sobolev embedding $H^s(\dom;\R)\hookrightarrow C_\text{b}(\dom;\R$)) and Eq.\ \eqref{SHE} is an abstract formulation of Problem \eqref{exProblem} if the first line in \eqref{exProblem} is replaced by
\[\dl u(t,x)=\big[\Delta u(t,x)+f\big(x,u(t,x)\big)\big]\dl t + g\big(x,u(t,x)\big)\dl W(t,x),\quad (t,x)\in[0,T]\times\dom.\]

{\bf (iii)} (multiplicative finite-dimensional noise)
Let $W=(W_1,\ldots,W_d)$ be a $d$-dimensional Wiener process and set $U:=U_0:=\R^d$.
Let $g_1,\ldots,g_d$ be real-valued functions on $\dom\times\R$ satisfying the same condition as formulated for the function $f$ in (i). Define $G:L_2(\dom;\R)\to\hs(U_0;L_2(\dom;\R))$ by
\[\big(G(v)w\big)(x):=\sum_{k=1}^d g_k\big(x,v(x)\big)w_k,\quad v\in L_2(\dom;\R),\;w=(w_1,\dots,w_d)\in \R^d,\;x\in\dom,\]
and let $F$ be defined as in (i).
Then, the conditions on $F$ and $G$ in Assumption \ref{AssFBu0} are fulfilled and Eq.\ \eqref{SHE} is an abstract formulation of Problem \eqref{exProblem} if the first line in \eqref{exProblem} is replaced by
\[\dl u(t,x)=\big[\Delta u(t,x)+f\big(x,u(t,x)\big)\big]\dl t + \sum_{k=1}^d g_k\big(x,u(t,x)\big)\dl W_k(t),\quad (t,x)\in[0,T]\times\dom.\]
\end{ex}

\begin{rem}
All the noise terms in Example \ref{exSHE} can be rewritten in the general form
$\sum_{k=1}^\infty g_k\big(x,u(t,x)\big)\dl W_k(t)$, where the $g_k$'s are suitably chosen functions on $\dom\times\R$ and the $W_k$'s are independent one-dimensional Wiener processes; compare, e.g., \cite[Section 8.2]{Kry99}.
\end{rem}

\subsection{The solution process as a tensor product-valued random variable}
\label{The solution process as a tensor product-valued random variable}

Our main result, Theorem \ref{mainRes}, and Corollary \ref{mainResCor} are formulated in terms of tensor products of Sobolev spaces of possibly negative order. In the present subsection we define the tensor product spaces, point out their natural embeddings and describe how the mild solution $u=(u(t))_{t\in[0,T]}$ to Eq.~\eqref{SHE} can be considered as a tensor product-valued random variable. Since the tensor product spaces we consider are rather non-standard in the context of stochastic evolution equations, we collect several supplementary details and references in Appendix~\ref{Tensor products of Sobolev spaces}.

With regard to the natural embeddings of tensor products of Sobolev spaces, it is convenient to define the (Hilbert-Schmidt) tensor product of two Hilbert spaces as the space of Hilbert-Schmidt functionals on the cartesian product of the {\it duals} of these spaces. The connection to alternative definitions in the literature is described in Appendix~\ref{Tensor products of Sobolev spaces}.
Let $\cH$ and $\cG$ be separable complex Hilbert spaces with orthonormal bases $(h_j)_{j\in\N}$ and $(g_k)_{k\in\N}$, respectively. Following \cite[Section 2.6]{KaRi83} we call a  {\it Hilbert-Schmidt functional on $\cH\times\cG$} a bounded bilinear functional $f:\cH\times\cG\to\C$, $(h,g)\mapsto f(h,g)$ such that
$\sum_{j,k\in\N}\big|f(h_j, g_k)\big|^2<\infty.$
The infinite sum does not depend on the specific choice of the orthonormal bases, and its square root defines a norm that makes the space of Hilbert-Schmidt functionals into a separable Hilbert space; see \cite[Section 2.6]{KaRi83}. Here, boundedness of $f$ means  $\sup_{\nnrm{h}\cH,\nnrm{g}\cG\leq 1}|f(h,g)|<\infty$.

\begin{definition}\label{defTP}
The Hilbert-Schmidt tensor product $\cH\hat\otimes \cG$ of two separable complex Hilbert spaces $\cH$ and $\cG$ is defined as the space of Hilbert-Schmidt functionals on $\cH'\times \cG'$ with norm given by
\begin{equation*}
\nnrm{f}{\cH\hat\otimes \cG}:=\left(\sum_{j,k\in\N}|f(h_j',g_k')|^2\right)^{1/2},\qquad f\in \cH\hat\otimes \cG,
\end{equation*}
$(h_j')_{j\in\N}$ and $(g_k')_{k\in\N}$ being arbitrary orthonormal bases of $\cH'$ and $\cG'$. For $h\in \cH$ and $g\in \cG$ we denote by $h\otimes g\in \cH\hat\otimes \cG$ the functional defined by 
\begin{equation*}
h\otimes g(h',g')
:=\langle h,h'\rangle_{\cH\times \cH'}\langle g,g'\rangle_{\cG\times \cG'}=h'(h)g'(g),\qquad h'\in \cH',\,g'\in \cG'.
\end{equation*}
\end{definition}

Given an arbitrary domain $D\subset\R^d$, we denote by $H^s(D)$ the $L_2(D)$-Sobolev-Slobodeckij space of order $s\geq0$; see Appendix~\ref{Tensor products of Sobolev spaces} for the definition. It is well known that for $0\leq s_1\leq s_2$ the space $H^{s_2}(D)$ is densely embedded into $H^{s_1}(D)$ via the identity operator.

\begin{convention}\label{convTP}
We identify $L_2(D)$ with its (topological) dual space $(L_2(D))'$ via the isometric isomorphism
$L_2(D)\ni v\mapsto \langle\,\cdot\,,\overline v\rangle_{L_2(D)}\in (L_2(D))'$. Thus, by duality, we 
obtain a chain of continuous and dense (linear) embeddings
\begin{equation}\label{embeddingsHs}
H^{s_2}(D)\hookrightarrow H^{s_1}(D)\hookrightarrow L_2(D)= (L_2(D))'\hookrightarrow (H^{s_1}(D))'\hookrightarrow (H^{s_2}(D))',\quad 0\leq s_1\leq s_2.
\end{equation}
Moreover, for $s\geq0$ we identify $H^s(D)$ with its bidual $(H^s(D))'':=\big((H^s(D))'\big)'$ via the canonical isometric isomorphism $\i:H^s(D)\to(H^s(D))''$ given by $(\i(f))(g)=g(f)$ for $f\in H^s(D)$, $g\in(H^s(D))'$.
We also write $\langle\cdot,\cdot\rangle_{(H^s(D))'\times H^s(D)}$ for the duality form $\langle\cdot,\cdot\rangle_{(H^s(D))'\times(H^s(D))''}$.
\end{convention}

As usual, for $s\geq0$ we denote by $H^s_0(D)$ the closure of $C^\infty_0(D)$ in $H^s(D)$ and by $H^{-s}(D):=(H^s_0(D))'$ its dual space.
If $D=\R^d$, we have $H^{-s}(\R^d)=(H^s(\R^d))'$ with equal norms, and \eqref{embeddingsHs} reads 
\begin{equation*}
H^{s_2}(\R^d)\hookrightarrow H^{s_1}(\R^d)\hookrightarrow L_2(\R^d)= (L_2(\R^d))'\hookrightarrow H^{-s_1}(\R^d)\hookrightarrow H^{-s_2}(\R^d),\quad 0\leq s_1\leq s_2.
\end{equation*}

Theorem \ref{mainRes} and Corollary \ref{mainResCor} are formulated in terms of tensor product spaces of the form  $(H^s(I))'\hat\otimes H^r(\dom)$ and $(H^s(I))'\hat\otimes H^1_0(\dom)$, $r,s\geq0$, where $I=(0,T)$ or $I=\R$. By Definition~\ref{defTP} and Convention~\ref{convTP}, the elements of $(H^s(I))'\hat\otimes H^r(\dom)$ are Hilbert-Schmidt functionals on $H^s(I)\times (H^r(\dom))'=(H^s(I))''\times (H^r(\dom))'$; the elements of $(H^s(I))'\hat\otimes H^1_0(\dom)$ are Hilbert-Schmidt functionals on $H^s(I)\times H^{-1}(\dom)=(H^s(I))''\times (H^1_0(\dom))'$.
According to Proposition \ref{TPemb} in Appendix~\ref{Tensor products of Sobolev spaces}, we have natural embeddings
\begin{equation}\label{TPemb1}
(H^{s_1}(I))'\hat\otimes H^{r_2}(\dom)\hookrightarrow (H^{s_2}(I))'\hat\otimes H^{r_1}(\dom),\quad 0\leq s_1\leq s_2,\;0\leq r_1\leq r_2,
\end{equation}
given by the tenor products $\i\hat\otimes\j$ of the embeddings $\i:(H^{s_1}(I))'\hookrightarrow (H^{s_2}(I))'$ and $\j:H^{r_2}(\dom)\hookrightarrow  H^{r_1}(\dom)$ as in \eqref{embeddingsHs}. The image $\i\hat\otimes\j\,(f)\in (H^{s_2}(I))'\hat\otimes H^{r_1}(\dom)$ of some $f\in(H^{s_1}(I))'\hat\otimes H^{r_2}(\dom)$ is nothing but the restriction of the bilinear functional $f:H^{s_1}(I)\times (H^{r_2}(\dom))'\to\C$ to the smaller domain $H^{s_2}(I)\times (H^{r_1}(\dom))'$. 
 Also by Proposition \ref{TPemb},
we have natural embeddings 
\begin{equation}\label{TPemb3}
(H^{s_1}(I))'\hat\otimes H^1_0(\dom)\hookrightarrow(H^{s_2}(I))'\hat\otimes H^1_0(\dom)\hookrightarrow(H^{s_2}(I))'\hat\otimes H^1(\dom),\quad 0\leq s_1\leq s_2.
\end{equation}
Note, however, that the second embedding in \eqref{TPemb3} is not dense. 

Let us describe in which sense the mild solution $u=(u(t))_{t\in[0,T]}$ to Eq.~\eqref{SHE} will be considered as a tensor-product valued random variable. To this end, take an arbitrary (predictable) version of $u$. We know from Theorem~\ref{thmKruLar} that $u(t)\in H^1_0(\dom;\R)$ $\wP$-almost surely for all $t\in[0,T]$. By the Kuratowski-Suslin theorem we have $H^1_0(\dom;\R)\in\cB(L_2(\dom;\R))$, so that $P:=\big\{(\omega,t)\in\Omega\times[0,T]:u(\omega,t)\not\in H^1_0(\dom)\big\}$ belongs to $\cP$, the predictable $\sigma$-algebra w.r.t.\ the filtration $(\mathcal F_t)_{t\in[0,T]}$. Consequently, by redefining $u(\omega,t):=0$ for all $(\omega,t)\in P$ we obtain a predictable ($\cP/\cB(L_2(\dom;\R))$-measurable) modification of our original solution such that $u(\omega,t)\in H^1_0(\dom;\R)$ for all $(\omega,t)\in\Omega\times[0,T]$. We fix this modification $u=(u(t))_{t\in[0,T]}$ from now on. From the $\cP/\cB(L_2(\dom;\R))$-measurability of $u$ and a standard approximation argument we obtain the $\cP/\cB(H^1_0(\dom;\R))$-measurability of $u$. Moreover, since $\sup_{t\in[0,T]}\E\nnrm{u(t)}{H^1_0(\dom;\R)}<\infty$ by Theorem~\ref{thmKruLar} and $H^1_0(\dom;\R)\hookrightarrow H^1_0(\dom)$, we have 
\begin{equation}\label{uinH10}
u\in L_2\big(\Omega\times[0,T],\cP,\wP\otimes\dl t;H^1_0(\dom)\big).
\end{equation}
By $u_+=(u_+(t))_{t\in\R}$ we denote the extension
of $u$ by zero to the whole real line. We will consider $u$ and $u_+$ as random variables with values in the spaces
$L_2(0,T)\hat\otimes H^1_0(\dom)$ and $L_2(\R)\hat\otimes H^1_0(\dom)$, respectively.

\begin{proposition}\label{tildeuu+}
After possibly redefining $u$ on a $\wP$-null set, the definitions of the mappings
\begin{align*}
\tilde u:\;&\Omega\to L_2(0,T)\hat\otimes H^1_0(\dom),\;\omega\mapsto\tilde u(\omega),\\
\tilde u_+:\;&\Omega\to L_2(\R)\hat\otimes H^1_0(\dom),\;\omega\mapsto\tilde u_+(\omega)
\end{align*} 
by
\[
\begin{aligned}
\tilde u(\omega)(\phi,\varphi):= \tilde u(\omega,\phi,\varphi):=
\int_0^T\langle u(\omega,t),\varphi&\rangle_{H^1_0(\dom)\times  H^{-1}(\dom)}\phi(t)\dl t,\\
&\phi\in L_2(0,T),\;\varphi\in  H^{-1}(\dom).
\end{aligned}
\]
and
\[
\begin{aligned}
\tilde u_+(\omega)(\phi,\varphi):= \tilde u_+(\omega,\phi,\varphi):=
\int_\R\langle u_+(\omega,t),&\varphi\rangle_{H^1_0(\dom)\times  H^{-1}(\dom)}\phi(t)\dl t,\\
&\phi\in L_2(\R),\;\varphi\in H^{-1}(\dom)
\end{aligned}
\]
are meaningful, and $\tilde u$ and $\tilde u_+$ belong to the spaces $L_2\big(\Omega,\cF_T,\wP;L_2(0,T)\hat\otimes H^1_0(\dom)\big)$ and
$L_2\big(\Omega,\cF_T,\wP;L_2(\R)\hat\otimes H^1_0(\dom)\big)$, respectively.
\end{proposition}

\begin{proof}
By \eqref{uinH10} and the theorems of Tonelli and Fubini, we know that all trajectories
\[u(\omega)=u(\omega,\cdot):[0,T]\to H^1_0(\dom),\;t\mapsto u(\omega,t),\qquad \omega\in\Omega,\] are $\mathcal B([0,T])/\mathcal B(H^1_0(\dom))$-measurable and that $u(\omega)\in L_2(0,T;H^1_0(\dom))$ for $\wP$-almost all $\omega\in\Omega$. After redefining $u(\omega):=0$ for all $\omega\in\Omega$ such that $u(\omega)\notin L_2(0,T;H^1_0(\dom))$, another application of the theorems of Tonelli and Fubini shows that the mapping $\Omega\ni\omega\mapsto u(\omega)\in L_2(0,T;H^1_0(\dom))$ is
$\cF_T/\mathcal B\big(L_2(0,T;H^1_0(\dom))\big)$-measurable; compare~\cite[Proposition~3.18]{DaPraZab}. Moreover, one has
$\int_\Omega\nnrm{u(\omega)}{L_2(0,T;H^1_0(\dom))}^2\wP(\dl\omega)=\nnrm{u}{L_2(\Omega\times[0,T],\mathcal P,\wP\otimes\dl t;H^1_0(\dom))}^2$.
The assertion concerning $\tilde u$ now follows from the fact that the operator $J:L_2(0,T;H^1_0(\dom))\mapsto L_2(0,T)\hat\otimes H^1_0(\dom)$,
which maps $f\in L_2(0,T;H^1_0(\dom))$ to the bilinear functional
\[(Jf)(\phi,\varphi):L_2(0,T)\times H^{-1}(\dom)\to\C,\;(\phi,\varphi)\mapsto
\int_0^T\langle f(t),\varphi\rangle_{H^1_0(\dom)\times H^{-1}(\dom)}\phi(t)\dl t,\]
is an isometric isomorphism. The assertion concerning $\tilde u_+$ follows analogously.
\end{proof}

We will always take for granted the redefinition of $u$ on a $\wP$-null set mentioned in Proposition \ref{tildeuu+}. This means that {\it all}
trajectories $[0,T]\ni t\mapsto u(\omega,t)\in H^1_0(\dom)$ and $\R\ni t\mapsto u_+(\omega,t)\in H^1_0(\dom)$, $\omega\in\Omega$, belong to $L_2(0,T;H^1_0(\dom))$ and $L_2(\R;H^1_0(\dom))$, respectively;
the mappings $\Omega\ni\omega\mapsto u(\omega)\in L_2(0,T;H^1_0(\dom))$ and $\Omega\ni\omega\mapsto u_+(\omega)\in L_2(\R;H^1_0(\dom))$ are
$\cF_T/\mathcal B\big(L_2(0,T;H^1_0(\dom))\big)$-measurable and $\cF_T/\mathcal B\big(L_2(\R;H^1_0(\dom))\big)$-measurable, respectively.

\begin{convention}\label{conventionu}
We identify the mild solution $u$ to Eq.\ \eqref{SHE} and its extension by zero to
the whole real line $u_+$ with the mappings $\tilde u$ and $\tilde u_+$ described
in Proposition \ref{tildeuu+}. We set $u(\omega,\phi,\varphi):=\tilde u(\omega,\phi,\varphi)$ and
$u_+(\omega,\phi,\varphi):=\tilde u_+(\omega,\phi,\varphi)$ for $\omega\in\Omega$,
$\varphi\in H^{-1}(\dom)$, $\phi\in L_2(0,T)$ and $\phi\in L_2(\R)$, respectively. In this sense we
have
\[u\in L_2\big(\Omega,\cF_T,\wP;L_2(0,T)\hat\otimes H^1_0(\dom)\big),
\quad u_+\in L_2\big(\Omega,\cF_T,\wP;L_2(\R)\hat\otimes H^1_0(\dom)\big).\]
\end{convention}

Note that, due to the embedding \eqref{TPemb3}, we have in particular
\[u\in L_2\big(\Omega,\cF_T,\wP;(H^s(0,T))'\hat\otimes H^1_0(\dom)\big),\quad u_+\in L_2\big(\Omega,\cF_T,\wP;H^{-s}(\R)\hat\otimes H^1_0(\dom)\big)\]
for all $s\geq0$.

\section{Main result}
\label{Main result}
Before formulating the main result we need to introduce some further notation. In order to keep the notational complexity at a reasonable level we make the following additional assumption on
the domain $\dom\subset\R^2$. We remark, however, that our results readily generalize to arbitrary bounded polygonal
domains as defined in Subsection \ref{Stochastic heat equation}; see Appendix \ref{General bounded polygonal domains} for the formulation of the results in the general case.
\begin{assumption}\label{AssO}
The domain $\dom$ has {\it exactly one non-convex corner}. The corresponding vertex is zero and the corresponding interior angle is denoted
by $\gamma\in(\pi,2\pi)$. In a neighborhood of zero, $\dom$ coincides
with the sector
\[\big\{x\in\R^2:x=(r\cos\theta,r\sin\theta),\;r>0,\;\theta\in (0,\gamma)\big\}.\]
\end{assumption}
Let $\eta\in C^\infty(\overline\dom;\R)$ be a smooth cut-off function that depends only on $r=\sqrt{x_1^2+x_2^2}$, equals one in a neighborhood of zero and vanishes in a neighborhood
of the sides of $\partial\dom$ which do not end at zero.
Set $\alpha:=\pi/\gamma\in (1/2,1)$ and define $S\in H^1_0(\dom)$ by
\begin{equation}\label{defS}
S(x):=\eta(x)r^\alpha\sin(\alpha\theta),\quad x=(x_1,x_2)=(r\cos\theta,r\sin\theta)\in\dom.
\end{equation}
The function $S$ belongs to $H^s(\dom)$ if, and only if, $s<1+\alpha$; see \cite[Theorem 1.4.5.3]{Gri85}. It represents the corner singularity for the Poisson problem on $\dom$ with zero-Dirichlet boundary condition; see \cite{Gri85}, \cite{Gri92}. That is, given $g\in L_2(\dom)$ and $w\in H^1_0(\dom)$ with $-\Delta w=g$, there exist a unique function $w_\text{R}\in H^2(\dom)\cap H^1_0(\dom)$ and a unique constant $c\in\C$ such that $w=w_{\text{R}}+cS$. 
It follows from \cite[Proposition 2.5.6]{Gri92} (compare also \cite[Section 2]{Gri87}) that $c=\langle g,v_0\rangle_{L_2(\dom)}$, where $v_0\in L_2(\dom;\R)$ is 
defined as $v_0=(1/\pi)(\psi_0-\varphi_0)$ with
\[\psi_0(x):=\eta(x)r^{-\alpha}\sin(\alpha\theta),\quad x=(x_1,x_2)=(r\cos\theta,r\sin\theta)\in\dom, \]
and $\varphi_0\in D(\Lap)$ being the unique solution in $H^1_0(\dom)$ to the problem $\Delta\varphi_0=\Delta\psi_0$. (Note that $\psi_0$ does not belong to $H^1(\dom)$, but it satisfies $\Delta\psi_0\in L_2(\dom)$ since it is harmonic near $0$.) For $z\in\C\setminus \sigma(\Lap)$ we define $v(z)\in L_2(\dom)$ by
\begin{equation}\label{v1}
\begin{aligned}
v(z)&:=v_0-z\big(z\id{L_2(\dom)}-\Lap\big)^{-1}v_0\\
&\;=\left[\id{L_2(\dom)}-z\big(z\id{L_2(\dom)}-\Lap\big)^{-1}\right]\frac 1\pi(\psi_0-\varphi_0),
\end{aligned}
\end{equation}
where $\big(z\id{L_2(\dom)}-\Lap\big)^{-1}\in\bo(L_2(\dom))$ is the $z$-resolvent of $\Lap$.

Further, we define a kernel function $E_0:\R\times\dom\to \R$ by
\begin{equation}\label{defE0}
E_0(t,x):=\one_{(0,\infty)}(t)(2\sqrt\pi)^{-1}t^{-3/2}r e^{-r^2/(4t)},\quad t\in\R,\;x=(x_1,x_2)=(r\cos\theta,r\sin\theta)\in\dom.
\end{equation}
For fixed $x\in\dom$ the function $t\mapsto E_0(t,x)$ is the inverse Laplace transform of $(0,\infty)+i\R\ni z\mapsto e^{-r\sqrt z}\in\C$; see \cite[Section 8.4]{Foe93} or \cite[Exercise 3A/3]{Gue91}.

Finally, let Assumption \ref{AssFBu0} hold, let $u=(u(t))_{t\in[0,T]}$ be the mild solution to Eq.~\eqref{SHE} and define
\begin{equation}\label{H(z)}
H(z):=\int_0^T e^{-zt}F(u(t))\dl t+\int_0^T e^{-zt}G(u(t))\dl W(t)-e^{-zT}u(T)+u_0,\quad z\in \C.
\end{equation}
The first integral in \eqref{H(z)} is an $\omega$-wise Bochner integral in $L_2(\dom)$. For  every $\omega\in\Omega$, all integrals $\int_0^Te^{-zt}F(u(\omega,t))\dl t,\;z\in\C$, exist since we have
$\int_0^T\nnrm{e^{-zt}F(u(\omega,t))}{L_2(\dom)}\dl t\leq$\linebreak$C\int_0^Te^{-\Re zt}(1+\nnrm{u(\omega,t)}{L_2(\dom)})\dl t<\infty$.
Moreover, 
\begin{align*}
\E\sgnnrm{\int_0^Te^{-zt}F(u(t))\dl t}{L_2(\dom)}^2
&\leq C\E\int_0^T\gnnrm{e^{-zt}F(u(t))}{L_2(\dom)}^2\dl t\\
&\leq C\E\int_0^Te^{-2\Re zt}\left(1+\nnrm{u(t)}{L_2(\dom;\R)}^2\right)\dl t<\infty.
\end{align*}
The second integral in \eqref{H(z)} is an $L_2(\dom)$-valued stochastic integral; for fixed $\omega$ and $t$,
$e^{-zt}G(u(\omega,t))$ is the operator in $\hs(L_2(\dom;\R);L_2(\dom))$
that maps $w\in L_2(\dom;\R)$ to \linebreak$e^{-zt}G(u(\omega,t))w\in L_2(\dom)$.
By It\^{o}'s isometry and the Lipschitz property of $G$, 
\begin{align*}
\E\sgnnrm{\int_0^Te^{-zt}G(u(t))\dl W(t)}{L_2(\dom)}^2
&=\E\int_0^T\gnnrm{e^{-zt}G(u(t))}{\hs(L_2(\dom;\R);L_2(\dom))}^2\dl t\\
&\leq C\E\int_0^Te^{-2\Re zt}\left(1+\nnrm{u(t)}{L_2(\dom;\R)}^2\right)\dl t<\infty,
\end{align*} so that we obtain
\begin{equation}\label{H(z)inL2}
H(z)\in L_2(\Omega,\cF_T,\wP;L_2(\dom)), \; z\in\C.
\end{equation}
We will later show (Lemma \ref{Gholom}) that the $L_2(\dom)$-valued random field $(H(z))_{z\in\C}$ has a holomorphic modification, i.e.,
a modification such that for all $\omega\in\Omega$ the mapping $\C\ni z\mapsto H(\omega,z):=H(z)(\omega)\in L_2(\dom)$ is holomorphic. We fix such a modification
once and for all.
\begin{rem}\label{RemStoIntCR}
The Hilbert-space theory of infinite-dimensional stochastic integrals is usually developed in terms of {\it real} Hilbert-spaces, cf.~\cite{DaPraZab}, \cite{Met}, \cite{MetPel80}, \cite{PesZab}, \cite{PreRoeck}. In the context of stochastic integrals such as in \eqref{H(z)}
we will in general consider $\C$-valued functions as $\R^2$-valued functions, and we will in general understand the
stochastic integrals in terms of the respective real Hilbert-spaces of
$\R^2$-valued functions. We do not indicate this explicitly, but we will point out this identification whenever it is needed.
\end{rem}

Here is our main result.

\begin{thm}\label{mainRes}
Let Assumptions \ref{AssFBu0} and \ref{AssO} hold, let $u=(u(t))_{t\in[0,T]}$ be the mild solution to Eq.~\eqref{SHE} and let $u_+$ be its extension by zero to the whole real line, considered as an element of
$L_2\big(\Omega,\cF_T,\wP;L_2(\R)\hat\otimes H^1_0(\dom)\big)$ as described in Subsection~\ref{The solution process as a tensor product-valued random variable}. Let $s> 1/2$ and set $\alpha:=\pi/\gamma$. 

There exist
\[u_{+,\mathrm R}\in L_2\big(\Omega,\cF_T,\wP;H^{-s}(\R)\hat\otimes H^2(\dom)\big)\cap L_2\big(\Omega,\cF_T,\wP;H^{-s}(\R)\hat\otimes H^1_0(\dom)\big)\]
and
\[\Phi\in L_2\big(\Omega,\cF_T,\wP; H^{(1-\alpha)/2-s}(\R)\big)\]
with $\supp\Phi(\omega)\subset [0,\infty)$ for all $\omega\in\Omega$
(in the sense of distributions) such that the equality
\begin{equation*}\label{mainRes1}
u_+=u_{+,\mathrm R}+\Phi*E_0\,S
\end{equation*}
holds in $L_2\big(\Omega,\cF_T,\wP;H^{-s}(\R)\hat\otimes H^1_0(\dom)\big)$.
Here $\Phi*E_0\,S$ denotes the element of the space $L_2\big(\Omega,\cF_T,\wP;H^{-s}(\R)\hat\otimes H^1_0(\dom)\big)$ that acts on test functions $(\phi,\varphi)\in H^s(\R)\times L_2(\dom)$ 
\linebreak$(\hookrightarrow H^s(\R)\times H^{-1}(\dom))$ via
\begin{equation}\label{defPhi*E_0S}
\begin{aligned}
\big(\Phi*E_0\,S\big)(\omega)(\phi,\varphi)
&:= \big(\Phi*E_0\,S\big)(\omega,\phi,\varphi)\\
&:=\Big\langle \Phi(\omega)*\int_\dom E_0(\cdot,x)S(x)\varphi(x)\dl x,\phi\Big\rangle_{H^{-s}(\R)\times H^s(\R)},\; \omega\in\Omega,
\end{aligned}
\end{equation}
where $S$ and $E_0$ are given by \eqref{defS} and \eqref{defE0}, $\int_\dom E_0(\cdot,x)S(x)\varphi(x)\dl x$ denotes the (locally integrable) function
$\R\ni t\mapsto\int_\dom E_0(t,x)S(x)\varphi(x)\dl x\in\C$, and $*$ is the usual convolution of Schwartz distributions.

We have
\begin{equation}\label{mainResuSnotin}
\big(\Phi*E_0\,S\big)(\omega)\notin \bigcup_{r\geq0}H^{-r}(\R)\hat\otimes H^{1+\alpha}(\dom)
\;\text{ on }\; \{\omega\in\Omega:\Phi(\omega)\not\equiv0\}
\end{equation}
and $\Phi$ is determined $\wP$-almost surely in terms of its Fourier transform w.r.t.\ the time variable $t\in\R$ as follows: For $\wP$-almost
every $\omega\in\Omega$,
\begin{equation}\label{mainResDefPhi}
\big[\mathcal F_{t\to\xi}\big(\Phi(\omega)\big)\big](\xi)=\Big\langle H(\omega,i\xi),
\overline{v(i\xi)}\Big\rangle_{L_2(\dom)} \text{ for $\lambda$-almost every }\xi\in\R,
\end{equation}
where $v$ and $H$ are defined by \eqref{v1} and \eqref{H(z)}.

Moreover, 
\begin{equation}\label{mainResEst}
\begin{aligned}
&\E\Big(\nnrm{u_{+,\mathrm R}}{H ^{-s}(\R)\hat\otimes H^2(\dom)}^2
+\nnrm{\Phi}{H^{(1-\alpha)/2-s}(\R)}^2\Big)\\
&\leq C\E\Big(\nnrm{u_0}{L_2(\dom)}^2+\nnrm{u(T)}{L_2(\dom)}^2+\int_0^T\gnnrm{F\big(u(t)\big)}{L_2(\dom)}^2\dl t+\sup_{t\in[0,T]}\gnnrm{G\big(\tilde u(t)\big)}{\hs(U_0;L_2(\dom))}^2\Big),
\end{aligned}
\end{equation}
where $C>0$ depends only on $s$, $T$, $\dom$ and the cut-off function $\eta$ in \eqref{defS}, and where $\tilde u=(\tilde u(t))_{t\in[0,T]}$ denotes the modification of $u=(u(t))_{t\in[0,T]}$ that is continuous in $L_2(\dom;\R)$.
\end{thm}

The proof of Theorem \ref{mainRes} is given is Section \ref{Proof of the main result}. Some remarks concerning Theorem~\ref{mainRes} seem to be in order. 
\begin{rem}
{\bf (i)}
It is a common convention not to distinguish explicitly between functions and equivalence classes of functions. The existence of $\Phi\in L_2\big(\Omega,\cF_T,\wP; H^{(1-\alpha)/2-s}(\R)\big)$ stated in Theorem~\ref{mainRes} is meant as the existence of an $\cF_T/\cB(H^{(1-\alpha)/2-s}(\R))$-measurable, square integrable {\it function} $\Phi:\Omega\to H^{(1-\alpha)/2-s}(\R)$ such that, for all $\omega\in\Omega$, $\supp\Phi(\omega)\subset[0,\infty)$ and the mapping $(\Phi*E_0\,S)(\omega):H^{s}(\R)\times L_2(\dom)\to\C$ defined by \eqref{defPhi*E_0S} extends to a Hilbert-Schmidt mapping on $H^{s}(\R)\times H^{-1}(\dom)$, i.e., to an element of $H^{-s}(\R)\hat\otimes H^1_0(\dom)$. (In Subsection~\ref{Inverse transform of c e^{-r sqrt z}S} we will show implicity that the convolution $\Psi *\int_\dom E_0(\cdot,x)S(x)\varphi(x)\dl x$ belongs to $H^{-s}(\R)$ for all $\Psi\in H^{-s}(\R)$ and $\varphi\in L_2(\dom)$.) The resulting mapping $\Phi*E_0\,S:\Omega\to H^{-s}(\R)\hat\otimes H^1_0(\dom)$ is $\cF_T/\cB(H^{-s}(\R)\hat\otimes H^1_0(\dom))$-measurable and represents an element of the space $L_2\big(\Omega,\cF_T,\wP;H^{-s}(\R)\hat\otimes H^1_0(\dom)\big)$.

{\bf (ii)}
The assertion `$\Phi(\omega)\not\equiv0$' in \eqref{mainResuSnotin} means that there exists $\phi\in H^{-(1-\alpha)/2+s}(\R)$ such that $\langle\Phi(\omega),\phi\rangle_{H^{(1-\alpha)/2-s}(\R)\times H^{-(1-\alpha)/2+s}(\R)}\neq0$.

{\bf (iii)}
For $z\in\C\setminus(-\infty,0)$, an alternative representation of the functions $v(z)\in L_2(\dom)$, defined by \eqref{v1} and appearing in \eqref{mainResDefPhi}, is described in Remark~\ref{remHelmholtz} below. 

{\bf (iv)}
Note that the expectation on the right hand side of \eqref{mainResEst} is finite due to Theorem~\ref{thmKruLar} and the linear growth property of $F$ and $G$, which follows from the global Lipschitz property.
\end{rem}

The following corollary describes a corresponding
decomposition of $u$ within the space $L_2\big(\Omega,\cF_T,\wP;(H^s(0,T))'\hat\otimes H^1_0(\dom)\big)$. For the construction of a linear and bounded extension operator $\mathcal E:H^s(0,T)\to H^s(\R)$ we refer to \cite[Section~4.2]{Tri78}.

\begin{cor}\label{mainResCor}
Let the setting of Theorem \ref{mainRes} be given and consider the mild solution $u=(u(t))_{t\in[0,T]}$ to Eq.~\eqref{SHE} as an element of $L_2\big(\Omega,\cF_T,\wP;L_2(0,T)\hat\otimes H^1_0(\dom)\big)$ as described in Subsection~\ref{The solution process as a tensor product-valued random variable}. Let $\mathcal E:H^s(0,T)\to H^s(\R)$ be a linear and bounded extension operator. For $\omega\in\Omega$, $\phi\in H^s(0,T)$ and $\varphi\in H^ {-1}(\dom)$ define
\[
u_{\mathrm R}(\omega,\phi,\varphi):=u_{+,\mathrm R}(\omega,\mathcal E\phi,\varphi),\qquad
u_{\mathrm S}(\omega,\phi,\varphi):=\big(\Phi*E_0\,S\big)(\omega,\mathcal E\phi,\varphi),
\]
where $u_{+,\mathrm R}$ and $\Phi*E_0\,S$ are as in Theorem \ref{mainRes}.

Then, 
\begin{equation}\label{mainResCor1}
\begin{aligned}
u_{\mathrm R}\in L_2\big(\Omega,\cF_T,\wP;(&H^{s}(0,T))'\hat\otimes H^2(\dom)\big)\cap L_2\big(\Omega,\cF_T,\wP;(H^{s}(0,T))'\hat\otimes H^1_0(\dom)\big),\\
u_{\mathrm S}&\in L_2\big(\Omega,\cF_T,\wP;(H^{s}(0,T))'\hat\otimes H^1_0(\dom)\big),
\end{aligned}
\end{equation}
and the decomposition
\begin{equation}\label{mainResCor2}
u=u_{\mathrm R}+u_{\mathrm S}
\end{equation}
holds as an equality in $L_2\big(\Omega,\cF_T,\wP;(H^{s}(0,T))'\hat\otimes H^1_0(\dom)\big)$. 
For $\wP$-almost every $\omega\in\Omega$,
\begin{equation}\label{mainResCor3}
u_{\mathrm S}(\omega)\notin\bigcup_{r\geq0}(H^r(0,T))'\hat\otimes H^{1+\alpha}(\dom)
\quad\Longleftrightarrow\quad
u_{\mathrm S}(\omega)\not\equiv0
\quad\Longleftrightarrow\quad
\Phi(\omega)\not\equiv0.
\end{equation}

Moreover, there exists a constant $C>0$, depending only on $s$, $T$, $\dom$ and the cut-off function $\eta$ in \eqref{defS}, such that
\begin{equation}\label{mainResCorEst}
\begin{aligned}
\E\nnrm{u_{\mathrm R}}{(H ^s(0,T))'\hat\otimes H^2(\dom)}^2
\leq C\E\Big(&\nnrm{u_0}{L_2(\dom)}^2+\nnrm{u(T)}{L_2(\dom)}^2\\
&+\int_0^T\gnnrm{F\big(u(t)\big)}{L_2(\dom)}^2\dl t+\sup_{t\in[0,T]}\gnnrm{G\big(\tilde u(t)\big)}{\hs(U_0;L_2(\dom))}^2\Big),
\end{aligned}
\end{equation}
where $\tilde u=(\tilde u(t))_{t\in[0,T]}$ denotes the modification of $u=(u(t))_{t\in[0,T]}$ that is continuous in $L_2(\dom;\R)$.
\end{cor}

The assertion `$u_{\mathrm S}(\omega)\not\equiv0$' in \eqref{mainResCor3} means that there exist $\phi\in H^s(0,T)$ and $\varphi\in H^{-1}(\dom)$ with $u_{\mathrm S}(\omega,\phi,\varphi)\neq0$.

\begin{proof}[Proof of Corollary~\ref{mainResCor}]
Assertions~\eqref{mainResCor1}, \eqref{mainResCor2} and Estimate~\eqref{mainResCorEst} follow from the boundedness of $\mathcal E$ and the corresponding properties of $u_{+,\mathrm R}$ and $\Phi*E_0\,S$. 

Let us verify \eqref{mainResCor3}. The implications in direction ``$\Rightarrow$''  are obvious,
so it remains to prove that the implication
\begin{equation}\label{mainResCor4}
\Phi(\omega)\not\equiv0
\quad\Longrightarrow\quad 
u_{\mathrm S}(\omega)\not\in\bigcup_{r\geq0}(H^r(0,T))'\hat\otimes H^{1+\alpha}(\dom)
\end{equation}
holds for $\wP$-almost every $\omega\in\Omega$. 
Let $\cR\in\bo\big(H^s(\R);H^s(0,T)\big)$ be the restriction
of functions in $H^s(\R)$ to $(0,T)$, i.e., $\cR\phi=\phi|_{(0,T)}$ for $\phi\in H^s(\R)$. For $\omega\in\Omega$, $\phi\in H^s(\R)$
and $\varphi\in H^{-1}(\dom)$ we set
\[
u_{\mathrm R}^{\cR}(\omega,\phi,\varphi):=u_{\mathrm R}(\omega,\cR\phi,\varphi),\qquad
u_{\mathrm S}^{\cR}(\omega,\phi,\varphi):=u_{\mathrm S}(\omega,\cR\phi,\varphi).
\]
Due to \eqref{mainResCor1} and the boundedness of $\cR$ we have $u_{\mathrm R}^\cR\in L_2\big(\Omega;(H^{s}(0,T))'\hat\otimes H^2(\dom)\big)\cap L_2\big(\Omega;(H^{s}(0,T))'\hat\otimes H^1_0(\dom)\big)$ and $u_{\mathrm S}^\cR\in L_2\big(\Omega;(H^{s}(0,T))'\hat\otimes H^1_0(\dom)\big)$.
Since $\cR|_{H^r(\R)}$ belongs to $\bo\big(H^r(\R);H^r(0,T)\big)$ for all $r\geq s$, the implication
\[u_{\mathrm S}^\cR(\omega)\notin\bigcup_{r\geq0}H^{-r}(\R)\hat\otimes H^{1+\alpha}(\dom)\quad\Longrightarrow\quad 
u_{\mathrm S}(\omega)\notin\bigcup_{r\geq0}(H^r(0,T))'\hat\otimes H^{1+\alpha}(\dom)\]
holds, so that \eqref{mainResCor4} follows if we can show 
\begin{equation}\label{mainResCor5}
\Phi(\omega)\not\equiv0
\quad\Longrightarrow\quad 
u_{\mathrm S}^\cR(\omega)\notin\bigcup_{r\geq0}H^{-r}(\R)\hat\otimes H^{1+\alpha}(\dom)
\end{equation}
for $\wP$-almost all $\omega\in\Omega$.

Assume that \eqref{mainResCor5} was not true and let $A\in\cA$ with $\wP(A)>0$
such that, for all $\omega\in A$, $\Phi(\omega)\not\equiv 0$ and $u_{S}^{\mathcal R}(\omega)\in H^{-r}(\R)\hat\otimes H^{1+\pi/\gamma}(\dom)$ for some $r=r(\omega)\geq s$.
Then, since $u_{+,\mathrm S}(\omega)\notin H^{-r}(\R)\hat\otimes H^{1+\alpha}(\dom)$ by Theorem~\ref{mainRes}, we have
$u_{\mathrm S}^{\mathcal R}(\omega)-u_{+,\mathrm S}(\omega)\notin H^{-r}(\R)\hat\otimes H^{1+\alpha}(\dom)$ for all $\omega\in A$.
Therefore, the $\wP$-almost sure equality $u_+=u_{+,\mathrm R}+u_{+,\mathrm S}=u_{\mathrm R}^{\mathcal R}+u_{\mathrm S}^{\mathcal R}$ implies 
\[u_{+,\mathrm R}(\omega)-u_{\mathrm R}^{\mathcal R}(\omega)
=u_{\mathrm S}^{\mathcal R}(\omega)-u_{+,\mathrm S}(\omega)
\notin H^{-r}(\R)\hat\otimes H^{1+\alpha}(\dom)\]
for $\omega\in A\setminus N$, where $N\subset\cA$ has $\wP$-measure zero.
But this is a contradiction to the fact that
both $u_{+,\mathrm R}$ and $u_{\mathrm R}^{\mathcal R}$ take values in $H^{-s}(\R)\hat\otimes H^2(\dom)\hookrightarrow H^{-r}(\R)\hat\otimes H^{1+\alpha}(\dom)$.
\end{proof}

Next, we give two concrete examples for applications of Theorem~\ref{mainRes} and Corollary~\ref{mainResCor}, respectively, to equations of the type \eqref{SHE}. In the first example we have $u_{\textrm S}\not\equiv0$ with probability one.

\begin{ex}\label{ex1}
Let the Wiener process $W=(W(t))_{t\in[0,T]}$ be $L_2(\dom;\R)$-valued, i.e., $U=L_2(\dom;\R)$, and assume that the range of its covariance operator $Q\in\nuc\big(L_2(\dom;\R)\big)$ is dense in $L_2(\dom;\R)$. For  $u_0\in L_2\big(\Omega,\mathcal F_0,\wP;H^1_0(\dom;\R)\big)\cap L_p\big(\Omega,\mathcal F_0,\wP;L_2(\dom;\R)\big)$
with $p>2$ consider the equation
\[\dl u(t)=\Lap u(t)\dl t+\dl W(t),\quad u(0)=u_0,\;t\in[0,T],\] 
which fits into our abstract setting with $F(v):=0$ and $G(v):=\id{L_2(\dom;\R)}$ for all $v\in L_2(\dom;\R)$, cf.\ Example~\ref{exSHE} {\bf(i)}. (Let us remark that, since we are considering an equation with additive noise, the assumption  $u_0\in L_p\big(\Omega,\mathcal F_0,\wP;L_2(\dom;\R)\big)$ is not really needed here to obtain the results of Theorem~\ref{mainRes} and
Corollary \ref{mainResCor}; we do not go into details.) Fix $s>1/2$ and let
$u_{+,\mathrm R}$, $\Phi$, $u_{\mathrm R}$ and $u_{\mathrm S}$ be as in Theorem~\ref{mainRes} and
Corollary \ref{mainResCor}. Using It\^{o}'s isometry, one sees that in this setting the estimates~\eqref{mainResEst} and \eqref{mainResCorEst} simplify to
\begin{equation*}
\begin{aligned}
\E\Big(\nnrm{u_{+,\mathrm R}}{H ^{-s}(\R)\hat\otimes H^2(\dom)}^2
+\nnrm{u_{\mathrm R}}{(H ^s(0,T))'\hat\otimes H^2(\dom)}^2&+\nnrm{\Phi}{H^{(1-\alpha)/2-s}(\R)}^2\Big)\\
&\leq C\Big(\E\nnrm{u_0}{L_2(\dom)}^2+\nnrm{Q^{1/2}}{\hs(L_2(\dom;\R))}^2\Big)
\end{aligned}
\end{equation*}
with $C>0$ depending only on $s$, $T$ and $\dom$.

Let us show that $u_{+,\mathrm S}$ and $u_{\mathrm S}$ are non-zero $\wP$-almost-surely. To this end, we have to show that $\Phi$ defined
by \eqref{mainResDefPhi} is non-zero $\wP$-almost-surely. Recall that we have fixed a holomorphic modification of the
$L_2(\dom)$-valued random field $(H(z))_{z\in\C}$. The resolvent map
$\C\setminus\sigma\big(\Lap\big)\ni u\mapsto\big(z\id{L_2(\dom)}-\Lap\big)^{-1}\in \bo(L_2(\dom))$ is holomorphic too, so that the function $\R\ni\xi\mapsto\big\langle H(\omega,i\xi),
\overline{v(i\xi)}\big\rangle_{L_2(\dom)}\in\C$ is infinitely smooth
for all $\omega\in\Omega$. 
Thus, the Fourier transform $\R\ni\xi\mapsto\big[\mathcal F_{t\to\xi}\Phi(\omega)\big](\xi)\in\C$ has an infinitely smooth, hence continuous version for $\wP$-almost all $\omega\in\Omega$; we consider these continuous versions from now on.
Consequently, it is enough to show that
$\big[\mathcal F_{t\to\xi}\Phi(\omega)\big](0)=\big\langle H(\omega,0),
\overline{v(0)}\big\rangle_{L_2(\dom)}=\langle H(\omega,0),
v_0\rangle_{L_2(\dom;\R)}$ is non-zero for
almost all $\omega\in\Omega$.

Observe that
\begin{equation}\label{eqEx1}
\begin{aligned}
\langle H(0),v_0\rangle_{L_2(\dom;\R)}\;=\;&
\Big\langle\int_0^T\big(\id{L_2(\dom;\R)}-e^{(T-t)\Lap}\big)\dl W(t),v_0\Big\rangle_{L_2(\dom;\R)}\\
&+\left\langle u_0-e^{T\Lap}u_0,v_0\right\rangle_{L_2(\dom;\R)}.
\end{aligned} 
\end{equation}
Setting $\Psi(t) w:=\big\langle\big(\id{L_2(\dom;\R)}-e^{(T-t)\Lap}\big)w,v_0\big\rangle_{L_2(\dom;\R)}$ for $w\in L_2(\dom;\R)$, we obtain
\begin{align*}
\E\Big(\Big\langle\int_0^T\big(\id{L_2(\dom;\R)}-e^{(T-t)\Lap}\big)&\dl W(t),v_0\Big\rangle_{L_2(\dom;\R)}^2\Big)\\
&=\E\Big(\Big|\int_0^T\Psi(t)\dl W(t)\Big|^2\Big)\\
&=\int_0^T\gnnrm{\Psi(t)Q^{1/2}}{\hs(L_2(\dom;\R);\R)}^2\dl t\\
&=\int_0^T\gnnrm{Q^{1/2}\big(\id{L_2(\dom;\R)}-e^{(T-t)\Lap}\big)v_0}{L_2(\dom;\R)}^2\dl t
\end{align*}
due to It\^{o}'s isometry, Parseval's identity and the symmetry of $Q^{1/2}$. Clearly, the kernel of $\id{L_2(\dom;\R)}-e^{(T-t)\Lap}$ is zero for all $t<T$.
Since the kernel of $Q^{1/2}=(Q^{1/2})^*$ is zero too (due to our assumption that the range of $Q$ is dense in $L_2(\dom;\R)$) and since
$v_0\neq0$, the last integral is strictly positive. This means that the Gaussian random variable
$\big\langle\int_0^T\big(\id{L_2(\dom;\R)}-e^{(T-t)\Lap}\big)\dl W(t),v_0\big\rangle_{L_2(\dom;\R)}$ is not degenerate and thus its probability distribution has a density w.r.t.\ Lebesgue measure. Note that $\big\langle u_0-e^{T\Lap}u_0,v_0\big\rangle_{L_2(\dom;\R)}$ in \eqref{eqEx1} is $\cF_0$-measurable and thus independent of $\int_0^T\big(\id{L_2(\dom;\R)}-e^{(T-t)\Lap}\big)\dl W(t)$. It follows that the probability distribution of $\langle H(0),v_0\rangle_{L_2(\dom;\R)}$ has a density w.r.t.\ Lebesgue measure, hence $\langle H(0),v_0\rangle_{L_2(\dom;\R)}$ is non-zero $\wP$-almost surely.
\end{ex}

We end this section with a toy example which shows that it may happen that $u_S\not\equiv0$ with probability greater than zero and less than one.

\begin{ex}\label{ex2}
Let the Wiener process $W=(W(t))_{t\in[0,T]}$ be one-dimensional, i.e., $U=U_0=\R$, and consider the domain $\dom=(-1,1)^2\setminus(-1,0]^2\subset\R^2$. Define $u_0\in H^2(\dom)\cap H^1_0(\dom)$ by
\[u_0(x):=\sin(\pi x_1)\sin(\pi x_2),\quad x=(x_1,x_2)\in\dom,\]
so that $-\Delta u_0=2\pi^2u_0$. Let $v_0=v(0)=(1/\pi)(\psi_0-\varphi_0)$ be as in \eqref{v1} and define $G:L_2(\dom;\R)\to L_2(\dom;\R)\cong\bo(\R;L_2(\dom;\R))=\hs(\R;L_2(\dom;\R))$ by
\[G(v):= u_0+g\big(\nnrm{v}{L_2(\dom;\R)}\big)v_0,\quad v\in L_2(\dom;\R),\]
where $g:\R\to\R$ is Lipschitz-continuous with $\supp g=[c,\infty)$ for some $c>0$. Then the equation 
\begin{equation*}
\dl u(t)=\Lap u(t)\dl t+ G(u(t))\dl W(t),\quad u(0)=u_0,\; t\in[0,T],
\end{equation*}
fits into the setting of Theorem~\ref{mainRes} and Corollary~\ref{mainResCor}, and we have $0<\wP(u_{\mathrm S}\neq0)<1$.

To prove the latter assertion, we fix a continuous version of the Fourier transform $\R\ni\xi\mapsto\big[\mathcal F_{t\to\xi}\Phi(\omega)\big](\xi)\in\C$ for $\wP$-almost all $\omega\in\Omega$ as in Example~\ref{ex1}. Then 
\[\wP\big(\langle H(0),
v_0\rangle_{L_2(\dom;\R)}\neq0\big)=\wP\big(\big[\mathcal F_{t\to\xi}\Phi\big](0)\neq0\big)\leq \wP(\Phi\not\equiv0)=\wP(u_{\mathrm S}\not\equiv0)\]
and
\begin{align*}
\langle H(0),v_0\rangle_{L_2(\dom;\R)}
&=\int_0^T\left[\left\langle\big(1-e^{-(T-t)2\pi^2}\big)u_0,v_0\right\rangle_{L_2(\dom;\R)}\right.\\
&\qquad\quad\left. +g\big(\nnrm{u(t)}{L_2(\dom;\R)}\big)\left\langle\big(\id{L_2(\dom;\R)}-e^{(T-t)\Lap}\big)v_0,v_0\right\rangle_{L_2(\dom;\R)}\right]\dl W(t)\\
&\quad +\left\langle\big(1-e^{-T2\pi^2}\big)u_0,v_0\right\rangle_{L_2(\dom;\R)}\\
&=\int_0^Tg\big(\nnrm{u(t)}{L_2(\dom;\R)}\big)\sgnnrm{\big(\id{L_2(\dom;\R)}-e^{(T-t)\Lap}\big)^{1/2}v_0}{L_2(\dom;\R)}^2\dl W(t). 
\end{align*}
The last step is due to the fact that $\langle w,v_0\rangle_{L_2(\dom)}=0$ for all $w\in \Delta\left(H^2(\dom)\cap H^1_0(\dom)\right)$, a consequence of Green's formula; see \cite[Theorem 2.3.3 and Lemma 2.5.4]{Gri92}.
By It\^{o}'s isometry,
\[\E\left(\langle H(0),v_0\rangle_{L_2(\dom;\R)}^2\right)=\E\int_0^Tg\big(\nnrm{u(t)}{L_2(\dom;\R)}\big)^2\sgnnrm{\big(\id{L_2(\dom;\R)}-e^{(T-t)\Lap}\big)^{1/2}v_0}{L_2(\dom;\R)}^4\dl t\]
and therefore we know that $\wP(u_{\mathrm S}\not\equiv0)>0$ if 
\begin{equation}\label{ex2eq1}
\wP\otimes\dl t\left(\Big\{(\omega,t)\in\Omega\times[0,T]:g\big(\nnrm{u(\omega,t)}{L_2(\dom;\R)}\big)>0\Big\}\right)>0.
\end{equation}
Note that $g\big(\nnrm{u(\omega,t)}{L_2(\dom;\R)}\big)>0$ if, and only if, $\nnrm{u(\omega,t)}{L_2(\dom;\R)}>c$. Due to the orthogonality of $u_0$ and $v_0$,
\[
\nnrm{u(t)}{L_2(\dom;\R)}^2=\nnrm{u_0}{L_2(\dom;\R)}^2(X(t))^2\\
+\sgnnrm{\int_0^tg\big(\nnrm{u(s)}{L_2(\dom;\R)}\big)e^{(t-s)\Lap}v_0\,\dl W(s)}{L_2(\dom;\R)}^2,
\]
where we have set $X(t):=e^{-t2\pi^2}+\int_0^te^{-(t-s)2\pi^2}\dl W(s)$.
Thus,
\begin{align*}
\wP\otimes\dl t\Big(\Big\{(\omega,t)\in\;&\Omega\times[0,T]:\nnrm{u(\omega,t)}{L_2(\dom;\R)}>c\Big\}\Big)\\
&\geq \wP\otimes\dl t\Bigg(\Bigg\{(\omega,t)\in\Omega\times[0,T]:|X(\omega,t)|>\frac c{\nnrm{u_0}{L_2(\dom;\R)}}\Bigg\}\Bigg)>0
\end{align*}
by standard properties of the one-dimensional Ornstein-Uhlenbeck process $X=(X(t))_{t\in[0,T]}$. This proves \eqref{ex2eq1}, hence $\wP(u_{\mathrm S}\not\equiv0)>0$. To see that $\wP(u_{\mathrm S}\not\equiv0)<1$, we fix a continuous modification of $X$ and estimate
\begin{align*}
\wP(u_{\mathrm S}\equiv0)&\geq \wP\Big(u\in L_2(0,T;H^2(\dom))\cong L_2(0,T)\hat\otimes H^2(\dom)\Big)\\
&\geq  \wP\Bigg(\sup_{t\in[0,T]}|X(t)|<\frac c{\nnrm{u_0}{L_2(\dom;\R)}}\Bigg)>0.
\end{align*}
The penultimate estimate holds because $\tilde u=(\tilde u(t))_{t\in[0,T]}$ defined by 
\[\tilde u(t):=\one_{\big\{\sup_{s\leq t}|X(s)|<\frac c{\nnrm{u_0}{L_2(\dom;\R)}}\big\}}X(t)u_0
+\one_{\big\{\sup_{s\leq t}|X(s)|\geq\frac c{\nnrm{u_0}{L_2(\dom;\R)}}\big\}}u(t)\]
is a predictable modification of $u$ (due to the uniqueness of the solution to Eq.~\eqref{SHE}) so that the equality $u(\omega,\cdot)=\tilde u(\omega,\cdot)$ holds in $L_2(0,T;H^1_0(\dom))$ for $\wP$-almost all $\omega\in\Omega$.
\end{ex}

\section{Auxiliary results}
\label{Auxiliary results}

\subsection{A Paley-Wiener type theorem}
\label{A Paley-Wiener type theorem}

As mentioned in the introduction, Theorem~\ref{mainRes} will be proved with the help of a Laplace transform argument. If we used only the Fourier transform instead, we would run into technical troubles when carrying out the inverse transform leading to the singular part \linebreak$\Phi*E_0S$. Moreover, we would not be able to show that the support of $\Phi(\omega)\in$\linebreak$ H^{-s+1/2(1-\alpha)}(\R)$,  defined for $\wP$-almost every $\omega\in\Omega$ by \eqref{mainResDefPhi}, is contained in $[0,\infty)$. The latter is a consequence of a Paley-Wiener type result, which we present in this subsection. Its proof is similar to the ones of Theorems~8.2-3 and 8.4-1 in \cite{Zem87}. However, since our assertion is slightly different, we present it here for the sake of completetness.

The Laplace transform of a tempered distribution $\phi\in\mathcal S'(\R)$ with $\supp \phi\subseteq[0,\infty)$ can be defined, at least for all $z\in (0,\infty)+i\R$, by setting 
\begin{equation}\label{defLT}
(\mathcal L\phi)(z):=\left\langle\lambda\, e^{-z(\cdot)},\phi\right\rangle_{\cS(\R)\times\cS'(\R)},
\end{equation}
cf.~\cite[Chapter~8]{Zem87}.
Here, $\lambda$ is a $C^\infty(\R)$-function with support bounded on the left, which equals one in a neighborhood of $[0,\infty)$, and $e^{-z(\cdot)}$ denotes the function $\R\ni t\mapsto e^{-zt}\in\C$.  The right hand side in \eqref{defLT} makes sense since $\lambda\,e^{-z(\cdot)}\in\mathcal S(\R)$ and the definition obviously does not depend on the specific choice of the function $\lambda$. If $\phi$ is a regular distribution, then $(\mathcal L\phi)(z)=\int_0^\infty e^{-zt}\phi(t)\,dt$. Remember that we use the normalization $(\cF \phi)(\xi)=\int_\R e^{-i\xi t}\phi(t)\,dt,\;\phi\in L_1(\R),\;\xi\in\R$ for the Fourier transform with its usual generalization to $\mathcal S'(\R)$.

%
%

\begin{thm}\label{PaleyWiener}
Let  $f:[0,\infty)\times i\R\to\C$ be continuous on $[0,\infty)\times i\R$ and holomorphic on $(0,\infty)\times i\R$. Assume that there exists a polynomial $P$ such that
\[|f(z)|\leq P(|z|),\quad z\in[0,\infty)+i\R.\]
Then,  the inverse Fourier transform 
\begin{equation*}
\phi:=\cF^{-1}\big(f(i\,\cdot\,)\big)\in \cS'(\R)
\end{equation*}
of the boundary function $f(i\,\cdot\,):\R\to\C,\;\xi\mapsto f(i\xi)$
satisfies 
\[\supp\phi\subset[0,\infty).\] 
Moreover, $(\mathcal L\phi)(z)=f(z)$ for all $z\in(0,\infty)+i\R$.
\end{thm}

The following well-known facts  concerning the Laplace transform will be used in the proof of Theorem~\ref{PaleyWiener} (statements (i), (ii)) and in Subsection~\ref{Inverse transform} (statement (iii)); for proofs see \cite[Sections 5.4, 8.3, 8.5]{Zem87}.
\begin{lem}\label{ZemThm}
Let $\phi,\,\eta\in\mathcal S'(\R)$ such that $\supp\phi,\,\supp\eta\subset [0,\infty)$. The following statements hold:
\begin{itemize}
\item[(i)] $(\mathcal L\phi)(c+i\,\cdot\,)=\mathcal F\big(e^{-c(\cdot)}\phi\big),\quad c>0,$
\item[(ii)] $\big[\cL\big(\frac{d^n}{dt^n}\phi\big)\big](z)=z^n(\mathcal L\phi)(z),\quad z\in (0,\infty)+i\R,\;n\in\Nn.$
\item[(iii)] The convolution $\phi*\eta$ is also an element of $\cS'(\R)$, $\supp(\phi*\eta)\subset[0,\infty)$, and \[\big[\cL(\phi*\eta)\big](z)=(\cL\phi)(z)\cdot(\cL\eta)(z),\quad z\in(0,\infty)+i\R.\]
\end{itemize}
\end{lem}

In the first statement of Theorem \ref{ZemThm}, the expressions $(\mathcal L\phi)(c+i\,\cdot\,)$ and $e^{-c(\cdot)}$ denote the functions $\R\ni\xi\mapsto (\mathcal L\phi)(c+i\xi)\in\C$ and $\R\ni t\mapsto e^{-ct}\in\R$, respectively.
The term $e^{-c(\cdot)}\phi$ is understood as the product of a $C^\infty(\R)$-function and a distribution in $\mathcal D'(\R)$. Using that $\phi$ belongs to $\mathcal S'(\R)$ and that the support of $\phi$ is bounded on the left, it is easy to show that $e^{-c(\cdot)}\phi$ is continuous w.r.t.\ the topology on $\mathcal S(\R)$, i.e.\ $e^{-c(\cdot)}\phi\in\mathcal S'(\R)$. One has $\langle\eta, e^{-c(\cdot)}\phi\rangle_{\mathcal S(\R)\times\mathcal S'(\R)}=\langle\lambda\, e^{-c(\cdot)}\eta,\phi\rangle_{\mathcal S(\R)\times\mathcal S'(\R)}$ for $\eta\in\mathcal S(\R)$, where $\lambda\in C^{\infty}(\R)$ has support bounded on the left and equals one in a neighborhood of $[0,\infty)$.

\begin{proof}[Proof of Theorem \ref{PaleyWiener}]
Due to the polynomial boundedness of $f$, the boundary function $f(i\,\cdot\,)$ belongs to $\cS'(\R)$ and $\phi=\cF^{-1}(f(i\,\cdot\,))$ defines an element in $\cS'(\R)$. Let $N\in\N$ be such that $\tilde f:[0,\infty)+i\R\to\C$ defined by
\[\tilde f(z):=\frac{f(z)}{(1+z)^N},\quad z\in[0,\infty)\times i\R\]
satisfies
\begin{equation}\label{PaleyWienerEq1}
\big|\tilde f(z)\big|\leq\frac C{1+|z|^{2}},\quad z\in[0,\infty)\times i\R.
\end{equation}
Set $\tilde\phi:=\cF^{-1}\big(\tilde f(i\,\cdot\,)\big)$, where $\tilde f(i\,\cdot\,)$ denotes the function $\tilde f(i\,\cdot\,):\R\to\C,\;\xi\mapsto \tilde f(i\xi)$. Then $\phi=\big(\id{\cS'(\R)}+\frac{\dl}{\dl t}\big)^N\tilde\phi$ by the analogue of Lemma~\ref{ZemThm}(ii) for the Fourier transform. We also have $\supp\phi\subset\supp\tilde\phi$ and, if $\supp\tilde\phi\subset[0,\infty)$, 
$(\cL\phi)(z)=(1+z)^N(\cL\tilde\phi)(z)$ for all $z\in(0,\infty)+i\R$ by Lemma~\ref{ZemThm}(ii).
Therefore it suffices to show
\begin{equation}\label{PaleyWienerTilde}
\supp\tilde\phi\subset[0,\infty)\quad\text{ and }\quad \big(\mathcal L\tilde\phi\big)(z)=\tilde f(z),\;z\in(0,\infty)+i\R.
\end{equation}

For $c,b>0$ let $\Gamma_{c,b}$ be the closed, rectangular path of integration with vertices $-ib$, $c-ib$, $c+ib$, $ib$ and counterclockwise orientation. By Cauchy's integral theorem we know 
\begin{equation}\label{oint}
\oint_{\epsilon+\Gamma_{c,b}}e^{tz}\tilde f(z)\dl z=0,\qquad t\in\R,
\end{equation}
for all $\epsilon>0$. Since $\tilde f$ is uniformly continuous on bounded subsets of $[0,\infty)+i\R$, \eqref{oint} remains true for $\epsilon=0$. Considering the limit $b\to\infty$ and using $\eqref{PaleyWienerEq1}$ we obtain
\[\int_\R e^{ti\xi}\tilde f(i\xi)\dl\xi=\int_\R e^{t(c+i\xi)}\tilde f(c+i\xi)\dl\xi,\quad t\in\R,\]
and therefore
\begin{equation}\label{PaleyWienerEq2}
\tilde\phi(t)=\frac{1}{2\pi}\int_\R e^{t(c+i\xi)}\tilde f(c+i\xi)\dl\xi,\quad t\in\R,
\end{equation}
for all $c>0$. Note that $\tilde\phi$ can be assumed to be continuous since $\cF\tilde\phi=\tilde f(i\,\cdot\,)\in L_1(\R)$ due to \eqref{PaleyWienerEq1}.

To prove the first assertion in $\eqref{PaleyWienerTilde}$ we denote further by $\Lambda_{c,b}$ the closed, rectangular path of integration with vertices $c-ib$, $(c+b)-ib$, $(c+b)+ib$, $c+ib$ and counterclockwise orientation,  $c,b>0$. Fix $t<0$ and $c>0$ and observe that
\[0=\oint_{\Lambda_{c,b}}e^{tz}\tilde f(z)\dl z\longrightarrow -\int_\R e^{t(c+i\xi)}\tilde f(c+i\xi)\dl\xi,\quad b\to\infty,\]
by Cauchy's integral theorem and \eqref{PaleyWienerEq1}. Thus $\tilde\phi(t)=0$ for all $t<0$ by \eqref{PaleyWienerEq2}.

The second assertion in $\eqref{PaleyWienerTilde}$ now follows by multiplying both sides of \eqref{PaleyWienerEq2} with $e^{-ct}$, applying the Fourier transform and using Lemma~\ref{ZemThm}(i).
\end{proof}

\subsection{Estimates for the Helmholtz equation}
\label{Estimates for the Helmholtz equation}

The application of the Laplace transform w.r.t.\ the time variable to Eq.~\eqref{SHE} will lead to a Helmholtz equation on $\dom$ with zero Dirichlet boundary condition and stochastic right hand side, cf.\ Lemma~\ref{StoGriEllEqnLem}. In this subsection we consider the same equation with deterministic right hand side. We derive a decomposition of its solution into a regular and a singular part (based on the decomposition of the solution to the Poisson problem) and state estimates from \cite{Gri87}, \cite{Gri95} for the $H^2(\dom)$-norm of the regular part and the coefficient of the singular part. In Subsection~\ref{Decomposition of the transformed equation}, this decomposition and these estimates will be applied to the solution to the random Helmholtz equation resulting from Eq.~\eqref{SHE}.

Given $g\in L_2(\dom)$ and $z\in\C\setminus\sigma(\Lap)$, we are interested in the unique solution $w$ in $H^1_0(\dom)$ to the problem
\begin{equation}\label{Helmholtz}
-\Delta w+zw=g.
\end{equation}
We may rewrite \eqref{Helmholtz} in the form of a Poisson equation,
\[-\Delta w=\tilde g:=\left[\id{L_2(\dom)}-z\big(z\id{L_2(\dom)}-\Lap\big)^{-1}\right]g.\]
By what has been said at the beginning of Section~\ref{Main result}, there exists a unique $c\in\C$, given by $c=\langle \tilde g,v_0\rangle_{L_2(\dom)}$, such that $w-cS$ belongs to $H^2(\dom)\cap H^1_0(\dom)$. In what follows we write $w=w(z)$ and $c=c(z)$ to indicate the dependence on $z$. Thus,
\begin{equation}\label{HelmholtzEq1}
w(z)-c(z)S\in H^2(\dom)\cap H^1_0(\dom).
\end{equation}
Note that we can rewrite $c(z)$ in the form
\begin{equation}\label{c(z)}
c(z)=\big\langle g,\overline{v(z)}\big\rangle_{L_2(\dom)}
\end{equation}
with $v(z)\in L_2(\dom)$ as in $\eqref{v1}$.

For $z\in\C\setminus(-\infty,0)$ we have $S-e^{-r\sqrt z}S\in H^2(\dom)\cap H^1_0(\dom)$, and consequently
\begin{equation}\label{HelmholtzEq2}
w(z)-c(z)e^{-r\sqrt z}S\in H^2(\dom)\cap H^1_0(\dom)
\end{equation}
with the same constant $c(z)\in\C$ as in \eqref{HelmholtzEq1}. Here and below we make slight abuse of notation and write $e^{-r\sqrt z}$ for the function
\[\dom\ni x=(x_1,x_2)=(r\cos\theta,r\sin\theta)\mapsto e^{-r\sqrt z}\in\C.\]
By $\sqrt z$ we mean the complex root of $z\in\C\setminus(-\infty,0)$ whose real part is nonnegative.
It turns out that defining the regular part $w_{\mathrm R}(z)$ of $w(z)$ as the function in \eqref{HelmholtzEq2} instead of the one in \eqref{HelmholtzEq1}, i.e., setting $w_{\mathrm R}(z):=w(z)-c(z)e^{-r\sqrt z}S$, leads to a faster asymptotic decay of the $H^2(\dom)$-norm of $w_{\mathrm R}(z)$ in dependence of $z$. For this reason we define $w_{\mathrm R}(z)$ as above, so that
\begin{equation}\label{GriDecomp}
w(z)=w_{\mathrm R}(z)+c(z)e^{-r\sqrt z}S.
\end{equation}

The following result is taken from \cite[Section 2]{Gri87}; compare also \cite[Section 2.5.2]{Gri92}, \cite[Theorem 5.1]{Gri95}. Recall that we have set $\alpha=\pi/\gamma$, where $\gamma\in(\pi,2\pi)$ is the interior angle of $\partial\dom$ at zero.
\begin{thm}\label{GriEllEst}
Consider the decomposition $\eqref{GriDecomp}$ of the solution $w=w(z)\in H^1_0(\dom)$ to Eq.~$\eqref{Helmholtz}$. Given any angle $\theta_0\in (0,\pi)$, there exists a constant $C>0$, depending only on $\theta_0$, $\dom$ and the cut-off function $\eta$ in \eqref{defS}, such that
\[
\nnrm{w_{\mathrm R}(z)}{H^2(\dom)}+(1+|z|)^{(1-\alpha)/2}|c(z)| \leq C\nnrm{g}{L_2(\dom)}
\]
for all $g\in L_2(\dom)$ and $z\in\C$ with $|\arg z| \leq\theta_0$.
\end{thm}

\begin{rem}\label{remHelmholtz}
{\bf (i)}
The proof of Theorem~\ref{GriEllEst} is based on an alternative representation of the function $v(z)\in L_2(\dom)$ in \eqref{c(z)} that we have defined in \eqref{v1}.
For $t\in(0,\infty)+i\R$, let $\psi(z)\in L_2(\dom)$ be given by
\[
\psi(z)(x):=\psi(z,x):=\eta(x)e^{-r\sqrt z}r^{-\alpha}\sin(\alpha\theta),\quad x=(x_1,x_2)=(r\cos\theta,r\sin\theta)\in\dom,
\]
with $\eta\in C^\infty(\overline\dom;\R)$ as in \eqref{defS}. One can show that $\big(-\Delta +z\id{L_2(\dom)}\big)\psi(z)\in H^{-1}(\dom)$, cf.~\cite[Lemma~2.3]{Gri87}, \cite[Lemma~2.5.4]{Gri92}. Let $\varphi(z)$ be the unique solution in $H^1_0(\dom)$ to the problem $\big(-\Delta +z\id{L_2(\dom)}\big)\varphi(z)=\big(-\Delta +z\id{L_2(\dom)}\big)\psi(z)$. Then, the arguments in \cite[Section~2]{Gri87} and \cite[Section~2.5.2]{Gri92} imply
\[
v(z)=\frac 1\pi(\psi(z)-\varphi(z)).
\]
Note that this representation of $v(z)$ is slightly different from the one derived in \cite[Section~2]{Gri87}; its proof follows along the lines of \cite[Section~2.5.2]{Gri92}.

{\bf (ii)}
In the course of the proof of Theorem~\ref{mainRes} it will be convenient to make use of the fact that $\Delta\big(H^2(\dom)\cap H^1_0(\dom)\big)$ is a closed subspace of $L_2(\dom)$. This follows from \cite[Theorem~2.2.3]{Gri92}. In particular, the orthogonal projection $P_N$ to the orthogonal complement $N:=\big[\Delta\big(H^2(\dom)\cap H^1_0(\dom)\big)\big]^\bot$ in $L_2(\dom)$ is well defined. Using this notation, it is not difficult to see that $N=\C\cdot v(0)=\C\cdot v_0$ and
\[v(0)=v_0=-\frac{P_N(\Delta S)}{\nnrm{\Delta S}{L_2(\dom)}^2}.\]
\end{rem}

\section{Proof of the main result}
\label{Proof of the main result}

In this section we present the proof of Theorem~\ref{mainRes}. We suppose that all assumptions of Sections~\ref{Setting and assumptions} and \ref{Main result} are fulfilled.

\subsection{Laplace transform of the stochastic heat equation}
\label{Laplace transform of the stochastic heat equation}

Let us denote the $\omega$-wise, vector-valued Laplace transform w.r.t.~$t$ of the mild solution $u=(u(t))_{t\in[0,T]}$ to Eq.~\eqref{SHE} by
\begin{equation}\label{U(z)}
U(z):=\int_0^T e^{-zt}u(t)\dl t,\quad z\in\C.
\end{equation}
The integral in \eqref{U(z)} is an $\omega$-wise Bochner integral in $H^1_0(\dom)$.
Recall from Subsection~\ref{The solution process as a tensor product-valued random variable} that we have fixed a modification of $u$ such that all
trajectories $[0,T]\ni t\mapsto u(\omega,t)\in H^1_0(\dom)$ belong to $L_2(0,T;H^1_0(\dom))$ and 
the mapping $\Omega\ni\omega\mapsto u(\omega)\in L_2(0,T;H^1_0(\dom))$ is
$\cF_T/\mathcal B\big(L_2(0,T;H^1_0(\dom))\big)$-measurable. We have
\begin{equation}\label{U(z)inH10}
U(z)\in L_2\big(\Omega,\cF_T,\wP;H^1_0(\dom)\big),\quad z\in\C,
\end{equation}
as a direct consequence of \eqref{uinH10}.
Also recall from Section~\ref{Main result} the definition \eqref{H(z)} of $H(z)$,
\[
H(z)=\int_0^T e^{-zt}F(u(t))\dl t+\int_0^T e^{-zt}G(u(t))\dl W(t)-e^{-zT}u(T)+u_0,\quad z\in \C,
\]
and the assertion \eqref{H(z)inL2},
\[H(z)\in  L_2\big(\Omega,\cF_T,\wP;L_2(\dom)\big),\quad z\in\C.\]

The following lemma describes how the Laplace transform w.r.t.~$t$ turns Eq.~\eqref{SHE} into a random Helmholtz equation.

\begin{lem}\label{StoGriEllEqnLem}
For all $z\in\C$ we have, $\wP$-almost surely,
\begin{equation}\label{StoGriEllEqn2}
\big\langle \nabla U(z),\nabla\varphi\big\rangle_{L_2(\dom;\C^2)}+\big\langle zU(z),\varphi\big\rangle_{L_2(\dom)} = \big\langle H(z) ,\varphi\big\rangle_{L_2(\dom)}\;\text{ for all }\;\varphi\in H^1_0(\dom).
\end{equation}
Thus, for all $z\in\C$ the equality
\[-\Delta U(z)+zU(z)=H(z)\]
holds as an equality in $L_2\big(\Omega,\cF_T,\wP;L_2(\dom)\big)$.
\end{lem}

\begin{proof}
Because of the separability of $H^1_0(\dom)$ it suffices to show that, for all $z\in\C$ and $\varphi\in H^1_0(\dom)$, the equality $\langle \nabla U(z),\nabla\varphi\rangle_{L_2(\dom;\C^2)}+\langle zU(z),\varphi\rangle_{L_2(\dom)} = \langle H(z),\varphi\rangle_{L_2(\dom)}$ holds $\wP$-almost surely.

Since the mild solution $u$ to Eq.~\eqref{SHE} is also a weak
solution (cf.~\cite[Theorem~9.15]{PesZab}), we have, for all $t\in[0,T]$ and all real-valued $\varphi\in D(\Lap)$,
\begin{equation}\label{weakSolution}
\begin{aligned}
\big\langle u(t)-u_0,\varphi\big\rangle_{L_2(\dom)}
&=\int_0^t\left[\big\langle u(s),\Lap\varphi\big\rangle_{L_2(\dom)} +\big\langle F\big(u(s)\big),\varphi\big\rangle_{L_2(\dom)}\right]\dl s\\
&\quad+\int_0^t\big\langle G\big(u(s)\big)[\,\cdot\,],\varphi\big\rangle_{L_2(\dom)}\dl W(s)
\end{aligned}
\end{equation}
$\wP$-almost surely. Here, for fixed $\omega\in\Omega$, $\big\langle G\big(u(\omega,s)\big)[\,\cdot\,],\varphi\big\rangle_{L_2(\dom)}$ denotes the operator in $\hs\big(U_0;\C\big)\cong\hs\big(U_0;\R^2\big)$ defined by  $\big\langle G\big(u(\omega,s)\big)[\,\cdot\,],\varphi\big\rangle_{L_2(\dom)}w:=$$\big\langle G\big(u(\omega,s)\big)w,\varphi\big\rangle_{L_2(\dom)}$ for $w\in U_0$.
Obviously, \eqref{weakSolution} extends to complex-valued test functions $\varphi\in D(\Lap)$.

As a consequence of \eqref{weakSolution}, for all $\varphi\in D(\Lap)$ the process
$\big(\langle u(t),\varphi\rangle_{L_2(\dom)}\big)_{t\in[0,T]}$ has a modification that is a
continuous semimartingale in $\C\cong\R^2$. For this modification, $\eqref{weakSolution}$
holds $\wP$-almost surely for all $t\in[0,T]$ simultaneously. Let us fix a $\varphi$ and this modification and apply
It{\^o}'s formula (see \cite{Met}, \cite{MetPel80}) to the $\C^2$-valued continuous semimartingale
\[\Big(\big(e^{-zt},\langle u(t),\varphi\rangle_{L_2(\dom)}\big)\Big)_{t\in[0,T]}\]
and the function $f:\C^2\to\C,(z_1,z_2)\mapsto f(z_1,z_2)=z_1\cdot z_2$.  We identify
these objects in the usual way with the corresponding $\R^4$-valued semimartingale and the function
$f:\R^4\to\R^2,(x_1,y_1,x_2,y_2)\mapsto f(x_1,y_1,x_2,y_2)=(x_1x_2-y_1y_2,x_1y_2+y_1x_2)$, compare Remark~\ref{RemStoIntCR}. For
clarity we indicate multiplications of two complex numbers by dots in the following calculations. It{\^o}'s formula gives
\begin{equation}\label{StoGriEllEqnLem1}
\begin{aligned}
\big\langle e^{-zT}\cdot u(T)&-u_0,\varphi\big\rangle_{L_2(\dom)}\\
&= \int_0^T\langle u(t),\varphi\rangle_{L_2(\dom)}\cdot \dl e^{-zt} +\int_0^T e^{-zt}\cdot \dl \langle u(t),\varphi\rangle_{L_2(\dom)}\\
&\quad+\frac 12 \tr\int_0^T D^2f\big(e^{-zt},\langle u(t),\varphi\rangle_{L_2(\dom)}\big)\,\dl \big\llbracket \big(e^{-z(\cdot)},\langle u(\cdot),\varphi\rangle_{L_2(\dom)}\big)\big\rrbracket_t,
\end{aligned}
\end{equation}
where the last integral equals zero since the cross-variation of a process with bounded variation and a
continuous process is zero, and $\frac{\partial^2}{\partial x_1\partial y_1}f=\frac{\partial^2}{\partial x_2\partial y_2}f=0$. (Actually, this application of It{\^o}'s formula is nothing but integration by parts, compare \cite[Section 26.9]{Met}.)

A common monotone class argument yields that the $\C\cong\R^2$-valued semimartingale integral $\int_0^T\langle u(t),\varphi\rangle_{L_2(\dom)}\cdot\dl e^{-zt}$ is an $\omega$-wise Lebesgue-Stieltjes integral for almost all $\omega\in\Omega$, so that we obtain
\begin{equation}\label{StoGriEllEqnLem2}
\begin{aligned}
\int_0^T\langle u(t),\varphi\rangle_{L_2(\dom)}\cdot\dl e^{-zt}&=\int_0^T\langle u(t),\varphi\rangle_{L_2(\dom)}\cdot(-z)\cdot e^{-zt}\,\dl t\\
&=-\Big\langle z\cdot\int_0^T e^{-zt}u(t)\,\dl t,\varphi\Big\rangle_{L_2(\dom)},
\end{aligned}
\end{equation}
$\int_0^T e^{-zt}u(t)\,\dl t$ being an $\omega$-wise Bochner integral in $H^1_0(\dom)$. (To be more precise, since we are dealing with complex multiplication, each component of the $\C\cong\R^2$-valued semimartingale integral $\int_0^T\langle u(t),\varphi\rangle_{L_2(\dom)}\cdot\dl e^{-zt}$ is a sum of $\omega$-wise Lebesgue-Stieltjes integrals for almost all $\omega\in\Omega$.)
By \eqref{weakSolution} and a standard rule from stochastic calculus (see \cite[Section 26.4]{Met}) we have
\begin{equation}\label{StoGriEllEqnLem3}
\begin{aligned}
\int_0^T e^{-zt}\cdot\dl\langle u(t),\varphi\rangle_{L_2(\dom)} 
&= \int_0^T e^{-zt}\cdot\left[\big\langle u(t),\Lap\varphi\big\rangle_{L_2(\dom)} +\big\langle F\big(u(t)\big),\varphi\big\rangle_{L_2(\dom)}\right]\dl t\\
&\quad+\int_0^T e^{-zt}\cdot\big\langle G\big(u(t)\big)[\,\cdot\,],\varphi\big\rangle_{L_2(\dom)}\dl W(t).
\end{aligned}
\end{equation}

Due to the construction of the stochastic integral,
\begin{equation}\label{StoGriEllEqnLem4}
\int_0^T e^{-zt}\cdot\big\langle G\big(u(t)\big)[\,\cdot\,],\varphi\big\rangle_{L_2(\dom)}\dl W(t)=\Big\langle \int_0^Te^{-zt}\cdot G\big(u(t)\big)\dl W(t),\varphi\Big\rangle_{L_2(\dom)},
\end{equation}
and due to the construction of the Bochner integral and the fact that $U(z)=\int_0^T e^{-zt}u(t)\,\dl t$ takes values in $H^1_0(\dom)$,
\begin{equation}\label{StoGriEllEqnLem5}
\begin{aligned}
\int_0^T e^{-zt}\cdot&\left[\big\langle u(t),\Lap\varphi\big\rangle_{L_2(\dom)} +\big\langle F\big(u(t)\big),\varphi\big\rangle_{L_2(\dom)}\right]\dl t\\
&=\Big\langle \int_0^Te^{-zt}\cdot u(t)\,dt,\Lap\varphi\Big\rangle_{L_2(\dom)}
+\Big\langle \int_0^Te^{-zt}\cdot F\big(u(t)\big)\,\dl t,\varphi\Big\rangle_{L_2(\dom)}\\
&=-\Big\langle\nabla U(z),\nabla\varphi\Big\rangle_{L_2(\dom;\C^2)}
+\Big\langle \int_0^Te^{-zt}\cdot F\big(u(t)\big)\,\dl t,\varphi\Big\rangle_{L_2(\dom)}.
\end{aligned}
\end{equation}
The combination of \eqref{StoGriEllEqnLem1}, \eqref{StoGriEllEqnLem2}, \eqref{StoGriEllEqnLem3}, \eqref{StoGriEllEqnLem4} and \eqref{StoGriEllEqnLem5} finishes the proof.
\end{proof}

If we find a continuous modification of the $L_2(\dom)$-valued random field $(H(z))_{z\in\C}$, then for this modification Lemma~\ref{StoGriEllEqnLem} will immediately be strengthened: The assertion 
\begin{equation*}
\big\langle \nabla U(z),\nabla\varphi\big\rangle_{L_2(\dom;\C^2)}+\big\langle zU(z),\varphi\big\rangle_{L_2(\dom)} = \big\langle H(z) ,\varphi\big\rangle_{L_2(\dom)}\;\text{ for all }\;\varphi\in H^1_0(\dom),\;z\in\C,
\end{equation*}
will hold $\wP$-almost surely.
In other words: $\wP$-almost surely,
\[-\Delta U(z)+zU(z)=H(z)\;\text{ for all }\;z\in\C.\]
In order to be able to apply Theorem~\ref{PaleyWiener} in a later step of our proof, we are going to show that the $L_2(\dom)$-valued random field $(H(z))_{z\in\C}$ has a holomorphic modification. One way to do this is with the help of the next lemma, which, above all, will play a crucial role in Subsections~\ref{Decomposition of the transformed equation} and \ref{Inverse transform}.

\begin{lem}\label{X(phi)Modification}
Consider the linear and continuous mapping
\begin{equation*}
X:L_2(0,T)\to L_2\big(\Omega,\cF_T,\wP;L_2(\dom)\big),\;\phi\mapsto X(\phi)
\end{equation*}
defined by
\begin{equation}\label{X(phi)}
X(\phi):=\int_0^T\phi(t)F\big(u(t)\big)\dl t+\int_0^T\phi(t)G\big(u(t)\big)\,\dl W(t),\; \phi\in L_2(0,T),
\end{equation}
and let $s>1/2$. There exists an operator-valued random variable
\begin{equation*}
\widetilde X\in L_2\big(\Omega,\cF_T,\wP;\hs(H^s(0,T);L_2(\dom))\big)
\end{equation*}
such that the assertion
\begin{equation}\label{eqX(phi)Modification}
X(\phi)(\omega)=\widetilde X(\omega)(\phi)\;\text{ for $\wP$-almost all }\omega\in\Omega
\end{equation}
holds for all $\phi\in H^s(0,T)$. 

Moreover, we have
\begin{equation}\label{EstTildeX}
\begin{aligned}
\E\gnnrm{\widetilde X}{\hs(H^s(0,T);L_2(\dom))}^2
&\leq\;C\Big(\E\int_0^T\gnnrm{F\big(u(t)\big)}{L_2(\dom)}^2\dl t+\E\sup_{t\in[0,T]}\gnnrm{G\big(\tilde u(t)\big)}{\hs(U_0;L_2(\dom))}^2\Big),
\end{aligned}
\end{equation}
where $C>0$ depends only on $s$ and $T$, and where $\tilde u=(\tilde u(t))_{t\in[0,T]}$ denotes the modification of $u=(u(t))_{t\in[0,T]}$ that is continuous in $L_2(\dom;\R)$, cf.~Theorem~\ref{thmKruLar}.
\end{lem}

\begin{proof}
Let $(\phi_k)_{k\in\N}$ be an orthonormal basis of $ H^s(0,T)$ and let $(\phi_k')_{k\in\N}:=(\langle\cdot,\phi_k\rangle_{ H^s(0,T)})_{k\in\N}$ be the respective dual orthonormal basis of $( H^s(0,T))'$. As described in Appendix~\ref{Tensor products of Sobolev spaces}, we identify $\phi_k'\otimes X(\phi_k)(\omega)\in( H^s(0,T))'\hat\otimes L_2(\dom)$ with the element in $\hs( H^s(0,T);L_2(\dom))$ that maps $\phi\in H^s(0,T)$ to $\langle\phi_k',\phi\rangle_{( H^s(0,T))'\times H^s(0,T)}X(\phi_k)(\omega)
=\phi_k'(\phi)X(\phi_k)(\omega)\in L_2(\dom)$. It is determined $\wP$-almost surely by \eqref{X(phi)}.
For $n,N\in\N$ we have
\begin{equation}\label{modX1}
\begin{aligned}
\E\Bigg\|\sum_{k=n}^N\,&
\phi_k'\otimes X(\phi_k)\Bigg\|_{\hs( H^s(0,T);L_2(\dom))}^2
=\; \sum_{k=n}^N\E\gnnrm{X(\phi_k)}{L_2(\dom)}^2\\
&\leq\;2\Bigg(\E\int_0^T\gnnrm{F\big(u(t)\big)}{L_2(\dom)}^2\dl t+\E\sup_{t\in[0,T]}\gnnrm{G\big(\tilde u(t)\big)}{\hs(U_0;L_2(\dom))}^2\Bigg) \sum_{k=n}^N\nnrm{\phi_k}{L_2(0,T)}^2.
\end{aligned}
\end{equation}
Here we have used H\"{o}lder's inequality and It\^{o}'s isometry. Note that $\E\nnrm{X(\phi_k)}{L_2(\dom)}^2$ and $\E\int_0^T\nnrm{F(u(t))}{L_2(\dom)}^2\dl t$ do not depend on the specific choice of a (measurable) modification of $u$. Moreover, both $\E\int_0^T\nnrm{F(u(t))}{L_2(\dom)}^2\dl t$ and $\E\sup_{t\in[0,T]}\nnrm{G(\tilde u(t))}{\hs(U_0;L_2(\dom))}^2$ are finite due to Assumption~\ref{AssFBu0} and Theorem~\ref{thmKruLar}.

Since for $s>1/2$ the embedding map $ H^s(0,T)\hookrightarrow L_2(0,T)$ is a Hilbert-Schmidt operator --- see Theorem 4.10.2 and Remark 4.10.2/4 in \cite{Tri78} --- the right hand side of \eqref{modX1} tends to zero as $n,N\to\infty$. This means that the limit
\begin{equation*}
\widetilde X:=L_2\big(\Omega,\cF_T,\wP;\hs\big( H^s(0,T);L_2(\dom)\big)\big)
\text{-}\lim_{N\to\infty}\sum_{k=1}^N\phi_k'\otimes X(\phi_k)
\end{equation*}
exists, and \eqref{EstTildeX} holds with a constant $C>0$ that depends only on $s$ and $T$.

Next, note that the evaluation at some $\phi\in H^s(0,T)$ is a continuous mapping from $L_2\big(\Omega,\cF_T,\wP;\hs\big( H^s(0,T);L_2(\dom)\big)\big)$ to $L_2(\Omega,\cF_T,\wP;L_2(\dom))$ and that the mapping $ H^s(0,T)\ni\phi\mapsto X(\phi)\in L_2(\Omega,\cF_T,\wP;L_2(\dom))$ is continuous. We obtain
\begin{equation*}
\begin{aligned}
\widetilde X(\,\cdot\,)(\phi)&=L_2(\Omega,\cF_T,\wP;L_2(\dom))\text{-}\lim_{N\to\infty}\sum_{k=1}^N
\phi_k'(\phi)X(\phi_k)\\
&= X\bigg( H^s(0,T)\text{-}\lim_{N\to\infty}\sum_{k=1}^N\phi_k'(\phi)\phi_k\bigg)\\
&= X\bigg( H^s(0,T)\text{-}\lim_{N\to\infty}\sum_{k=1}^N\langle \phi,\phi_k\rangle_{H^{s}(0,T)}\phi_k\bigg)\\
&= X(\phi)
\end{aligned}
\end{equation*}
for all $\phi\in H^s(0,T)$, where $\widetilde X(\,\cdot\,)(\phi)$ denotes the random variable $\Omega\ni\omega\mapsto\widetilde X(\omega)(\phi)\in L_2(\dom)$.
\end{proof}

\begin{rem}
The restriction of the mapping $X$ in Lemma \ref{X(phi)Modification} to the domain $C^\infty_0((0,T))\subset L_2(0,T)$ is an $L_2(\dom)$-valued, generalized stochastic process on $(0,T)$. It can be interpreted as the distributional time-derivative of the $L_2(\dom)$-valued stochastic process \begin{equation}\label{eqRemX(phi)Modification}
\Big(\int_0^tF\big(u(s)\big)\dl s+ \int_0^tG\big(u(s)\big)\dl W(s)\Big)_{t\in(0,T)}.
\end{equation} 
Since the embedding $H^s_0(0,T)\hookrightarrow H^s(0,T)$ is isometric, hence continuous, the $\omega$-wise restriction of the mapping $\widetilde X$ in Lemma \ref{X(phi)Modification} to the domain $H^s_0(0,T)$ belongs to \linebreak$L_2\big(\Omega,\cF_T,\wP;\hs\big( H^s_0(0,T);L_2(\dom)\big)\big)$. Thus, \eqref{eqX(phi)Modification} asserts that the distributional time-derivative of \eqref{eqRemX(phi)Modification} has a modification that extends to an element in \linebreak$L_2\big(\Omega,\cF_T,\wP;\hs\big( H^s_0(0,T);L_2(\dom)\big)\big)$. Identifying the spaces $\hs\big( H^s_0(0,T);L_2(\dom)\big)$ and \linebreak$H^{-s}((0,T))\hat\otimes L_2(\dom)$ as described in Appendix~\ref{Tensor products of Sobolev spaces}, we have, in a formal sense,
\[F(u)+G(u)\frac{\dl W(t)}{\dl t}\in L_2\left(\Omega,\cF_T,\wP;H^{-s}((0,T))\hat\otimes L_2(\dom)\right).\]
\end{rem}

\begin{rem}
For the sake of clarity we have distinguished notationally between the evaluation $X(\phi)(\omega)$ of the $L_2(\Omega,\cF_T,\wP;L_2(\dom))$-valued operator $X$ and the evaluation $\widetilde X(\omega)(\phi)$ of the operator-valued random variable $\widetilde X$ in the formulation and the proof of  Lemma \ref{X(phi)Modification}.
In what follows it will be more convenient to write $\widetilde X(\omega,\phi)$ instead of $\widetilde X(\omega)(\phi)$ and to denote the mapping $\omega\mapsto\widetilde X(\omega,\phi)$ by $\widetilde X(\,\cdot\,,\phi)$ or $\widetilde X(\phi)$ instead of $\widetilde X(\,\cdot\,)(\phi)$. This is coherent with our notation for evaluations of the solution $u\in L_2(\Omega,\cF_T,\wP;L_2(0,T)\hat\otimes H^1_0(\dom))$ at $\omega\in\Omega$ and test functions $\phi\in L_2(0,T)$, $\varphi\in H^{-1}(\dom)$ introduced in Subsection \ref{The solution process as a tensor product-valued random variable}. We will use similar notation for all other operator-valued random variables to appear below and explain its precise meaning whenever need be.
\end{rem}

Lemma \ref{X(phi)Modification} implies in particular the existence of a holomorphic modification of $(H(z))_{z\in\C}$.

\begin{lem}\label{Gholom}
The $L_2(\dom)$-valued random field
$(H(z))_{z\in\C}$ defined by \eqref{H(z)}
has a holomorphic modification.
\end{lem}

\begin{proof}
For fixed $z,w\in\C$ consider
the function
\begin{equation}\label{holom}
(0,T)\ni t\mapsto\frac{e^{-(z+w)t}-e^{-zt}}{w}\in\C
\end{equation}
By a straightforward calculation, as $w\to0$, this function converges in $C^1((0,T))\hookrightarrow  H^s(0,T)$ to the function
$(0,T)\ni t\mapsto -te^{-zt}\in\C$.
By Lemma \ref{X(phi)Modification} we know that $\big(\widetilde X\big(e^{-z(\cdot)}|_{(0,T)}\big)\big)_{z\in\C}$ is a modification of 
\begin{align*}
\big(X\big(e^{-z(\cdot)}&|_{(0,T)}\big)\big)_{z\in\C}
=\bigg(\int_0^Te^{-zt}F\big(u(t)\big)\dl t+\int_0^Te^{-zt}G\big(u(t)\big)\,\mathrm dW(t)\bigg)_{z\in\C}
\end{align*}
and that the mapping $ H^s(0,T)\ni\phi\mapsto \widetilde X(\omega,\phi)\in L_2(\dom)$ is linear and continuous for all $\omega\in\Omega$. Here $e^{-z(\cdot)}|_{(0,T)}$ denotes the function $(0,T)\ni t\mapsto e^{-zt}\in\C$.
It follows that the mapping $\C\ni z\mapsto \widetilde X(\omega,e^{-z(\cdot)}|_{(0,T)})\in L_2(\dom)$ is holomorphic for all $\omega\in\Omega$. Consequently, a holomorphic modification of $(G(z))_{z\in\C}$ is given by $\big(\widetilde X\big(e^{-z(\cdot)}|_{(0,T)}\big)-e^{-zT}u(T)+u_0\big)_{z\in\C}$ for arbitrary fixed versions of $u_0$, $u(T)\in L_2(\Omega;L_2(\dom))$.
\end{proof}

For $\omega$-wise argumentations concerning the random field $(H(z))_{z\in\C}$ we always refer to a fixed holomorphic modification from now on.

\begin{rem}
Alternatively to using Lemma~\ref{X(phi)Modification}, one can prove Lemma~\ref{Gholom} with the help of It\^o's formula, which implies the equality
\begin{equation}\label{holom2}
\begin{aligned}
\int_0^Te^{-zt}G\big(u(s)\big)\dl W(t)
=e^{-zT}\int_0^TG\big(u(s)\big)\dl W(t)+z\int_0^Te^{-zt}\Big(\int_0^tG\big(u(s)\big)\dl W(s)\Big)\dl t
\end{aligned}
\end{equation}
holding $\wP$-almost surely for every fixed $z\in\C$. Here one takes a continuous version of the $L_2(\dom)$-valued process $\big(\int_0^tG(u(s))\dl W(s)\big)_{t\in[0,T]}$. It can be shown that the right hand side of \eqref{holom2} defines a holomorphic modification of $\big(\int_0^Te^{-zt}G(u(s))\dl W(t)\big)_{z\in\C}$.
\end{rem}

\subsection{Decomposition of the transformed equation}
\label{Decomposition of the transformed equation}

The results of Subsection~\ref{Laplace transform of the stochastic heat equation} imply that there exists $\Omega_0\in\cF_T$ with $\wP(\Omega_0)=1$ such that, for all $\omega\in\Omega_0$ and all $z\in\C$, $U(\omega,z)=\int_0^Te^{-zt}u(\omega,t)\dl t\in D(\Lap)$ satisfies
\begin{equation}\label{stoHelmholtz}
-\Delta U(\omega,z)+zU(\omega,z)=H(\omega,z).
\end{equation}
Here and in the sequel we write $U(\omega,z)$ and $H(\omega,z)$ instead of $U(z)(\omega)$ and $H(z)(\omega).$ 

We apply the results of Subsection~\ref{Estimates for the Helmholtz equation}. For $\omega\in\Omega$ and $z\in\C\setminus\sigma(\Lap)$ define
\begin{equation}\label{c(omega,z)}
\begin{aligned}
c(\omega,z)&=\big\langle H(\omega,z),\overline{v(z)}\big\rangle_{L_2(\dom)}\\
&=\left\langle\left[\id{L_2(\dom)}-z\big(z\id{L_2(\dom)}-\Lap\big)^{-1}\right] H(\omega,z),v_0\right\rangle_{L_2(\dom)},
\end{aligned}
\end{equation}
where $v(z)\in L_2(\dom)$ and $v_0=v(0)\in L_2(\dom;\R)$ are as in \eqref{v1}; compare \eqref{c(z)}. Also, for $\omega\in\Omega$ and $z\in\C\setminus(-\infty,0)$ we set
\begin{equation}\label{UR}
U_{\mathrm R}(\omega,z):=\big(z\id{L_2(\dom)}-\Lap\big)^{-1}H(\omega,z)-c(\omega,z)
e^{-r\sqrt z}S,
\end{equation}
compare \eqref{HelmholtzEq2}. Then $U_{\mathrm R}(\omega,z)\in H^2(\dom)\cap H^1_0(\dom)$ for all $\omega\in\Omega$, and for all $\omega\in\Omega_0$ and $z\in\C\setminus(-\infty,0)$ we have a decomposition of $U(\omega,z)$ of the type \eqref{GriDecomp},
\begin{equation}\label{StoGriDecomp}
U(\omega,z)=U_{\mathrm R}(\omega,z)+c(\omega,z)e^{-r\sqrt{z}}S.
\end{equation}

Let us collect some properties of $c(\omega,z)$ and $U_{\mathrm R}(\omega,z)$.
\begin{lem}\label{cURholom}
{\bf (i)} 
For $\omega\in\Omega$ and $z\in\C\setminus D(\Lap)$, let $c(\omega,z)\in\C$ be defined by \eqref{c(omega,z)}. The mappings
\[c(z):\Omega\to\C,\;\omega\mapsto c(z)(\omega):=c(\omega,z),\qquad z\in\C\setminus D(\Lap),\]
are $\cF_T/\cB(\C)$-measurable. All trajectories
\[c(\omega,\cdot):\C\setminus D(\Lap)\to\C,\;z\mapsto c(\omega,z),\qquad \omega\in\Omega,\]
are holomorphic.

{\bf (ii)}
For $\omega\in\Omega$ and $z\in\C\setminus(-\infty,0)$, let $U_{\mathrm R}(\omega,z)\in H^2(\dom)\cap H^1_0(\dom)$ be defined by \eqref{UR}. The mappings
\[U_{\mathrm R}(z):\Omega\to H^2(\dom),\;\omega\mapsto U_{\mathrm R}(z)(\omega):=U_{\mathrm R}(\omega,z),\qquad z\in\C\setminus(-\infty,0),\]
are $\cF_T/\cB\big(H^2(\dom)\big)$-measurable. All trajectories
\[U_{\mathrm R}(\omega,\cdot):\C\setminus(-\infty,0)\to H^2(\dom),\;z\mapsto U_{\mathrm R}(\omega,z),\qquad \omega\in\Omega,\]
are continuous on $\C\setminus(-\infty,0)$ and holomorphic on $\C\setminus(-\infty,0]$.
\end{lem}

\begin{proof} 
{\bf (i)}
The measurability property is obvious. The holomorphy property follows from the holomorphy of the resolvent map $\C\setminus\sigma(\Lap)\ni z\mapsto\big(z\id{L_2(\dom)}-\Lap\big)^{-1}\in \bo\big(L_2(\dom)\big)$ and the holomorpy of $\C\ni z\mapsto H(\omega,z)\in L_2(\dom)$ for all $\omega\in\Omega$.

{\bf (ii)}
Since $\nnrm{w}{H^2(\dom)}\leq C\nnrm{\Delta w}{L_2(\dom)}$ for all $w\in H^2(\dom)\cap H^1_0(\dom)$ with a constant $C>0$ that does not depend $w$ (see, e.g., \cite[Theorem~2.2.3]{Gri92}), it suffices to verify the assertions with $U_{\mathrm R}(\omega,z)$ replaced by $\Delta U_{\mathrm R}(\omega,z)$ and with $H^2(\dom)$ replaced by $L_2(\dom)$. For all $\omega\in\Omega$ and $z\in\C\setminus(-\infty,0)$,
\begin{align*}
\Delta U_{\mathrm R}(\omega,z)
&=\Delta \big(z\id{L_2(\dom)}-\Lap\big)^{-1}H(\omega,z) - c(\omega,z)\Delta\big(e^{-r\sqrt z}S\big)\\
&=\left[z\big(z\id{L_2(\dom)}-\Lap\big)^{-1}-\id{L_2(\dom)}\right]H(\omega,z)-c(\omega,z)\Delta\big(e^{-r\sqrt z}S\big).
\end{align*}
Now the measurability property is obvious. As in {\bf (i)} one sees that the mapping 
\[\C\setminus\sigma(\Lap)\ni z\mapsto
\left[z\big(z\id{L_2(\dom)}-\Lap\big)^{-1}-\id{L_2(\dom)}\right]H(\omega,z)
\in L_2(\dom)\]
is holomorphic for all $\omega\in\Omega$. A direct calculation and an application of the dominated convergence theorem shows that the mapping
\[\C\setminus(-\infty,0)\ni z\mapsto \Delta\big(e^{-r\sqrt z}S\big)\in L_2(\dom)\]
is continuous on $\C\setminus(-\infty,0)$ and holomorphic on $\C\setminus(-\infty,0]$.
\end{proof}

As a direct consequence of Theorem~\ref{GriEllEst} we have
\begin{lem}\label{StoGriEllEst}
Consider the decomposition $\eqref{StoGriDecomp}$ of $U(\omega,z)=\int_0^Te^{-zt}u(\omega,t)\dl t\in H^1_0(\dom)$, which holds for all $\omega\in\Omega_0$ and all $z\in\C\setminus(-\infty,0)$, where $\wP(\Omega_0)=1$. Given any angle $\theta_0\in (0,\pi)$, there exists a constant $C>0$, depending only on $\theta_0$, $\dom$ and the cut-off function $\eta$ in \eqref{defS}, such that
\[
\gnnrm{U_{\mathrm R}(\omega,z)}{H^2(\dom)}+\big(1+|z|\big)^{(1-\alpha)/2}\big|c(\omega,z)\big| \leq C\gnnrm{H(\omega,z)}{L_2(\dom)}
\]
for all $\omega\in\Omega_0$ and $z\in\C$ with $|\arg z| \leq\theta_0$.
In particular,
\[
\gnnrm{U_{\mathrm R}(z)}{L_2(\Omega;H^2(\dom))}+\big(1+|z|\big)^{(1-\alpha)/2}\gnnrm{c(z)}{L_2(\Omega;\C)}
\leq C\gnnrm{H(z)}{L_2(\Omega;L_2(\dom))}
\]
for all $z\in\C$ with $|\arg z| \leq\theta_0$.
\end{lem}

In combination with Lemma~\ref{StoGriEllEst}, the following result will enable us to derive assertions concerning the Sobolev regularity and the supports of the inverse Fourier transforms of $\xi\mapsto U_{\mathrm R}(\omega,i\xi)$ and $\xi\mapsto c(\omega,i\xi)$. 

\begin{lem}\label{HIntAbschLem}
Let $s>1/2$ and let 
\[\widetilde X\in L_2\big(\Omega,\cF_T,\wP;\hs(H^s(0,T);L_2(\dom))\big)\] 
be as in Lemma~\ref{X(phi)Modification}. There exists a constant $C>0$, depending only on $s$, such that, $\wP$-almost surely,
\begin{equation}\label{HIntAbschLemEq1}
\begin{aligned}
\int_\R\big(1+\xi^2\big)^{-s}&\gnnrm{H(i\xi)}{L_2(\dom)}^2\dl\xi\\
&\leq C\left(
\gnnrm{\widetilde X}{\hs(H^s(0,T);L_2(\dom))}^2
+\gnnrm{u(T)}{L_2(\dom)}^2+\gnnrm{u_0}{L_2(\dom)}^2
\right),
\end{aligned}
\end{equation}
and
\begin{equation}\label{HIntAbschLemEq2}
\begin{aligned}
\gnnrm{H(z)}{L_2(\dom)}
&\;\leq \;
e^{|\Re z|T\one_{(-\infty,0]}(\Re z)}
\gnnrm{\widetilde X}{\hs(H^s(0,T);L_2(\dom))}(1+|z|)\\
&\qquad+e^{|\Re z|T\one_{(-\infty,0]}(\Re z)}\gnnrm{u(T)}{L_2(\dom)}+\gnnrm{u_0}{L_2(\dom)},\qquad z\in\C.
\end{aligned}
\end{equation}
\end{lem}

\begin{proof}
Fix a version of the random variable $\widetilde X\in L_2(\Omega;\hs( H^s(0,T);L_2(\dom)))$ from Lemma~\ref{X(phi)Modification}
and observe that, for $\wP$-almost all $\omega\in\Omega$, we have
\begin{equation*}
H(\omega,z)=\widetilde X\big(\omega,e^{-z(\cdot)}|_{(0,T)}\big)-e^{-zT}u(\omega,T)+u_0(\omega)\quad\text{ for all }z\in\C,
\end{equation*}
where $e^{-z(\cdot)}|_{(0,T)}$ denotes the function $(0,T)\ni t\mapsto e^{-zt}\in\C$. (Recall that we consider a fixed holomorphic, hence continuous, modification of the $L_2(\dom)$-valued random field $(H(z))_{z\in\C}$.)

Assertion \eqref{HIntAbschLemEq2} follows from
\begin{align*}
\gnnrm{H(z)}{L_2(\dom)}&\leq \gnnrm{\widetilde X\big(e^{-z(\cdot)}|_{(0,T)}\big)}{L_2(\dom)}+\gnnrm{u_0-e^{-zT}u(T)}{L_2(\dom)}\\
&\leq\gnnrm{\widetilde X}{\hs(H^s(0,T);L_2(\dom))}\gnnrm{e^{-z(\cdot)}|_{(0,T)}}{H^s(0,T)}+\gnnrm{u_0-e^{-zT}u(T)}{L_2(\dom)}
\end{align*}
together with
\[\gnnrm{e^{-z(\cdot)}|_{(0,T)}}{H^s(0,T)}\leq T\gnnrm{e^{-z(\cdot)}|_{(0,T)}}{C^1((0,T))}\leq e^{|\Re z|T\one_{(-\infty,0]}(\Re z)}
(1+|z|)\]
and
\[\gnnrm{u_0-e^{-zT}u(T)}{L_2(\dom)}
\leq \gnnrm{u_0}{L_2(\dom)}+ e^{|\Re z|T\one_{(-\infty,0]}(\Re z)}\gnnrm{u(T)}{L_2(\dom)}.\]

In order to prove \eqref{HIntAbschLemEq1} we write
\begin{equation}\label{GIntAbsch0}
\begin{aligned}
\int_\R\big(1+\xi^2\big)^{-s}\gnnrm{H(i\xi)}{L_2(\dom)}^2\dl\xi
&\leq 2\int_\R\big(1+\xi^2\big)^{-s}\left\|\widetilde X\big(e^{-i\xi(\cdot)}|_{(0,T)}\big)\right\|_{L_2(\dom)}^2\dl\xi\\
&\quad+2\int_\R\big(1+\xi^2\big)^{-s}\gnnrm{u_0-e^{-i\xi T}u(T)}{L_2(\dom)}^2\dl\xi
\end{aligned}
\end{equation}
and estimate each term separately. 
Clearly,
\begin{equation}\label{GIntAbsch4}
\int_\R\big(1+\xi^2\big)^{-s}\gnnrm{u_0-e^{-i\xi T}u(T)}{L_2(\dom)}^2\dl\xi
\leq C\left(\|u_0\|_{L_2(\dom)}+\|u(T)\|_{L_2(\dom)}\right)^2
\end{equation}
where $C=2\int_\R\big(1+\xi^2\big)^{-s}\dl\xi$ is finite since $s>1/2$.

To estimate the first term on the right hand side of the inequality~\eqref{GIntAbsch0}, we extend $\widetilde X\in L_2(\Omega;\hs( H^s(0,T);L_2(\dom))$ to an $\hs(H^s(\R);L_2(\dom))$-valued random-variable by setting
\[\widetilde X_{\text{ext}}(\omega,\phi):=\widetilde X(\omega,\phi|_{(0,T)}),\qquad \omega\in\Omega,\;\phi\in H^s(\R).\]
This definition makes sense since $\|\phi|_{(0,T)}\|_{ H^s(0,T)}\leq \|\phi\|_{H^s(\R)}$ according to Definition~\ref{DefSobolev}, and we have
\begin{equation}\label{WidetildeXExtHS}
\gnnrm{\widetilde X_{\text{ext}}(\omega)}{\hs(H^s(\R);L_2(\dom))}
\leq \gnnrm{\widetilde X(\omega)}{\hs(H^s(0,T);L_2(\dom))},\quad \omega\in\Omega.
\end{equation}
Moreover, for all $\omega\in\Omega$ and $\xi\in\R$,
\begin{equation}\label{WidetildeXExt=}
\widetilde X_{\text{ext}}\big(\omega,\lambda(\cdot) e^{-i\xi(\cdot)}\big)=\widetilde X\big(\omega,e^{-i\xi(\cdot)}|_{(0,T)}\big).
\end{equation}
Here and below $\lambda(\cdot)$ is a $C^\infty_0(\R)$-function which equals one in a neighborhood of $[0,T]$.

Next, let $(\varphi_k)_{k\in\N}$ be an orthonormal basis of $L_2(\dom)$ and, for $k\in\N$, define $\widetilde X_k\in L_2(\Omega,\cF_T,\wP;H^{-s}(\R))$ by
\begin{equation*}
\big\langle\widetilde X_k(\omega),\phi\big\rangle_{ H^{-s}(\R)\times H^{s}(\R)}:=\big\langle\widetilde X_{\text{ext}}(\omega,\phi),\varphi_k\big\rangle_{L_2(\dom)},\quad\omega\in\Omega,\;\phi\in H^s(\R).
\end{equation*}
With the natural embedding of $H^{-s}(\R)$ into $\cS'(\R)$ the identity
$\big\langle\widetilde X_k(\omega),\phi\big\rangle_{H^{-s}(\R)\times H^{s}(\R)}
=\big\langle\widetilde X_k(\omega),\phi\big\rangle_{\cS'(\R)\times\cS(\R)}$ holds for $\phi\in\mathcal S(\R)$. Also, for all $\omega\in\Omega$ and $k\in\N$, we have 
\begin{equation}\label{FTwidetildeXk}
\big[\cF_{t\to\xi}\big(\widetilde X_k(\omega)\big)\big](\xi)=
\big\langle\widetilde X_k(\omega),\lambda(\cdot)e^{-i\xi(\cdot)}\big\rangle_{\cS'(\R)\times\cS(\R)}\;\text{ for }\lambda\text{-almost all }\xi\in\R,
\end{equation} 
see, e.g., \cite[Theorem~7.4-3]{Zem87}.
Using \eqref{WidetildeXExt=}, Parseval's identity, \eqref{FTwidetildeXk} and the norm equivalence mentioned subsequent to \eqref{HsFT}, we obtain
\begin{equation}\label{GIntAbsch1}
\begin{aligned}
\int_\R\big(1+\xi^2\big)^{-s}\Big\|\widetilde X\big(e^{-i\xi(\cdot)}&|_{(0,T)}\big)\Big\|_{L_2(\dom)}^2\dl\xi\\
&=\int_\R\big(1+\xi^2\big)^{-s}\left\|\widetilde X_{\text{ext}}\left(\lambda(\cdot)e^{-i\xi(\cdot)}\right)\right\|_{L_2(\dom)}^2\dl\xi\\
&=\sum_{k\in\N}\int_\R\big(1+\xi^2\big)^{-s}\left|\left\langle\widetilde X_k,\lambda(\cdot)e^{-i\xi(\cdot)}\right\rangle_{\cS'(\R)\times\cS(\R)}\right|^2\dl\xi\\
&=\sum_{k\in\N}\int_\R\big(1+\xi^2\big)^{-s}\left|\big[\mathcal F_{t\to\xi}\widetilde X_k\big](\xi)\right|^2\dl\xi\\
&\leq C \sum_{k\in\N}\big\|\widetilde X_k\big\|^2_{H^{-s}(\R)}.
\end{aligned}
\end{equation}

For fixed $\omega\in\Omega$, let $\widetilde X_{\text{ext}}^*(\omega)=\widetilde X_{\text{ext}}^*(\omega,\,\cdot\,)\in\bo(L_2(\dom);H^s(\R))$ be the adjoint operator of $\widetilde X_{\text{ext}}(\omega)=\widetilde X_{\text{ext}}(\omega,\,\cdot\,)\in\hs(H^s(\R);L_2(\dom))$ in the Hilbert space sense. Then
\begin{align*}
\big\|\widetilde X_k(\omega)\big\|_{H^{-s}(\R)}
=\sup_{\phi\in H^s(\R)}\big\langle\widetilde X_{\text{ext}}(\omega,\phi),\varphi_k\big\rangle_{L_2(\dom)}
&=\sup_{\phi\in H^s(\R)}\big\langle\phi,\widetilde X^*_{\text{ext}}(\omega,\varphi_k)\big\rangle_{H^s(\R)}\\
&=\big\|\widetilde X^*_{\text{ext}}(\omega,\varphi_k)\big\|_{H^{s}(\R)}
\end{align*}
and
\[\big\|\widetilde X^*_{\text{ext}}(\omega)\big\|_{\hs(L_2(\dom);H^s(\R))}=\big\|\widetilde X_{\text{ext}}(\omega)\big\|_{\hs(H^s(\R);L_2(\dom))}.\]
Together with \eqref{WidetildeXExtHS} and \eqref{GIntAbsch1} this yields
\begin{equation}\label{GIntAbsch3}
\int_\R\big(1+\xi^2\big)^{-s}\Big\|\widetilde X\big(\omega,e^{-i\xi(\cdot)}|_{(0,T)}\big)\Big\|_{L_2(\dom)}^2\dl\xi\leq C \gnnrm{\widetilde X(\omega)}{\hs(H^s(0,T);L_2(\dom))}^2,\;\omega\in\Omega.
\end{equation}

The combination of \eqref{GIntAbsch0}, \eqref{GIntAbsch4} and \eqref{GIntAbsch3} yields \eqref{HIntAbschLemEq1}.
\end{proof}

\subsection{Inverse transform}
\label{Inverse transform}

We are now ready to invert the vector-valued Laplace transform of $u=(u(t))_{t\in[0,T]}$
in terms of the decomposition \eqref{StoGriDecomp},
\begin{equation*}
\begin{aligned}
U(\omega,z)= U_{\mathrm R}(\omega,z) + c(\omega,z)e^{-r\sqrt z}S,\quad z\in\C\setminus(-\infty,0),
\end{aligned}
\end{equation*}
which holds for $\wP$-almost all $\omega\in\Omega$.
It will be convenient to introduce the notation
\begin{equation}\label{M}
M(\omega):=\gnnrm{\widetilde X(\omega)}{\hs(H^s(0,T);L_2(\dom))}^2
+\gnnrm{u(\omega,T)}{L_2(\dom)}^2+\gnnrm{u_0(\omega)}{L_2(\dom)}^2,\quad\omega\in\Omega,
\end{equation}
so that $M\in L_1(\Omega,\cF_T,\wP)$.
Here we consider again an arbitrary fixed version of the random variable $\widetilde X\in L_2(\Omega;\hs( H^s(0,T);L_2(\dom)))$ introduced in Lemma~\ref{X(phi)Modification}.

\subsubsection{Inverse transform of $U_R$}
\label{Inverse transform of U_R}

Lemma~\ref{StoGriEllEst} and Lemma~\ref{HIntAbschLem} imply that, for $\wP$-almost all $\omega\in\Omega$,
\begin{equation}\label{InverseUR1}
\int_\R\big(1+\xi^2\big)^{-s}\gnnrm{U_{\mathrm R}(\omega,i\xi)}{H^2(\dom)}^2\dl\xi\\
\leq CM(\omega)
\end{equation}
and
\begin{equation}\label{InverseUR2}
\gnnrm{U_{\mathrm R}(\omega,z)}{H^2(\dom)}
\leq \sqrt{3M(\omega)}(1+|z|),\quad z\in[0,\infty)+i\R,
\end{equation}
where the constant $C>0$ depends only on $s$, $\dom$ and the cut-off function $\eta$ in \eqref{defS}. For convenience, let us redefine $U_{\mathrm R}(\omega,z):=0\in H^2(\dom)$ for $z\in[0,\infty)+i\R$ and all $\omega\in\Omega$ such that \eqref{InverseUR1} and \eqref{InverseUR2} does not hold.

For $\varphi\in (H^2(\dom))'$ and $\omega\in\Omega$ we define
\begin{equation}\label{DefuR1}
u_{+,\mathrm R}^\varphi(\omega):=\mathcal F^{-1}_{\xi\to t}\big\langle U_{\mathrm R}(\omega,i\,\cdot\,),\varphi\big\rangle_{H^2(\dom)\times (H^2(\dom))'},
\end{equation}
where $\big\langle U_{\mathrm R}(\omega,i\,\cdot\,),\varphi\big\rangle_{H^2(\dom)\times (H^2(\dom))'}$ denotes the function \[\R\ni\xi\mapsto \big\langle U_{\mathrm R}(\omega,i\xi),\varphi\big\rangle_{H^2(\dom)\times (H^2(\dom))'}\in\C.\]
By Theorem~\ref{PaleyWiener}, Lemma~\ref{cURholom}, \eqref{InverseUR1}, \eqref{InverseUR2} and \eqref{HsFT}, we have
$u_{+,\mathrm R}^\varphi(\omega)\in H^{-s}(\R)$ and $\supp u_{+,\mathrm R}^\varphi(\omega)\subset [0,\infty)$. Moreover, for all $\omega\in\Omega$ and all $\varphi\in (H^2(\dom))'$,
\begin{equation*}
\begin{aligned}
\gnnrm{u_{+,\mathrm R}^\varphi(\omega)}{H^{-s}(\R)}^2
&\leq C\int_\R\big(1+\xi^2\big)^{-s}\big|\big\langle U_{\mathrm R}(\omega,i\xi),\varphi\big\rangle_{H^2(\dom)\times (H^2(\dom))'}\big|^2\dl\xi\\
&\leq CM(\omega)
\nnrm{\varphi}{(H^2(\dom))'}^2,
\end{aligned}
\end{equation*}
where $C>0$ depends only on $s$, $\dom$ and $\eta$.

For all $\omega\in\Omega$ the linear and bounded mapping $u_{+,\mathrm R}^{(\,\cdot\,)}(\omega):(H^2(\dom))'\to H^{-s}(\R),\;\varphi\mapsto u_{+,\mathrm R}^\varphi(\omega)$ is a Hilbert-Schmidt operator: If $(\varphi_k)_{k\in\N}$ is an orthonormal basis of $(H^2(\dom))'$, then
\begin{equation}\label{InverseUR3}
\begin{aligned}
\gnnrm{u_{+,\mathrm R}^{(\,\cdot\,)}(\omega)}{\hs((H^2(\dom))';H^{-s}(\R))}^2
&=\sum_{k\in\N}\gnnrm{u_{+,\mathrm R}^{\varphi_k}(\omega)}{H^{-s}(\R)}^2\\
&\leq C\sum_{k\in\N}\int_\R\big(1+\xi^2\big)^{-s}\big|\big\langle U_{\mathrm R}(\omega,i\xi),\varphi_k\big\rangle_{H^2(\dom)\times (H^2(\dom))'}\big|^2\dl\xi\\
&= C \int_\R\big(1+\xi^2\big)^{-s}\gnnrm{U_{\mathrm R}(\omega,i\xi)}{H^2(\dom)}^2\dl\xi\\
&\leq CM(\omega).
\end{aligned}
\end{equation}
We define
\begin{equation}\label{DefuR2}
\begin{aligned}
u_{+,\mathrm R}(\omega)(\phi,\varphi):=u_{+,\mathrm R}(&\omega,\phi,\varphi)
:=\big\langle u_{+,\mathrm R}^\varphi(\omega),\phi\big\rangle_{H^{-s}(\R)\times H^s(\R)},\\
&\,\omega\in\Omega,\;\phi\in H^s(\R),\;\varphi\in(H^2(\dom))',
\end{aligned}
\end{equation} 
i.e., $u_{+,\mathrm R}(\omega)=L^{-1}\big(u_{+,\mathrm R}^{(\,\cdot\,)}(\omega)\big)$, where $L$ is the canonical isomorphism from $H^{-s}(\R)\hat\otimes H^2(\dom)$ to $\hs((H^2(\dom))';H^{-s}(\R))$ described in Appendix~\ref{Tensor products of Sobolev spaces}; compare~\eqref{HSFasHSO}. Thus, for all $\omega\in\Omega$,
\begin{equation}\label{EstuR1}
\gnnrm{u_{+,\mathrm R}(\omega)}{H^{-s}(\R)\hat\otimes H^2(\dom)}^2
\leq C\Big(
\gnnrm{\widetilde X(\omega)}{\hs(H^s(0,T);L_2(\dom))}^2
+\gnnrm{u(\omega,T)}{L_2(\dom)}^2+\gnnrm{u_0(\omega)}{L_2(\dom)}^2
\Big)
\end{equation}
with a constant $C>0$ that depends only on $s$, $\dom$ and the cut-off function $\eta$ in \eqref{defS}.

Let us check that $u_{+,\mathrm R}(\omega)$ belongs to $H^{-s}(\R)\hat\otimes H^1_0(\dom)$ for all $\omega\in\Omega$. For $\omega\in\Omega$, $\phi\in H^s(\R)$ and $\varphi\in H^{-1}(\dom)$, set
\begin{equation*}
\tilde u_{+,\mathrm R}^\varphi(\omega):=\mathcal F^{-1}_{\xi\to t}\big\langle U_{\mathrm R}(\omega,i\,\cdot\,),\varphi\big\rangle_{H^1_0(\dom)\times H^{-1}(\dom)}
\end{equation*}
and
\begin{equation*}
\tilde u_{+,\mathrm R}(\omega)(\phi,\varphi):=\tilde u_{+,\mathrm R}(\omega,\phi,\varphi)
:=\big\langle \tilde u_{+,\mathrm R}^\varphi(\omega),\phi\big\rangle_{H^{-s}(\R)\times H^s(\R)}.
\end{equation*} 
In analogy to the argument above, we obtain
\begin{equation*}
\gnnrm{\tilde u_{+,\mathrm R}(\omega)}{H^{-s}(\R)\hat\otimes H^1_0(\dom)}^2
\leq CM(\omega),\quad \omega\in\Omega.
\end{equation*}
Now let $\j$ be the natural embedding of $H^1_0(\dom)$ into $H^2(\dom)$ and let $\id{H^{-s}(\R)}\hat\otimes\,\j$ be the corresponding embedding of $H^{-s}(\R)\hat\otimes H^1_0(\dom)$ into $H^{-s}(\R)\hat\otimes H^2(\dom)$. Then, the identity $\id{H^{-s}(\R)}\hat\otimes\,\j\,\big(\tilde u_{+,\mathrm R}(\omega)\big)= u_{+,\mathrm R}(\omega)$ holds for all $\omega\in\Omega$, i.e., $u_{+,\mathrm R}(\omega)\in H^{-s}(\R)\hat\otimes H^1_0(\dom)$. Indeed, by Proposition~\ref{TPemb}, 
\[\id{H^{-s}(\R)}\hat\otimes\,\j\,\big(\tilde u_{+,\mathrm R}(\omega)\big)(\phi,\varphi)=\tilde u_{+,\mathrm R}\big(\omega,\phi,\j'(\varphi)\big)=\big\langle \tilde u_{+,\mathrm R}^{\j'(\varphi)}(\omega),\phi\big\rangle_{H^{-s}(\R)\times H^s(\R)}\]
for $\omega\in\Omega$, $\phi\in H^{-s}(\R)$ and $\varphi\in (H^2(\dom))'$, and
\begin{align*}
\tilde u_{+,\mathrm R}^{\j'(\varphi)}(\omega)=\mathcal F^{-1}_{\xi\to t}\big\langle U_{\mathrm R}(\omega,i\,\cdot\,),\j'(\varphi)\big\rangle_{H^1_0(\dom)\times H^{-1}(\dom)}&=\mathcal F^{-1}_{\xi\to t}\big\langle \j(U_{\mathrm R}(\omega,i\,\cdot\,)),\varphi\big\rangle_{H^2(\dom)\times (H^2(\dom))'}\\
&=u_{+,\mathrm R}^\varphi(\omega).
\end{align*}

In order to verify the  $\cF_T/\cB\big(H^{-s}(\R)\hat\otimes H^2(\dom)\big)$-measurability of $u_{+,\mathrm R}$, we note that the continuity of the mapping $\R\ni\xi\mapsto U_{\mathrm R}(\omega,i\xi)\in H^2(\dom)$ for all $\omega\in\Omega$ implies the $\cF_T/\cB\big(L_2\big(\R,(1+\xi^2)^{-s}\dl\xi;H^2(\dom)\big)\big)$-measurability of 
\[\Omega\ni\omega\mapsto U_{\mathrm R}(\omega,i\,\cdot\,)\in L_2\big(\R,(1+\xi^2)^{-s}\dl\xi;H^2(\dom)\big).\]
Now the measurability of $u_{\mathrm R}$ follows from the continuity of the inverse Fourier transform $\cF_{\xi\to t}^{-1}:L_2(\R,(1+\xi^2)^{-s}\dl\xi;\C)\to H^{-s}(\R)$ and the fact that
\[u_{+,\mathrm R}(\omega)=H^{-s}(\R)\hat\otimes H^2(\dom)\text{-}\lim_{N\to\infty}\sum_{j,k=1}^N\Big\langle u_{+,\mathrm R}^{\varphi_k'}(\omega),\phi_j'\Big\rangle_{H^{-s}(\R)\times H^s(\R)}\phi_j\otimes\varphi_k\]
for all $\omega\in\Omega$, where $(\phi_j)_{j\in\N}$ and $(\varphi_k)_{k\in\N}$ are orthonormal bases of $H^{-s}(\R)$ and $H^2(\dom)$ and where $(\phi_j')_{j\in\N}$ and $(\varphi_k')_{k\in\N}$ are the respective dual orthonormal bases of  $H^{s}(\R)$ and $(H^2(\dom))'$.

\subsubsection{Inverse transform of $c\,e^{-r\sqrt z}S$}
\label{Inverse transform of c e^{-r sqrt z}S}

By Lemma~\ref{StoGriEllEst} and Lemma~\ref{HIntAbschLem} we have, for $\wP$-almost all $\omega\in\Omega$,
\begin{equation}\label{Inversec1}
\int_\R\big(1+\xi^2\big)^{(1-\alpha)/2-s}|c(\omega,i\xi)|^2\dl\xi\leq CM(\omega)
\end{equation}
and
\begin{equation}\label{Inversec2}
|c(\omega,z)|\leq \sqrt{3M(\omega)}(1+|z|)^{(1+\alpha)/2},\quad z\in[0,\infty)+i\R,
\end{equation}
where the constant $C>0$ depends only on $s$, $\dom$ and the cut-off function $\eta$ in \eqref{defS}, and where $M(\omega)$ is defined by \eqref{M}. We redefine $c(\omega,z):=0$ for $z\in[0,\infty)+i\R$ and all $\omega\in\Omega$ such that \eqref{Inversec1} and \eqref{Inversec2} does not hold.

We set
\begin{equation}\label{DefPhi}
\Phi(\omega):=\cF_{\xi\to t}^{-1} \big(c(\omega,i\,\cdot\,)\big),\quad  \omega\in\Omega,
\end{equation}
where $c(\omega,i\,\cdot\,)$ denotes the function $\R\ni\xi\to c(\omega,i\xi)\in\C$.
Theorem~\ref{PaleyWiener}, Lemma~\ref{cURholom}, \eqref{Inversec1}, \eqref{Inversec2} and \eqref{HsFT} imply
$\Phi(\omega)\in H^{(1-\alpha)/2-s}(\R)$, 
\begin{equation}\label{suppPhi}
\supp\Phi(\omega)\subset [0,\infty),
\end{equation} 
and
\begin{equation}\label{Inversec3}
\nnrm{\Phi(\omega)}{H^{(1-\alpha)/2-s}(\R)}^2\leq C\Big(
\gnnrm{\widetilde X(\omega)}{\hs(H^s(0,T);L_2(\dom))}^2
+\gnnrm{u(\omega,T)}{L_2(\dom)}^2+\gnnrm{u_0(\omega)}{L_2(\dom)}^2
\Big)
\end{equation}
for all $\omega\in\Omega$,  where $C>0$ depends only on $s$, $\dom$ and $\eta$.
Moreover, 
\begin{equation}\label{Inversec4}
\big[\cL\big(\Phi(\omega)\big)\big](z)=c(\omega,z),\qquad \omega\in\Omega,\;z\in(0,\infty)+i\R.
\end{equation}

Next, fix $\varphi\in L_2(\dom)$ and consider the function
$
\R\ni t\mapsto
\int_\dom E_0(t,x)S(x)\varphi(x)\dl x\in\C,
$
where $E_0$ is defined by \eqref{defE0};
we denote it by $\int_\dom E_0(\cdot,x)S(x)\varphi(x)\dl x$.
A direct calculation gives
\begin{align*}
\Big|\int_\dom E_0(t,x)S(x)\varphi(x)\dl x\Big|
&\leq\Big(\int_\dom |E_0(t,x)S(x)|^2\dl x\Big)^{1/2}\nnrm{\varphi}{L_2(\dom)}\\
&\leq C\left(t^{-1/2}+e^{-C/(4t)}(t^{-1}+t^{-1/2})\right)\nnrm{\varphi}{L_2(\dom)},
\end{align*}
where $C>0$ does not depend on $t$. In particular, the right sided function $\int_\dom E_0(\cdot,x)S(x)\varphi(x)\dl x$ belongs to $\cS'(\R)$. For $z\in(0,\infty)+i\R$, its Laplace transform is
\begin{equation}\label{InverseUS1}
\begin{aligned}
\int_0^\infty e^{-zt}\Big[\int_\dom E_0(t,x)S(x)\varphi(x)\dl x\Big]\dl t
&=\int_\dom \Big[\int_0^\infty e^{-zt}E_0(t,x)\dl t\Big]S(x)\varphi(x)\dl x\\
&=\int_\dom e^{-r\sqrt{z}}S(x)\varphi(x)\dl x,
\end{aligned}
\end{equation}
where, as in Section~\ref{Estimates for the Helmholtz equation}, $e^{-r\sqrt z}$ denotes the function $\dom\ni x=(r\cos\theta,r\sin\theta)\mapsto e^{-r\sqrt z}\in\C$.
The second step in \eqref{InverseUS1} is due to the identity \[\frac r{2\sqrt\pi}\int_0^\infty e^{-zt}t^{-3/2}e^{-r^2/(4t)}\dl t=e^{-r\sqrt z},\qquad r>0,\; z\in(0,\infty)+i\R,\]
a proof of which can be found in \cite[Section 8.4]{Foe93}, compare also \cite[Exercise 3A/3]{Gue91}. The application of Fubini's theorem in the first step in \eqref{InverseUS1}
is possible since
\begin{align*}
\int_\dom\int_0^\infty\big|e^{-zt}E_0(t,x)S(x)\varphi(x)\big|\dl t\,\dl x
&= \int_\dom\Big[\int_0^\infty e^{-(\Re z)t}E_0(t,x)\dl t\Big]S(x)|\varphi(x)|\dl x\\
&=\int_\dom e^{-r\sqrt{\Re z}}S(x)|\varphi(x)|\dl x <\infty.
\end{align*}

By Lemma~\ref{ZemThm}(iii), \eqref{Inversec4} and \eqref{InverseUS1} we obtain
\begin{equation*}
\Big[\cL\Big(\Phi(\omega)*\int_\dom E_0(\cdot,x)S(x)\varphi(x)\dl x\Big)\Big](z)= c(\omega,z)\int_\dom e^{-r\sqrt{z}}S(x)\varphi(x)\dl x
\end{equation*}
for all $\omega\in\Omega$ and $z\in(0,\infty)+i\R$. This, together with Theorem~\ref{PaleyWiener} and the uniqueness of the Fourier and the Laplace transform, implies the equality in $\cS'(\R)$
\begin{equation}
\cF\Big(\Phi(\omega)*\int_\dom E_0(\cdot,x)S(x)\varphi(x)\dl x\Big)= \int_\dom c(\omega,i\,\cdot\,)e^{-r\sqrt{i(\,\cdot\,)}}S(x)\varphi(x)\dl x,\quad\omega\in\Omega,
\end{equation}
where the right hand side denotes the function $\R\ni\xi\mapsto \int_\dom c(\omega,i\xi) e^{-r\sqrt{i\xi}}S(x)\varphi(x)\dl x\in\C$.
Thus, for $\varphi\in L_2(\dom)=(L_2(\dom))'$ and $\omega\in\Omega$ we may set, in analogy to \eqref{DefuR1},
\begin{equation}\label{DefuS1}
\begin{aligned}
u_{+,\mathrm S}^\varphi(\omega)&:=\mathcal F^{-1}_{\xi\to t}\Big\langle c(\omega,i\,\cdot\,)e^{-r\sqrt{i(\,\cdot\,)}}S,\varphi\Big\rangle_{L_2(\dom)\times (L_2(\dom))'}\\
&= \Phi(\omega)*\int_\dom E_0(\cdot,x)S(x)\varphi(x)\dl x.
\end{aligned}
\end{equation}

Similar to the argument for $u_{+,\mathrm R}^{\varphi}(\omega)$ above, using \eqref{Inversec1} and $\nnrm{e^{-r\sqrt{i\xi}}S}{L_2(\dom)}\leq\nnrm{S}{L_2(\dom)}$, one sees that for all $\omega\in\Omega$ the mapping $\varphi\mapsto u_{+,\mathrm S}^\varphi(\omega)$ is a Hilbert-Schmidt operator from $L_2(\dom)$ to $H^{(1-\alpha)/2-s}(\R)$. It follows in particular that $\big(\Phi*E_0\,S\big)(\omega)$ defined by
\begin{equation}\label{DefuS2}
\begin{aligned}
\big(\Phi*E_0\,S\big)(\omega)(\phi,\varphi)
&:= \big(\Phi*E_0\,S\big)(\omega,\phi,\varphi)\\
&:=\Big\langle \Phi(\omega)*\int_\dom E_0(\cdot,x)S(x)\varphi(x)\dl x,\phi\Big\rangle_{H^{-s}(\R)\times H^s(\R)}\\
&\;=\big\langle u_{+,\mathrm S}^\varphi(\omega),\phi\big\rangle_{H^{-s}(\R)\times H^s(\R)},\\
&\hspace{3.5cm}\omega\in\Omega,\;\phi\in H^s(\R),\;\varphi\in L_2(\dom),
\end{aligned}
\end{equation}
belongs to $H^{-s}(\R)\hat\otimes L_2(\dom)$ for all $\omega\in\Omega$.

Now we verify the assertion \eqref{mainResuSnotin}, i.e.,
\begin{equation*}
\big(\Phi*E_0\,S\big)(\omega)\notin \bigcup_{r\geq0}H^{-r}(\R)\hat\otimes H^{1+\alpha}(\dom)
\;\text{ on }\; \{\omega\in\Omega:\Phi(\omega)\not\equiv0\}.
\end{equation*}
It suffices to show that, for all $\omega\in\Omega$ such that $\Phi(\omega)\not\equiv 0$, the linear mapping $L_2(\dom)\ni\varphi\mapsto u_{+,\mathrm S}^\varphi(\omega)\in H^{(1-\alpha)/2-s}(\R)$ can {\it not} be extended to an element of
$\bo\big((H^{1+\alpha}(\dom))'; H^{-r}(\R)\big)$ for any $r\geq0$.
(Recall that we have $L_2(\dom)=(L_2(\dom))'\hookrightarrow (H^{1+\alpha}(\dom))'$ by Convention~\ref{convTP}.) Due to \eqref{DefuS1} and the norm equivalence mentioned subsequent to \eqref{HsFT},
\[
\gnnrm{u_{+,\mathrm S}^\varphi(\omega)}{H^{-r}(\R)}^2
\geq C\int_\R\big(1+|\xi|^2\big)^{-r}\Big|c(\omega,i\xi)\int_\dom e^{-r\sqrt{i\xi}}S(x)\varphi(x)\dl x\Big|^2\dl\xi.
\]
By the definition \eqref{DefPhi} of $\Phi(\omega)$ we know that $\Phi(\omega)\not\equiv 0$ if, and only if, $c(\omega,i\,\cdot\,)\not\equiv0$. In this case there exists a bounded set $B=B(\omega)\in\cB(\R)$ of positive Lebesgue measure and $\delta_0=\delta_0(\omega)>0$ such that $|c(\omega,i\xi)|\geq\delta_0$ for all $\xi\in B$. W.l.o.g.\ we can assume that $B\subset[0,\infty)$; the case $B\subset(-\infty,0]$ is treated similarly. Let $\varphi\in L_2(\dom)=(L_2(\dom))'$ be real-valued with $\supp\varphi\subset\overline{\dom_\epsilon}$, where $\dom_\epsilon:=\dom\cap\{x=(r\cos\theta,r\sin\theta):r<\epsilon\}$ with $\epsilon=\epsilon(B)>0$ such that $\cos\big(-r\sin(\pi/4)\sqrt{\xi}\big)\geq\delta_1>0$ for all $\xi\in B$ and $r\leq\epsilon$. Then, for all $\xi\in B$,
\begin{align*}
\Big|\int_\dom e^{-r\sqrt{i\xi}}S(x)\varphi(x)\dl x\Big|
&\geq C\Big|\int_\dom\cos\big(-r\sin(\pi/4)\sqrt\xi\big)S(x)\varphi(x)\dl x\Big|\\
&\geq C\delta_1\Big|\int_{\dom_\epsilon} S(x)\varphi(x)\dl x\Big|,
\end{align*}
where $C>0$ depends on $\dom$ and $B$. Since the restriction of $S$ to $\dom_\epsilon$ does not belong to $H^{1+\alpha}(\dom_\epsilon)$ (see \cite[Theorem 1.4.5.3]{Gri85}), one can find a sequence $(\varphi_k)_{k\in\N}\subset L_2(\dom_\epsilon;\R)\subset L_2(\dom_\epsilon)=(L_2(\dom_\epsilon))'$ which is bounded in $(H^{1+\alpha}(\dom_\epsilon))'$ and satisfies 
\[\lim_{k\to\infty}\Big|\int_{\dom_\epsilon} S(x)\varphi_k(x)\dl x\Big|=\infty.\] 
Identifying each $\varphi_k$ with its extension by zero to $\dom$, the sequence $(\varphi_k)_{k\in\N}\subset L_2(\dom)=(L_2(\dom))'$ is bounded in $(H^{1+\alpha}(\dom))'$ and we have $\lim_{k\to\infty}\nnrm{u_{+,\mathrm S}^{\varphi_k}(\omega)}{H^{-r}(\R)}=\infty$,
which implies the assertion.

\subsubsection{Combining the results}
\label{Combining the results}

We know that, for $\wP$-almost all $\omega\in\Omega$, the decomposition 
\begin{equation*}
U(\omega,i\xi)=U_{\mathrm R}(\omega,i\xi)+ c(\omega,i\xi)e^{-r\sqrt z}S,\quad\xi\in\R,
\end{equation*}
takes place in $D(\Lap)$, cf.~\eqref{StoGriDecomp}. In particular, for $\wP$-almost all $\omega\in\Omega$,
\begin{equation}\label{Combining1}
\big\langle U(\omega,i\xi),\varphi\big\rangle=\big\langle U_{\mathrm R}(\omega,i\xi),\varphi\big\rangle+\big\langle c(\omega,i\xi)e^{-r\sqrt z}S,\varphi\big\rangle,\quad\xi\in\R,\;\varphi\in L_2(\dom),
\end{equation}
where we abbreviate $\langle\,\cdot\,,\,\cdot\,\rangle=\langle\,\cdot\,,\,\cdot\,\rangle_{L_2(\dom)\times (L_2(\dom))'}$. By the definition \eqref{U(z)} of $U$ and the embedding $L_2(\dom)\hookrightarrow H^{-1}(\dom)$, the left hand side in \eqref{Combining1} can be rewritten as
\begin{equation}\label{Combining2}
\begin{aligned}
\big\langle U(\omega,i\xi),\varphi\big\rangle_{H^1_0(\dom)\times H^{-1}(\dom)}
&=\int_0^T e^{-i\xi t}\langle u(\omega,t),\varphi\rangle_{H^1_0(\dom)\times H^{-1}(\dom)}\dl t\\
&= \big[\cF_{t\to\xi}\big(\langle u_+(\omega),\varphi\rangle_{H^1_0(\dom)\times H^{-1}(\dom)}\big)\big](\xi).
\end{aligned}
\end{equation}
Similarly, by the definition \eqref{DefuR1} of $u_{+,\mathrm R}^\varphi$ and the embedding $L_2(\dom)\hookrightarrow (H^2(\dom))'$, the first term on the  right hand side in \eqref{Combining1} equals
\begin{equation}\label{Combining3}
\big\langle U_{\mathrm R}(\omega,i\xi),\varphi\big\rangle_{H^2(\dom)\times (H^2(\dom))'}
=\big[\cF_{t\to\xi}\big(u^\varphi_{+,\mathrm R}(\omega)\big)\big](\xi),
\end{equation}
and the second term on the right hand side in \eqref{Combining1} is 
\begin{equation}\label{Combining4}
\big\langle c(\omega,i\xi)e^{-r\sqrt z}S,\varphi\big\rangle_{L_2(\dom)\times (L_2(\dom))'}
=\big[\cF_{t\to\xi}\big(u^\varphi_{+,\mathrm S}(\omega)\big)\big](\xi)
\end{equation}
due to \eqref{DefuS1}.

The combination of \eqref{Combining1}, \eqref{Combining2}, \eqref{Combining3}, \eqref{Combining4}, together with the definitions \eqref{DefuR1}, \eqref{DefuS2} and the uniqueness of the Fourier transform, implies  that, for $\wP$-almost all $\omega\in\Omega$,
\begin{equation}\label{Combining5}
u_+(\omega,\phi,\varphi)=u_{+,\mathrm R}(\omega,\phi,\varphi)+\big(\Phi*E_0S\big)(\omega,\phi,\varphi),\quad \phi\in H^s(\R),\;\varphi\in L_2(\dom),
\end{equation}
where we use the notation introduced in Convention~\ref{conventionu}.
It has been shown in Subsections~\ref{The solution process as a tensor product-valued random variable} and \ref{Inverse transform of U_R} that both $u_{+}(\omega)$ and $u_{+,\mathrm R}(\omega)$ belong to $H^{-s}(\R)\hat\otimes H^1_0(\dom)$ for all $\omega\in\Omega$. As a consequence, the bilinear mapping 
\[\big(\Phi*E_0S\big)(\omega):H^s(\R)\times L_2(\dom)\to\C\] extends for $\wP$-almost all $\omega\in\Omega$ continuously to the domain $H^{-s}(\R)\times H^{-1}(\dom)$.  This extension also belongs to $H^{-s}(\R)\hat\otimes H^1_0(\dom)$. For all $\omega\in\Omega$ such that \eqref{Combining5} does not hold we redefine $\Phi(\omega):=0$, so that $\big(\Phi*E_0S\big)(\omega)$ belongs to $H^{-s}(\R)\hat\otimes H^1_0(\dom)$ for all $\omega\in\Omega$. 

Estimates \eqref{DPZ(7.17)}, \eqref{EstTildeX}, \eqref{EstuR1} and the measurability of $u_{+,\mathrm R}$ proved at the end of Subsection~\ref{Inverse transform of U_R} imply that $u_{+,\mathrm R}$ is an element of $L_2\big(\Omega,\cF_T,\wP;H^{-s}(\R)\hat\otimes H^2(\dom)\big)$. The embedding $H^{-s}(\R)\hat\otimes H^2(\dom)\hookrightarrow H^{-s}(\R)\hat\otimes H^1(\dom)$ yields $u_{+,\mathrm R}\in L_2\big(\Omega,\cF_T,\wP;H^{-s}(\R)\hat\otimes H^1_0(\dom)\big)$.
Since $u_{+}$ also belongs to $L_2\big(\Omega,\cF_T,\wP;H^{-s}(\R)\hat\otimes H^1_0(\dom)\big)$, so does $\Phi*E_0S$. 
Moreover, from the definition \eqref{DefPhi} of $\Phi$, the estimates \eqref{DPZ(7.17)}, \eqref{EstTildeX}, \eqref{Inversec2}, Lemma~\ref{cURholom}(i) and the continuity of the inverse Fourier transform $\cF_{\xi\to t}^{-1}:L_2(\R,(1+\xi^2)^{(1-\alpha)/2-s}\dl\xi;\C)\to H^{(1-\alpha)/2-s}(\R)$, one can derive the $\cF_T/\cB(H^{(1-\alpha)/2-s}(\R))$-measurability of $\Phi$. This, together with the estimates \eqref{DPZ(7.17)}, \eqref{EstTildeX} and \eqref{Inversec2}, yields $\Phi\in L_2(\Omega,\cF_T,\wP;H^{(1-\alpha)/2-s}(\R))$.
Together with \eqref{suppPhi}, this finishes the proof of the first assertion of Theorem~\ref{mainRes}

The assertion \eqref{mainResuSnotin} has been verified at the end of Subsection~\ref{Inverse transform of c e^{-r sqrt z}S}. The assertion \eqref{mainResDefPhi} is a consequence of \eqref{c(omega,z)} and \eqref{DefPhi}. Finally, the estimate \eqref{mainResEst} follows from \eqref{EstTildeX}, \eqref{EstuR1} and \eqref{Inversec3}. Theorem~\ref{mainRes} is proved.

\begin{appendix}

\section{Tensor products of Sobolev spaces}
\label{Tensor products of Sobolev spaces}

In this section we present several supplementary details concerning tensor products of Sobolev spaces as introduced in Subsection~\ref{The solution process as a tensor product-valued random variable}. Let us first look at the connection to altenative definitions and at further properties of tensor products of Hilbert spaces; our references for the described setting are Defant and Floret~\cite[Sections~2 and 26]{DeFl93}, Kadison and Ringrose~\cite[Section~2.6]{KaRi83}, Tr\`eves~\cite[Part~III]{Tre67} and Weidmann~\cite[Section~3.4]{Wei80}.

\subsubsection*{Tensor products of Hilbert spaces}

Let $\cH$ and $\cG$ be separable complex Hilbert spaces. The (Hilbert-Schmidt) tensor product $\cH\hat\otimes\cG$ is often introduced as an abstract completion of the algebraic tensor product $\cH\otimes\cG$. 

As in Subsection~\ref{The solution process as a tensor product-valued random variable}, for $h\in \cH$ and $g\in \cG$ we denote by $h\otimes g:\cH'\times \cG'\to\C$ the bilinear functional defined by
\begin{equation*}
h\otimes g(h',g')
:=\langle h,h'\rangle_{\cH\times \cH'}\langle g,g'\rangle_{\cG\times \cG'}=h'(h)g'(g),\qquad h'\in \cH',\,g'\in \cG'.
\end{equation*}
Then, (a realization of) the {\it algebraic tensor product} $\cH\otimes \cG$ of $\cH$ and $\cG$ is given by the vector space of all bilinear functionals on $\cH'\times \cG'$ of the form
\begin{equation}\label{algTen}
\sum_{j=1}^N h_j\otimes g_j,
\end{equation}
where $h_j\in \cH$, $g_j\in \cG$, $j=1,\ldots,N$ and $N\in\N$, endowed with the natural vector addition and scalar multiplication; compare \cite[Section 2.2]{DeFl93} or \cite[Chapter 42]{Tre67}. 
The algebraic tensor product is often alternatively introduced via the `universal property' (see \cite[Section 2.2]{DeFl93} or \cite[Chapter 39]{Tre67}) and thus uniquely determined up to an isomorphism, or it is defined as a quotient space of formal expansions of the form \eqref{algTen}, cf.\ \cite[Section 3.4]{Wei80}.
We choose here the specific realization of the algebraic tensor product as a space of bilinear functionals on the product of the dual spaces to be coherent with Definition~\ref{defTP}.

By setting
\begin{equation}\label{TPalt}
\left\langle\sum_{j=1}^N h_j\otimes g_j,\sum_{k=1}^M\tilde h_k\otimes\tilde g_k\right\rangle_{\cH\hat\otimes \cG}^\sim:=\sum_{j=1}^N\sum_{k=1}^M\langle h_j,\tilde h_k\rangle_\cH\langle g_j,\tilde g_k\rangle_\cG
\end{equation}
where $h_j,\tilde h_k\in \cH$, $g_j,\tilde g_k\in \cG$, $j=1,\ldots,N$, $k=1,\dots,M$ and $N,M\in\N$, one defines a scalar product on $\cH\otimes \cG$. The Hilbert-Schmidt tensor product $\cH\hat\otimes \cG$ is often defined as the abstract completion of $\cH\otimes \cG$ w.r.t.\ this scalar product, cf.~\cite[Section 3.4]{Wei80}. Let us show that this definition and Definition~\ref{defTP} are equivalent in the sense that they yield isometric isomorphic spaces. We note that the family of functionals
\begin{equation*}
h_j\otimes g_k:\cH'\times \cG'\to\C,(h',g')\mapsto\langle h_j,h'\rangle_{\cH\times \cH'}\langle g_k,g'\rangle_{\cG\times \cG'},\quad j,k\in\N,
\end{equation*}
is an orthonormal basis of the space of Hilberts-Schmidt functionals on $\cH'\times \cG'$ (as defined in Subsection~\ref{The solution process as a tensor product-valued random variable}) whenever $(h_j)_{j\in\N}$ and $(g_k)_{k\in\N}$ are orthonormal bases of $\cH$ and $\cG$. This follows from 
\cite[Proposition~2.6.2]{KaRi83}, the fact that
$(\langle\cdot,h_j\rangle_\cH)_{j\in\N}$ and $(\langle\cdot,g_k\rangle_\cG)_{k\in\N}$ are orthonormal bases of $\cH'$ and  $\cG'$, and the identities $\langle h',\langle\cdot,h_j\rangle_\cH\rangle_{\cH'}=\langle h_j,R_\cH h'\rangle_\cH=\langle h_j,h'\rangle_{\cH\times \cH'}$, $\langle g',\langle\cdot,g_k\rangle_\cG\rangle_{\cG'}=\langle g_k,R_\cG g'\rangle_\cG=\langle g_k,g'\rangle_{\cG\times \cG'}$. Here $R_\cH$ and $R_\cG$ denote the conjugate linear Riesz mappings from $\cH'$ to $\cH$ and from $\cG'$ to $\cG$, respectively. As a consequence, the algebraic tensor product $\cH\otimes\cG$ is dense in the space $\cH\hat\otimes\cG$ as introduced in Definition~\ref{defTP}.
Therefore, in order to verify that both definitions of $\cH\hat\otimes \cG$ are equivalent, it is sufficient to check that the norm $\nnrm{\cdot}{\cH\hat\otimes\cG}^\sim$ corresponding to the scalar product \eqref{TPalt} and the norm $\nnrm{\cdot}{\cH\hat\otimes\cG}$ in Definition~\ref{defTP} coincide on $\cH\otimes\cG$. To this end, consider a finite linear combination of simple tensors as in \eqref{algTen}, let $(h_k')_{k\in\N}$ and $(g_l')_{l\in\N}$ be orthonormal bases of $\cH'$ and $\cG'$, and observe that
\begin{align*}
\left\|\sum_{j=1}^Nh_j\otimes g_j\right\|_{\cH\hat\otimes \cG}^2&=\sum_{k,l\in\N}\left|\left(\sum_{j=1}^Nh_j\otimes g_j\right)(h_k',g_l')\right|^2\\
&=\sum_{k,l\in\N}\left|\sum_{j=1}^N\langle h_j,h_k'\rangle_{\cH\times \cH'}\langle g_j,g_l'\rangle_{\cG\times \cG'}\right|^2\\
&=\sum_{k,l\in\N}\sum_{i,j=1}^Nh_k'(h_i)\overline{h_k'(h_j)}
g_l'(g_i)\overline{g_l'(g_j)}\\
&=\sum_{i,j=1}^N\langle h_i,h_j\rangle_\cH\langle g_i,g_j\rangle_\cG\\
&=\Bigg(\Bigg\|\sum_{j=1}^Nh_j\otimes g_j\Bigg\|_{\cH\hat\otimes \cG}^\sim\Bigg)^2
\end{align*}
by Parseval's equality.

In the course of the proof of Theorem~\ref{mainRes} we make use of the fact that the definition of $\cH\hat\otimes \cG$ as the space of Hilbert-Schmidt functionals on $\cH'\times \cG'$ entails the isometric isomorphisms to spaces of Hilbert-Schmidt operators
\[K:\cH\hat\otimes \cG\to\hs(\cH';\cG)
\qquad\text{ and }\qquad L:\cH\hat\otimes \cG\to\hs(\cG';\cH)\]
given by
\begin{equation}\label{HSFasHSO}
f(h',g')=\langle(Kf)h',g'\rangle_{\cG\times \cG'}=\langle (Lf)g',h'\rangle_{\cH\times \cH'},\quad f\in \cH\hat\otimes \cG,\;h'\in \cH',\;g'\in \cG'.
\end{equation}
We may therefore consider elements $f\in \cH\hat\otimes \cG$ as operators $Kf\in\hs(\cH';\cG)$ or $Lf\in\hs(\cG';\cH)$ whenever this point of view appears to be convenient. In doing so, we omit the explicit notation of the mappings $K$ and $L$ if no confusion is possible.

\subsubsection*{Embeddings of tensor products of Sobolev spaces}

The Sobolev-Slobodeckij spaces of order $s\geq0$ on an arbitrary domain $D\subset\R^d$ can be defined as follows, cf.~\cite{Gri85}.
\begin{definition}\label{DefSobolev}
Let $D$ be an open subset of $\R^d$ ($d\in\N$) and $s\geq0$. The Sobolev-Slobodeckij space $H^s(D)$ is the space of all square integrable, complex-valued functions $f\in L_2(D)$ such that
\begin{itemize}
\item[(i)] $D^\alpha f\in L_2(D)$ for $|\alpha|\leq s,\;\alpha\in\N^d,$ if $s\in\Nn$,
\item[(ii)] $f\in H^{\lfloor s\rfloor}(D)$ and
\[\int_D\int_D\frac{\big|D^\alpha f(x)-D^\alpha f(y)\big|^2}{|x-y|^{2(s-\lfloor s\rfloor)+d}}\,\mathrm dx\,\mathrm  dy<\infty\]
for $|\alpha|=\lfloor s\rfloor,\;\alpha\in\N^d$, if $s\notin\Nn$. 
\end{itemize}
Here, $|\alpha|=\sum_{j=1}^d|\alpha_j|$, $D^\alpha=\partial^{|\alpha|}/(\partial^{\alpha_1}\ldots\partial^{\alpha_d})$, and $\lfloor s\rfloor\in\N$ is the largest integer smaller than or equal to $s$.
The space $H^s(D)$ is endowed with the norm
\begin{equation*}\label{Wspi}
\nnrm{f}{H^s(D)}:=\Bigg(\sum_{|\alpha|\leq s}\gnnrm{D^\alpha f}{L_2(D)}^2\Bigg)^{1/2}
\end{equation*}
in case (i), and with the norm
\begin{equation*}\label{Wspii}
\nnrm{f}{H^s(D)}:=\Bigg(\nnrm{f}{H^{\lfloor s\rfloor}(D)}^2+\sum_{|\alpha|=\lfloor s\rfloor}\int_D\int_D\frac{\big|D^\alpha f(x)-D^\alpha f(y)\big|^2}{|x-y|^{2(s-\lfloor s\rfloor)+d}}\,\mathrm dx\,\mathrm  dy\Bigg)^{1/2}
\end{equation*}
in case (ii).
\end{definition}

Recall from Subsection~\ref{The solution process as a tensor product-valued random variable} that we put $H^{-s}(D):=(H^s_0(D))'$, $s\geq0$. In particular, $H^{-s}(\R^d)=(H^s_0(\R^d))'=(H^s(\R^d))'$. In the proof of Theorem~\ref{mainRes} we use the following well-known characterization of the spaces $H^s(\R^d)$, $s\in\R$, in terms of the Fourier transform: One has
\begin{equation}\label{HsFT}
H^s(\R^d)=\left\{f\in\mathcal S'(\R^d):\gnnrm{\mathcal F^{-1}(1+|\cdot|^2)^{s/2}\mathcal Ff}{L_2(\R^d)}<\infty\right\},
\end{equation}
for all $s\in\R$ and  $\nnrm{\mathcal F^{-1}(1+|\cdot|^2)^{s/2}\mathcal F\cdot}{L_2(\R^d)}$ is an equivalent norm for $H^s(\R^d)$; see, e.g., \cite{Tri78} (Definition 2.3.1, the last identity in Theorem 2.3.2(d), Theorem 2.3.3, Theorem~2.5.1, Remarks~2.5.1/3, 2.5.1/4 and Theorem 2.6.1(a) therein).

Also recall from Subsection~\ref{The solution process as a tensor product-valued random variable} that we identify $L_2(D)$ with its (topological) dual space $(L_2(D))'$ via the isometric isomorphism
$L_2(D)\ni v\mapsto \langle\,\cdot\,,\overline v\rangle_{L_2(D)}\in (L_2(D))'$, so that we obtain a chain of continuous and dense (linear) embeddings
\begin{equation*}
H^{s_2}(D)\hookrightarrow H^{s_1}(D)\hookrightarrow L_2(D)= (L_2(D))'\hookrightarrow (H^{s_1}(D))'\hookrightarrow (H^{s_2}(D))',\quad 0\leq s_1\leq s_2,
\end{equation*}
see Convention~\ref{convTP}.
The next proposition leads to a useful characterization of the respective embeddings of tensor products of Sobolev spaces, which will be used in the proof of Theorem~\ref{mainRes}, too.
For Hilbert spaces $\cH_1$, $\cH_2$, $\cG_1$, $\cG_2$ and bounded linear operators $S\in\bo(\cH_1;\cH_2)$, $T\in\bo(\cG_1;\cG_2)$, we denote by $S\hat\otimes T\in\bo\big(\cH_1\hat\otimes\cG_1;\cH_2\hat\otimes\cG_2\big)$ the bounded linear extension of the tensor product $S\otimes T:\cH_1\otimes\cG_1\to\cH_2\otimes\cG_2$ given by $S\otimes T\,(h\otimes g):=(Sh)\otimes(Tg)$, $h\in\cH_1$, $g\in\cG_1$.

\begin{proposition}\label{TPemb}
Let $\cH_1$, $\cH_2$, $\cG_1$, $\cG_2$ be separable complex Hilbert spaces and let $\i:\cH_1\hookrightarrow\cH_2$, $\j:\cG_1\hookrightarrow\cG_2$ be continous linear embeddings. Then, the tensor product map $\i\hat\otimes\j:\cH_1\hat\otimes\cG_1\to\cH_2\hat\otimes\cG_2$ is a continous linear embedding that acts on elements $f\in\cH_1\hat\otimes\cG_1$ via
\begin{equation*}
\big(\i\hat\otimes\j\,(f)\big)(h',g')=f\big(\i'(h'),\j'(g')\big),\quad h'\in\cH_2',\;g'\in\cG_2',
\end{equation*}
where $\i':\cH_2'\to\cH_1'$ and $\j':\cG_2'\to\cG_1'$ are the dual operators to $\i$ and $\j$.
\end{proposition}

\begin{proof}
Let $(b_k)_{k\in\N}$ and $(e_l)_{l\in\N}$ be orthonormal bases of $\cH_1$ and $\cG_1$, respectively, with dual orthonormal bases $(b_k')_{k\in\N}:=\big(\langle\cdot,b_k\rangle_{\cH_1}\big)_{k\in\N}$ and $(e_l')_{l\in\N}:=\big(\langle\cdot,e_l\rangle_{\cG_1}\big)_{l\in\N}$ of $\cH_1'$ and $\cG_1'$.
Fix $f\in\cH_1\hat\otimes\cG_1$. Since $(b_k\otimes e_l)_{k,l\in\N}$ is an orthonormal basis of $\cH_1\hat\otimes\cG_1$ and since $\langle f,b_k\otimes e_l\rangle_{\cH_1\hat\otimes\cG_1}=f(b_k',e_l')$, we have the expansion
\[f=\sum_{k,l\in\N}f(b_k',e_l')\,b_k\otimes e_l\]
in the space $\cH_1\hat\otimes\cG_1$. The tensor product $\i\hat\otimes\j$ maps $f$ to
\[\i\hat\otimes\j\,(f)=\sum_{k,l\in\N}f(b_k',e_l')\,\i(b_k)\otimes\j(e_l),\]
where the infinite sum converges unconditionally in $\cH_2\hat\otimes\cG_2$. Therefore, for all $h'\in\cH_2'$ and $g'\in\cH_2'$,
\begin{align*}
\big(\i\hat\otimes\j\,(f)\big)(h',g')&=\sum_{k,l\in\N}f(b_k',e_l')\,\big\langle\i(b_k),h'\big\rangle_{\cH_2\times\cH_2'}\big\langle\j(e_l),g'\big\rangle_{\cG_2\times\cG_2'}\\
&=\sum_{k,l\in\N}f(b_k',e_l')\,\big\langle b_k,\i'(h')\big\rangle_{\cH_1\times\cH_1'}\big\langle e_l,\j'(g')\big\rangle_{\cG_1\times\cG_1'}\\
&= f\big(\i'(h'),\j'(g')\big),
\end{align*}
where the infinite sum converges unconditionally in $\C$. The injectivity of $\i\hat\otimes\j$ follows from the injectivity of $\i$ and $\j$ and the bilinear structure of the tensor product spaces.
\end{proof}

For an arbitrary domain $D\subset\R^d$ and smoothness parameters $0\leq s_2\leq s_2$, the dual embbedding $\i':(H^{s_1}(D))'\hookrightarrow(H^{s_2}(D))'$ to the  natural embedding $\i:H^{s_2}(D)\hookrightarrow H^{s_1}(D)$ is the restriction of functionals on $H^{s_1}(D)$ to the smaller domain $H^{s_2}(D)$. Therefore, with respect to the tensor product spaces $(H^s(I))'\hat\otimes H^r(\dom)$, $r,s\geq0$, where $I=(0,T)$ or $I=\R$, Proposition~\ref{TPemb} states the following: For $0\leq s_1\leq s_2$ and $0\leq r_1\leq r_2$ the embedding
\begin{equation}\label{TPemb1App}
(H^{s_1}(I))'\hat\otimes H^{r_2}(\dom)\hookrightarrow (H^{s_2}(I))'\hat\otimes H^{r_1}(\dom),
\end{equation}
given by the tenor product $\i\hat\otimes\j$ of the natural embeddings $\i:(H^{s_1}(I))'\hookrightarrow (H^{s_2}(I))'$ and $\j:H^{r_2}(\dom)\hookrightarrow  H^{r_1}(\dom)$, acts on an element $f\in(H^{s_1}(I))'\hat\otimes H^{r_2}(\dom)$ by restricting the bilinear functional $f:H^{s_1}(I)\times(H^{r_2}(\dom))'$ to the smaller domain $H^{s_2}(I)\times (H^{r_1}(\dom))'$. 
Identifying $(H^s(I))'\hat\otimes H^r(\dom)$ with the spaces of Hilbert-Schmidt operators $\hs\big(H^s(I);H^r(\dom)\big)$ and  $\hs\big((H^{r}(\dom))';(H^{s}(I))'\big)$ according to \eqref{HSFasHSO}, it is clear that the embedding \eqref{TPemb1App} becomes the restriction of operators $T\in\hs\big(H^{s_1}(I);H^{r_2}(\dom)\big)$ to the domain $H^{s_2}(I)$ and the restriction of operators $\hs\big((H^{r_2}(\dom))';(H^{s_1}(I))'\big)$ to the domain $(H^{r_1}(\dom))'$, respectively.

\section{General bounded polygonal domains}
\label{General bounded polygonal domains}

In this section we state the extension of our main result to general bounded polygonal domains. It can be verified by following the lines of Secion~\ref{Proof of the main result}, using instead of the results of Subsection~\ref{Estimates for the Helmholtz equation} their generalization to the Helmholtz equation on arbitrary bounded polygonal domains; see \cite[Chapter~2]{Gri92}. 

Let $\dom\subset\R^2$ be a general bounded polygonal domain as specified at the beginning of Subsection~\ref{Stochastic heat equation}. By $V_j\in\R^2,\;j\in\{1,\ldots,n\},$ we denote the vertices of the boundary $\partial\dom$, numbered according to their
order within $\partial\dom$ in counter-clockwise orientation. The respective interior angles of $\partial\dom$
at these vertices are denoted by $\gamma_j\in(0,2\pi)$. With each vertex $V_j$ we associate a system of polar coordinates $(r_j,\theta_j)\in[0,\infty)\times[0,2\pi)$
with origin at $V_j$. For any point $P$ in the plane, the $r_j$-coordinate is the distance from $P$ to $V_j$ and the
$\theta_j$-coordinate is the angle between the line segments $V_jV_{j+1}$ and $V_jP$; see Figure~\ref{fig:Bild} below.


\psset{yunit=1.7,xunit=1.7}

\begin{figure}
\begin{center}
\begin{pspicture}(3.6,3.5)

\pspolygon(0.3,0.9)(1.6,1.5)(2,0)(3.7,1)(3.3,3.5)(2.4,2.7)(1,3.6)(0.7,2.6)(0,2)

\psdot(0.3,0.9)
\psdot(1.6,1.5)
\psdot(2,0)
\psdot(3.7,1)
\psdot(3.3,3.5)
\psdot(2.4,2.7)
\psdot(1,3.6)
\psdot(0.7,2.6)
\psdot(0,2)

\psline{-}(1.6,1.5)(3,1.8)
\psdot(3,1.8)
\psarc{<->}(1.6,1.5){1.1}{-75}{13}
\psarc{<->}(1.6,1.5){0.9}{-75}{204}

\uput[d](1.5,1.4){$V_j$}
\uput[u](3,1.8){$P$}
\uput[r](0.2,0.7){$V_{j-1}$}
\uput[l](2,0){$V_{j+1}$}

\uput[u](1.55,1.6){$\gamma_j$}
\uput[r](2.05,1){$\theta_j$}
\uput[u](2.4,1.6){$r_j$}

\uput[u](1.3,2.4){\large$\dom$}

\end{pspicture}
\end{center}

\caption[Bild]{\begin{tabular}[t]{l}  General bounded polygonal domain $\dom\subset\R^2$\end{tabular}}\label{fig:Bild}
\end{figure}


In analogy to \eqref{defS}, we set $\alpha_j:=\pi/\gamma_j$ and define corner singularity functions $S_j:\dom\to\R$ associated to every vertex $V_j$ such that $\gamma_j>\pi$ by 
\begin{equation}\label{defSj}
S_j(r_j,\theta_j):=\eta_j(r_j,\theta_j)r_j^{\alpha_j}\sin(\alpha_j\theta_j).
\end{equation}
Here $\eta_j\in C^\infty(\overline\dom;\R)$ is a truncation function which does
not depend on the $\theta_j$-coordinate, equals one near $V_j$ and vanishes in a neighborhood
of the sides of $\partial\dom$ which do not end at $V_j$. We choose
these truncation functions in such a way that their supports are disjoint. Further, for all $j\in\{1,\ldots,n\}$ such that $\gamma_j>\pi$ we define $\psi_j\in L_2(\dom)$ by
\[\psi_j(r_j,\theta_j):=\eta_j(r_j,\theta_j)r_j^{-\alpha_j}\sin(\alpha_j\theta_j),\]
and $\varphi_j\in D(\Lap)$ is the unique solution in $H^1_0(\dom)$
to the problem $\Delta\varphi_j=\Delta\psi_j$. In analogy to \eqref{v1}, we define functions $v_j(z)\in L_2(\dom)$ for all $z\in\C\setminus \sigma(\Lap)$ and $j\in\{1,\ldots,n\}$ such that $\gamma_j>\pi$ by
\begin{equation}\label{vj}
v_j(z):=\left[\id{L_2(\dom)}-z\big(z\id{L_2(\dom)}-\Lap\big)^{-1}\right]\frac 1\pi(\psi_j-\varphi_j),
\end{equation}
where $\big(z\id{L_2(\dom)}-\Lap\big)^{-1}\in\bo(L_2(\dom))$ is the $z$-resolvent of $\Lap$.
In analogy to \eqref{defE0}, we define kernel functions $E_j:\R\times\dom\to \R$ for all $j\in\{1,\ldots,n\}$ such that $\gamma_j>\pi$ by
\begin{equation}\label{defEj}
E_j(t,x):=\one_{(0,\infty)}(t)(2\sqrt\pi)^{-1}t^{-3/2}r_j e^{-r_j^2/(4t)},\quad t\in\R,\;x=(r_j\cos\theta_j,r_j\sin\theta_j)\in\dom.
\end{equation}

Now we can formulate the generalization of Theorem~\ref{mainRes}.

\begin{thm}\label{mainResGen}
Let $\dom\subset\R^2$ be a bounded polygonal domain with vertices $V_j\in\R^2$, $j\in\{1,\ldots,n\}$ and corresponding interior angles $\gamma_j\in (0,2\pi)$ as described above.
Let Assumption \ref{AssFBu0} hold, let $u=(u(t))_{t\in[0,T]}$ be the mild solution to Eq.~\eqref{SHE} and let $u_+$ be its extension by zero to the whole real line, considered as an element of
$L_2\big(\Omega,\cF_T,\wP;L_2(\R)\hat\otimes H^1_0(\dom)\big)$ as described in Subsection~\ref{The solution process as a tensor product-valued random variable}. Let $s> 1/2$ and set $\alpha_j:=\pi/\gamma_j$. 

There exist
\[u_{+,\mathrm R}\in L_2\big(\Omega,\cF_T,\wP;H^{-s}(\R)\hat\otimes H^2(\dom)\big)\cap L_2\big(\Omega,\cF_T,\wP;H^{-s}(\R)\hat\otimes H^1_0(\dom)\big)\]
and
\[\Phi_j\in L_2\big(\Omega,\cF_T,\wP; H^{(1-\alpha_j)/2-s}(\R)\big),\quad j\in\{1,\ldots,n\}\text{ such that }\gamma_j>\pi,\]
with $\supp\Phi_j(\omega)\subset [0,\infty)$ for all $\omega\in\Omega$
(in the sense of distributions), such that the equality
\begin{equation*}\label{mainRes1}
u_+=u_{+,\mathrm R}+\sum_{\gamma_j>\pi}\Phi_j*E_j\,S_j
\end{equation*}
holds in $L_2\big(\Omega,\cF_T,\wP;H^{-s}(\R)\hat\otimes H^1_0(\dom)\big)$.
Here $\Phi_j*E_j\,S_j$ denotes the element of the space $L_2\big(\Omega,\cF_T,\wP;H^{-s}(\R)\hat\otimes H^1_0(\dom)\big)$ that acts on test functions $(\phi,\varphi)\in H^s(\R)\times L_2(\dom)$ 
\linebreak$(\hookrightarrow H^s(\R)\times H^{-1}(\dom))$ via
\begin{equation*}
\begin{aligned}
\big(\Phi_j*E_j\,S_j\big)(\omega)(\phi,\varphi)
&:= \big(\Phi_j*E_j\,S_j\big)(\omega,\phi,\varphi)\\
&:=\Big\langle \Phi_j(\omega)*\int_\dom E_j(\cdot,x)S_j(x)\varphi(x)\dl x,\phi\Big\rangle_{H^{-s}(\R)\times H^s(\R)},\; \omega\in\Omega,
\end{aligned}
\end{equation*}
where $S_j$ and $E_j$ are given by \eqref{defSj} and \eqref{defEj}, $\int_\dom E_j(\cdot,x)S_j(x)\varphi(x)\dl x$ denotes the function
$\R\ni t\mapsto\int_\dom E_j(t,x)S_j(x)\varphi(x)\dl x\in\C$, and $*$ is the usual convolution of Schwartz distributions.

We have
\begin{equation*}
\big(\Phi_j*E_j\,S_j\big)(\omega)\notin \bigcup_{r\geq0}H^{-r}(\R)\hat\otimes H^{1+\alpha_j}(\dom)
\;\text{ on }\; \{\omega\in\Omega:\Phi_j(\omega)\neq0\}
\end{equation*}
and $\Phi_j$ is determined in terms of its Fourier transform
w.r.t.\ the time variable $t\in\R$ as follows: For $\wP$-almost
every $\omega\in\Omega$,
\begin{equation*}\label{mainResGenDefPhij}
\big[\mathcal F_{t\to\xi}\big(\Phi_j(\omega)\big)\big](\xi)=\Big\langle H(\omega,i\xi),
\overline{v_j(i\xi)}\Big\rangle_{L_2(\dom)} \text{ for $\lambda$-almost every }\xi\in\R,
\end{equation*}
where $v_j$ and $H$ are defined by \eqref{vj} and \eqref{H(z)}.

Moreover, 
\begin{equation*}\label{mainResGenCorEst}
\begin{aligned}
&\E\Big(\nnrm{u_{+,\mathrm R}}{H ^{-s}(\R)\hat\otimes H^2(\dom)}^2
+\sum_{\gamma_j>\pi}\nnrm{\Phi_j}{H^{(1-\alpha_j)/2-s}(\R)}^2\Big)\\
&\leq C\E\Big(\nnrm{u_0}{L_2(\dom)}^2+\nnrm{u(T)}{L_2(\dom)}^2+\int_0^T\gnnrm{F\big(u(t)\big)}{L_2(\dom)}^2\dl t+\sup_{t\in[0,T]}\gnnrm{G\big(\tilde u(t)\big)}{\hs(U_0;L_2(\dom))}^2\Big),
\end{aligned}
\end{equation*}
where $C>0$ depends only on $s$, $T$, $\dom$ and the cut-off functions $\eta_j$ in \eqref{defSj}, and where $\tilde u=(\tilde u(t))_{t\in[0,T]}$ denotes the modification of $u=(u(t))_{t\in[0,T]}$ that is continuous in $L_2(\dom;\R)$.
\end{thm}

Similar to the proof of Corollary~\ref{mainResCor}, using the disjointness of the supports of the corner singularity functions $S_j$, one obtains the following decomposition of $u$ in the space $L_2\big(\Omega,\cF_T,\wP;(H^{s}(0,T))'\hat\otimes H^1_0(\dom)\big)$.

\begin{cor}
Let the setting of Theorem \ref{mainResGen} be given and consider the mild solution $u=(u(t))_{t\in[0,T]}$ to Eq.~\eqref{SHE} as an element of $L_2\big(\Omega,\cF_T,\wP;L_2(0,T)\hat\otimes H^1_0(\dom)\big)$ as described in Subsection~\ref{The solution process as a tensor product-valued random variable}. Let $\mathcal E:H^s(0,T)\to H^s(\R)$ be a linear and bounded extension operator. For $\omega\in\Omega$, $\phi\in H^s(0,T)$, $\varphi\in H^{-1}(\dom)$ and $j\in\{1,\ldots,n\}$ with $\gamma_j>\pi$ define
\[
u_{\mathrm R}(\omega,\phi,\varphi):=u_{+,\mathrm R}(\omega,\mathcal E\phi,\varphi),\qquad
u_{\mathrm S,j}(\omega,\phi,\varphi):=\big(\Phi_j*E_j\,S_j\big)(\omega,\mathcal E\phi,\varphi),
\]
where $u_{+,\mathrm R}$ and $\Phi_j*E_j\,S_j$ are as in Theorem \ref{mainResGen}.

Then, 
\begin{equation*}
u_{\mathrm R}\in L_2\big(\Omega,\cF_T,\wP;(H^{s}(0,T))'\hat\otimes H^2(\dom)\big)\cap L_2\big(\Omega,\cF_T,\wP;(H^{s}(0,T))'\hat\otimes H^1_0(\dom)\big),
\end{equation*}
\[
u_{\mathrm S,j}\in L_2\big(\Omega,\cF_T,\wP;(H^{s}(0,T))'\hat\otimes H^1_0(\dom)\big)
\]
for all $j\in\{1,\ldots,n\}$ such that $\gamma_j>\pi$,
and the decomposition
\begin{equation*}
u=u_{\mathrm R}+\sum_{\gamma_j>\pi}u_{\mathrm S,j}
\end{equation*}
holds as an equality in $L_2\big(\Omega,\cF_T,\wP;(H^{s}(0,T))'\hat\otimes H^1_0(\dom)\big)$. 
For $\wP$-almost every $\omega\in\Omega$,
\begin{equation*}
u_{\mathrm S,j}(\omega)\notin\bigcup_{r\geq0}(H^r(0,T))'\hat\otimes H^{1+\alpha_j}(\dom)
\quad\Longleftrightarrow\quad
u_{\mathrm S,j}(\omega)\not\equiv0
\quad\Longleftrightarrow\quad
\Phi_j(\omega)\not\equiv0.
\end{equation*}

Moreover, there exists a constant $C>0$, depending only on $s$, $T$, $\dom$ and the cut-off functions $\eta_j$ in \eqref{defSj}, such that
\begin{equation*}
\begin{aligned}
\E\nnrm{u_{\mathrm R}}{(H ^s(0,T))'\hat\otimes H^2(\dom)}^2
\leq C\E\Big(&\nnrm{u_0}{L_2(\dom)}^2+\nnrm{u(T)}{L_2(\dom)}^2\\
&+\int_0^T\gnnrm{F\big(u(t)\big)}{L_2(\dom)}^2\dl t+\sup_{t\in[0,T]}\gnnrm{G\big(\tilde u(t)\big)}{\hs(U_0;L_2(\dom))}^2\Big),
\end{aligned}
\end{equation*}
where $\tilde u=(\tilde u(t))_{t\in[0,T]}$ denotes the modification of $u=(u(t))_{t\in[0,T]}$ that is continuous in $L_2(\dom;\R)$.
\end{cor}

\noindent
{\bf Acknowledgement.} I would like to thank Ren\'e L.~Schilling for helpful comments on the manuscript.

\end{appendix}

\providecommand{\bysame}{\leavevmode\hbox to3em{\hrulefill}\thinspace}

\section*{}
\noindent Felix~Lindner\\
TU Dresden \\
Institut f{\"u}r Mathematische Stochastik \\
01062 Dresden, Germany \\
Phone: 00\,49\,351\,463\,32\,437\\
E-mail: felix.lindner@tu-dresden.de \\


\begin{thebibliography}{99}
\addcontentsline{toc}{chapter}{Bibliography}


\bibitem{BorKon06}
M.~Borsuk, V.~Kondratiev: {\it Elliptic boundary value problems of second order on piecewise smooth domains}.
Elsevier, Amsterdam 2006.

%

\bibitem{CioDahKin11}
P.A.\ Cioica, S.\ Dahlke, S.\ Kinzel, F.\ Lindner, T.\ Raasch, K.\ Ritter, R.L.\ Schilling:
Spatial Besov regularity for stochastic partial differential equations on Lipschitz domains.
{\it Studia Math.} {\bf 207} (2011) 197--234.

\bibitem{CioKimLee12}
P.A.\ Cioica, K.-H.\ Kim, K.\ Lee, F.\ Lindner:
On the $L_q(L_p)$-regularity and Besov smoothness of stochastic parabolic equations on bounded Lipschitz domains (Preprint).
{\it DFG-SPP 1324 Preprint} {\bf 130} (2012)
\url{http://www.dfg-spp1324.de/download/preprints/preprint130.pdf}.


\bibitem{DaPraZab}
G.~Da~Prato, J.~Zabczyk:
{\it Stochastic equations in infinite dimensions}.
Cambridge University Press, Cambridge 1992.

\bibitem{Dau88}
M.~Dauge: {\it Elliptic boundary value problems on corner domains:\ smoothness and asymptotics of solutions}.
Springer, Lecture Notes in Mathematics {\bf 1341}, Berlin 1988.

\bibitem{DeFl93}
A.~Defant, K.~Floret:
{\it Tensor norms and operator ideals}.
North-Holland, Amsterdam 1993.


\bibitem{DeV98}
R.A.~DeVore:
Nonlinear approximation,
{\it Acta Numer.}~{\bf 7} (1998) 51--150.





\bibitem{Fla90}
F. Flandoli: Dirichlet boundary value problem for stochastic parabolic equations: compatibility relations and regularity of solutions. {\it Stochastics} {\bf 29} (1990) 331--357.

\bibitem{Foe93}
O. F\"{o}llinger: {\it Laplace- und Fourier-Transformation}.
H\"{u}thig, Heidelberg 1993.

\bibitem{Gri85}
P.~Grisvard:
{\it Elliptic problems in nonsmooth domains}.
Pitman, Boston 1985.

\bibitem{Gri87}
P.~Grisvard:
Edge behavior of the solution of an elliptic problem.
{\it Math.~Nachr.}~{\bf 132} (1987) 281--299.

\bibitem{Gri92}
P.~Grisvard:
{\it Singularities in boundary value problems},
Springer, Berlin 1992.

\bibitem{Gri95}
P.~Grisvard:
Singular behavior of elliptic problems in non Hilbertian Sobolev spaces.
{\it J.~Math.\ Pures Appl.}~{\bf 74} (1995) 3--33.

\bibitem{Gue91}
P.B. Guest:
{\it Laplace transforms and an introduction to distributions.}
Ellis Horwood,
London 1991.

\bibitem{JenRoeck12}
A. Jentzen, M. R{\"o}ckner:
Regularity analysis for stochastic partial differential equations with nonlinear multiplicative trace class noise.
{\it J. Differential Equations}~{\bf 252} (2012) 114--136.

\bibitem{JerKen81}
D.~Jerison, C.E.~Kenig:
The Dirichlet problem on non-smooth domains.
{\it Ann.~of Math.~(2)} {\bf 113} (1981) 367--382.

\bibitem{JerKen95}
D.~Jerison, C.E.~Kenig:
The inhomogeneous dirichlet problem in Lipschitz domains.
{\it J.~Funct.\ Anal.}~{\bf 130} (1995) 161--219.

\bibitem{KaRi83}
R.V. Kadison, J.R. Ringrose:
{\it Fundamentals of the theory of operator algebras. Volume~I: Elementary theory.}
Academic Press, San Diego 1983.

\bibitem{Kim04}
K.H.~Kim, 
On stochastic partial differential equations with  variable coefficients in ${C}^1$ domains.
{\it Stochastic Process.\ Appl.}~{\bf 112} (2004) 261--283.

\bibitem{Kim11}
K.H.~Kim, 
A weighted Sobolev space theory of parabolic stochastic PDEs on non-smooth domains.
{\it J.~Theoret.\ Probab.}, published online: 9 November 2012, DOI 201210.1007/s10959-012-0459-7.

\bibitem{KozMazRos97}
V.A.~Kozlov, V.G.~Maz'ya, J.~Ro{\ss}mann:
{\it Elliptic boundary values problems in domains with point singularities}.
American Mathematical Society, 
Providence, Rhode Island 1997. 

\bibitem{KozMazRos01}
V.A.~Kozlov, V.G.~Maz'ya, J.~Ro{\ss}mann: 
{\it Spectral problems associated
with corner singularities of solutions to elliptic equations}. 
American Mathematical Society, 
Providence, Rhode Island 
2001.

\bibitem{KruLar12}
R. Kruse, S. Larsson:
Optimal regularity for semilinear stochastic partial differential equations with multiplicative noise.
{\it Electron.\ J.~Probab.}~{\bf 17} (2012) 1--19.

\bibitem{Kry94}
N.V. Krylov: 
A ${W}^n_2$-theory of the Dirichlet problem for
SPDEs in general smooth domains. 
{\it Probab.\ Theory Relat.\ Fields} {\bf 98} (1994) 389--421.

\bibitem{Kry99}
N.V. Krylov: 
An analytic approach to SPDEs.
in: B.L. Rozovskii, R. Carmona (eds.):
{\it Stochastic partial differential equations. Six perspectives.}
American Mathematical Society, 
Providence, Rhode Island 
1999, 185--242.

\bibitem{KryLot99}
N.V.~Krylov, S.V.~Lototsky:
A Sobolev space theory of SPDE with constant coefficients on a half line.
{\it SIAM J.~Math.\ Anal.}~{\bf 30} (1999) 298--325.

\bibitem{KryLot99b}
N.V.~Krylov, S.V.~Lototsky:
A Sobolev space theory of SPDEs with constant
coefficients in a half space.
{\it SIAM J.~Math.\ Anal.}~{\bf 31} (1999) 19--33.

\bibitem{Kwe12}
J.R.\ Kweon:
Edge singular behavior for the heat equation on polyhedral cylinders in $\R^3$.
{\it Potential Anal.}~{\bf 38} (2013) 589--610.

\bibitem{Lin11}
F.~Lindner:
{\it Approximation and regularity of stochastic PDEs (doctoral thesis).}
Shaker, Aachen 2011.

\bibitem{MazRos10}
V.G.~Maz'ya, J.~Ro{\ss}mann: 
{\it Elliptic equations in polyhedral domains.}
American Mathematical Society, 
Providence, Rhode Island 
2010.

\bibitem{Met}
M. M{\'e}tivier:
{\it Semimartingales. A course on stochastic processes.}
de Gruyter, Berlin 1982.

\bibitem{MetPel80}
M. M{\'e}tivier, J. Pellaumail: 
{\it Stochastic integration.}
Academic Press, New York 1980.

\bibitem{PesZab}
S.~Peszat, J.~Zabczyk:
{\it Stochastic partial differential equations with L\'evy noise. An evolution equation approach.}
Cambridge University Press, Cambridge 2007.

\bibitem{PreRoeck}
C. Pr\'{e}v\^{o}t, M. R\"ockner:
{\it A concise course on stochastic partial differential equations.}
Springer, 
Lecture Notes in Mathematics {\bf 1905},
Berlin 2007.

\bibitem{Tre67}
F. Tr\`{e}ves:
{\it Topological vector spaces, distributions and kernels.}
Academic Press,
New York 1967.

\bibitem{Tri78}
H.~Triebel:
{\it Interpolation theory, function spaces, differential operators (2nd edn).}
Johann Ambrosius Barth, Heidelberg 1995 .

\bibitem{NeeVerWei08}
J.M.A.M.~van~Neerven, M.C.~Veraar, L.~Weis: 
Stochastic  evolution equations in UMD Banach spaces. 
{\it J.\ Funct.\ Anal.}~{\bf 255} (2008) 940--993.

\bibitem{NeeVerWei12}
J.M.A.M.~van~Neerven, M.C.~Veraar, L.~Weis: 
Maximal ${L}^p$-regularity for stochastic evolution equations. 
{\it SIAM J.\ Math.\ Anal.}~{\bf 44} (2012) 1372--1414.

\bibitem{Wei80}
J. Weidmann:
{\it Linear operators in Hilbert spaces.}
Springer, New York 1980.

\bibitem{Zem87}
A.H.~Zemanian:
{\it Distribution theory and transform analysis.}
Dover, New York 1987.

\end{thebibliography}
\end{document}